\documentclass[11pt,a4paper,leqno]{amsart}
\usepackage{graphicx} 
\usepackage{amssymb}
\usepackage{amsmath}
\usepackage{amsthm}
\usepackage{cancel}
\usepackage{amsfonts,amssymb}
\usepackage[abbrev,non-sorted-cites, alphabetic]{amsrefs}
\usepackage{latexsym}
\usepackage{color}
\usepackage{appendix}
\usepackage{hyperref}
\hypersetup{
    colorlinks=true,
    linkcolor=black,
    citecolor=black,
}

\allowdisplaybreaks

\oddsidemargin=.in
\evensidemargin=.in

\newtheorem{theorem}{Theorem}[section]
\newtheorem{lem}[theorem]{Lemma}
\theoremstyle{plain}

\newtheorem{claim}{Claim}

\newtheorem{corollary}[theorem]{Corollary}

\newtheorem{definition}[theorem]{Definition}

\newtheorem{lemma}[theorem]{Lemma}

\newtheorem{proposition}[theorem]{Proposition}
\newtheorem{remark}{Remark}

\newenvironment{claimproof}{\par \noindent \textit{Proof of Claim.}}{\hfill $\diamond$ \par}

\newcommand{\N}{\mathbb{N}}     
\newcommand{\R}{\mathbb{R}}     
\newcommand{\C}{\mathbb{C}}     

\newcommand{\ol}[1]{\overline{#1}}  

\newcommand\dd{{\mathrm d}}

\DeclareMathOperator\im{Im}

\DeclareMathOperator\osc{osc}

\title[Spectral Analysis for Singularities of LMCF]{Spectral Analysis for Finite-Time Singularities of Lagrangian Mean Curvature Flow}

\author{Maxwell Stolarski}
\address{Maxwell Stolarski: Warwick Mathematics Institute, Zeeman Building, University of Warwick, Coventry CV4 7AL, United Kingdom}
\email{max.stolarski@warwick.ac.uk}

\author{Wei-Bo Su}
\address{Wei-Bo Su: Department of Mathematics, National Central University, No. 300, Zhongda Rd., Zhongli District, Taoyuan City 320317, Taiwan}
\email{weibosu@math.ncu.edu.tw}

\date{\today}

\begin{document}

\begin{abstract}
    Let $\mathcal C$ be a $G$-invariant special Lagrangian cone admitting a scaled family of $G$-invariant special Lagrangian desingularizations $a\overline L$ which converge to $\mathcal C$ as $a\searrow 0$. We study the linearized self-shrinker operator on $a\overline L$ in a Gaussian weighted $L^2$ space of $G$-equivariant functions. For $0<a\ll1$, we construct any prescribed finite number of eigenfunctions whose eigenvalues converge to those of the limiting conical operator, and we prove a spectral gap estimate on the orthogonal complement of these modes. We also identify the lowest eigenfunction with the scaling mode of the special Lagrangian desingularization. This spectral basis provides the analytic foundation for the construction of Type II blow-up solutions of Lagrangian mean curvature flow in the companion paper~\cite{SS26II}.
\end{abstract}

\maketitle

\tableofcontents

\section{Introduction}

A central problem in geometric flows, and more generally in parabolic PDEs, is to understand the asymptotic behavior of solutions near a limiting solution. This question is particularly important when the limiting solution arises as a parabolic blow-up limit at a singularity of the flow. In many such cases, the limit is self-similar. The first-order approximation of the flow near the limiting solution is then governed by the linearized operator at the limit. When the limit is smooth, the linearized operator admits a suitable spectral decomposition in an $L^{2}$ space, and the nonlinear error terms can be controlled, the asymptotic behavior of the flow is determined by the spectral properties of this operator. 
This strategy has played a central role in the construction of self-similar blow-up solutions and has been used to obtain sharp estimates near limiting solutions; see, for instance, Merle--Zaag~\cite{MZ98}, Angenent--Daskalopoulos--Sesum~\cite{ADS19}, Choi--Haslhofer--Hershkovits~\cite{CHH22}, the first named author~\cite{Stolarski24b}, Lee-Zhao~\cite{LeeZhao24}, and Bamler-Lai~\cite{BamlerLai26}.

The preceding scheme becomes considerably more delicate when the self-similar solution arising as a parabolic blow-up limit is singular. In this situation, the singular behavior of the original flow develops at a scale much smaller than the self-similar scale, often called the Type II scale. Consequently, the analysis of the asymptotic behavior becomes a genuinely multi-scale problem: at the self-similar scale, the flow is described by the linearization around the singular self-similar solution, whereas at Type II scales, the flow is often modeled on a desingularization of that singular self-similar solution.

There are two distinct regimes. The first is the case where the desingularization can be viewed as an $L^{2}$ perturbation of the self-similar solution. In this regime, the linearized operator on the self-similar solution remains dominant, and the analysis is still more or less dictated by the spectral theory described above. See, for example, Vel{\'a}zquez's construction of Type II blow-up in mean curvature flow~\cite{Velazquez94}, Herrero-Vel{\'a}zquez in the semi-linear heat equation~\cite{HV}, and the first named author in the Ricci flow~\cite{Stolarski24}. The second regime arises when the desingularization is not an $L^{2}$ perturbation of the self-similar solution. In this case, the deviation at the Type II scale is sufficiently large that the linear theory around the self-similar solution no longer governs the leading-order behavior. This is the regime addressed in the present paper.

The difficulty in the non-$L^{2}$ regime can be addressed by linearizing around the desingularization itself, rather than around the self-similar solution. This shifts the problem to the construction of $L^{2}$ eigenfunctions on the desingularization, which serve as the analytic foundation for the asymptotic analysis. Since the desingularization is a small, though non-$L^{2}$, perturbation of the self-similar solution, one expects its spectral theory to be related to that of the self-similar solution through a perturbative mechanism. When the desingularization is asymptotic to the self-similar solution at infinity, the relevant eigenfunctions can be constructed by a gluing procedure, matching eigenfunctions associated with the self-similar solution to eigenfunctions associated with the desingularization. This type of perturbative spectral theory was studied and used by, for example, Had{\v z}i{\'c}--Rapha{\"e}l~\cite{HadzicRaphael19}, Collot--Merle--Rapha{\"e}l~\cite{CollotMerleRaphael20}, Collot--Ghoul--Masmoudi--Nguyen~\cites{CGMN_Spectral, CGMN_Flow}, and Bensouilah--Duong--Ghoul~\cite{BDG25}.

In this paper, we develop a perturbative spectral theory for the linearized Lagrangian self-shrinker operator on the special Lagrangian desingularization of a special Lagrangian cone in a cohomogeneity-one setting. This will provide the analytic foundation for the construction of Type II singularities in Lagrangian mean curvature flow in the companion paper~\cite{SS26II}.

A family of Lagrangian submanifolds $\{L(t)\}_{t\in[0,T)}$ in $\mathbb{C}^n$ evolves by Lagrangian mean curvature flow (LMCF) if
\begin{align}\label{eq: LMCF in heat equation form}
    \left(\frac{d}{dt}\mathbf{x}\right)^{\perp}
    = \mathbf{H}_{L(t)}(\mathbf{x}),
    \quad \forall \mathbf{x}\in L(t),
\end{align}
where $\mathbf{H}_{L(t)}(\mathbf{x}) = \Delta_{L(t)}\mathbf{x}$ is the mean curvature of $L(t)$, and $\Delta_{L(t)}$ denotes the Laplace--Beltrami operator of the induced metric on $L(t)$. Suppose that, as $t\to T$, the flow develops a finite-time singularity at the origin. In the self-similar variables
\begin{align*}
    {\mathbf{s}}=e^{\tau/2}\mathbf{x},
    \qquad
    \tau=-\log(T-t),
\end{align*}
the LMCF \eqref{eq: LMCF in heat equation form} becomes the rescaled flow
\begin{align}\label{eq: RLMCF in intro}
    \left(\frac{d}{d\tau}{\mathbf{s}}\right)^{\perp}
    =
    \mathbf{H}_{{M}(\tau)}({\mathbf{s}})
    +
    \frac{{\mathbf{s}}^{\perp}}{2},
    \quad
    \forall {\mathbf{s}}\in
    {M}(\tau):=e^{\tau/2}L(T-e^{-\tau}).
\end{align}
By the compactness result of Neves~\cite{Neves07}, every subsequential limit of ${M}_\tau$ as $\tau\to\infty$ converges weakly to a union of special Lagrangian cones. Thus, given a special Lagrangian cone $\mathcal{C}\subset\mathbb{C}^n$ admitting a special Lagrangian desingularization $a\overline{L}\to\mathcal{C}$ as $a\to0$, the formation of a singularity modeled on $\mathcal{C}$ is naturally modeled by perturbations of $a(\tau)\overline{L}$, where $a(\tau)\to0$ as $\tau\to\infty$.

Under a cohomogeneity-one symmetry assumption, described in Sections~2.2 and~2.3, the LMCF equation \eqref{eq: LMCF in heat equation form} reduces to the quasilinear parabolic equation
\begin{align}\label{eq: equivariant LMCF in intro}
    \partial_t h
    =
    \frac{h_{xx}}{1+h_x^2}
    +
    (n-1)\frac{xh_x-h}{x^2+h^2}
\end{align}
for a function $h:\mathbb{R}\times[0,T)\to\mathbb{R}$. There exists a unique even stationary solution
$f:\mathbb{R}\to\mathbb{R}_+$ of \eqref{eq: equivariant LMCF in intro} satisfying $f(0) = 1$ and $f'(0) = 0$. By scale-invariance, for every $a>0$, $f_{a}(x) := af(\frac{x}{a})$ is also a stationary solution. This family of stationary solutions $\{f_{a}\}_{a>0}$ satisfies
\begin{align*}
    f_a(x)\to x\cot\left(\frac{\pi}{2n}\right)
\end{align*}
smoothly on compact subsets of $(0,\infty)$ as $a\to0$. For each $a>0$, the graph of $f_{a}$ in $\mathbb{C}$ is the profile curve of an equivariant special Lagrangian $a\overline{L}$ in $\mathbb{C}^{n}$ asymptotic to a special Lagrangian cone $\mathcal{C}$ at infinity.

In the self-similar variable $s=e^{\tau/2}x$, and writing $h=\partial_s u$, the rescaled flow \eqref{eq: RLMCF in intro} becomes the fully nonlinear parabolic equation
\begin{align}\label{eq: RLMCF at potential level}
    \partial_\tau u
    =
    \arctan(u_{ss})
    +
    (n-1)\arctan(s^{-1}u_s)
    -
    \frac{s}{2}u_s
    +
    u.
\end{align}
Linearizing the right-hand side at the desingularization determined by $f_a$ gives the operator
\begin{align}
    H_a = L_{s,a} - \frac{s}{2}\partial_s + 1,
\end{align}
where
\begin{align*}
    L_{s,a} =\frac{1}{\sqrt{\bigl(1+(\partial_s f_a)^2\bigr)\bigl(s^2+f_a^2\bigr)^{n-1}}}
    \frac{\partial}{\partial s}\left(\sqrt{\frac{\bigl(s^2+f_a^2\bigr)^{n-1}}{1+(\partial_s f_a)^2}}
    \frac{\partial}{\partial s}\right)
\end{align*}
is the Laplace--Beltrami operator for the special Lagrangian $a \overline L$ acting on equivariant functions.
$H_{a}$ is self-adjoint with respect to the weighted $L^{2}$ inner product
\begin{align*}
    \langle u, v\rangle_{L^{2}_{a}} = \int_{\mathbb{R}}uv\:e^{-F_a(s)} dV_a,
\end{align*}
with
\begin{equation*}
            dV_a := \sqrt{ ( 1 + (\partial_s f_a)^2 )( s^2 + f_a^2 )^{n-1} } \, ds \text{ and } 
            F_a(s) := \frac{s^2}4 + \frac12 \int_0^s \tilde s ( \partial_s f_a)^2 d \tilde s.
\end{equation*}

The main result of this paper gives a precise description of an arbitrary
finite portion of the spectrum of $H_a$ in $L^2_a$ near the singular limit
$a\to0$. Since the special Lagrangian desingularization
$a\overline{L}$ degenerates to the cone $\mathcal{C}$ in this limit,
the spectral problem for $H_a$ is not a regular perturbation problem. Instead,
it has two natural regions. In the inner region, after rescaling by $a^{-1}$,
the operator is modeled on the special Lagrangian desingularization $\overline{L}$.
In the outer region, it is modeled on the linearized operator associated with
the limiting cone, whose eigenfunctions can be computed explicitly. We construct
eigenfunctions by solving the corresponding interior and exterior eigenvalue problems
and matching their values and first derivatives at an intermediate point.
This produces eigenfunctions of $H_a$ with eigenvalues
asymptotic to $1-k$ as $a\to0$, and it yields a spectral gap estimate on the
orthogonal complement of these modes.

Our main spectral theorem is the following.

\begin{theorem}[Eigenfunctions and Spectral Gap]\label{main thm intro}
    Let $n\geq 3$, $K\in\mathbb{N}$, and $\epsilon>0$. 
    There exist
    $s_0^*=s_0^*(n,K)>0$ and, for each $0<s_0\leq s_0^*$,
    $a^*=a^*(n,K,s_0,\epsilon)>0$  and $C = C(n,K,s_0) > 0$ with the following property.

    For every $0<a\leq a^*$ and every integer $0\leq k\leq K$, there exists 
    $\widetilde\lambda_k=\widetilde\lambda_{k}(a;n,k,s_0)  \in \R$ 
    and a smooth eigenfunction $\phi_{k,a} = \phi_{k, a , s_0, n}$ such that 
    \begin{equation*}
         H_a\phi_{k,a}=(1-k + \widetilde \lambda_k) \phi_{k,a} \qquad \text{with} \qquad 
        |\widetilde\lambda_k(a)|\le C a
        \quad \text{and} \quad 
        |\partial_a\widetilde\lambda_k(a)|\le C.
    \end{equation*}


    Moreover, the eigenfunctions $\{ \phi_{k,a} \}_{0 \le k \le K}$ satisfy the following spectral gap estimate.
    If $v:\mathbb{R}\to\mathbb{R}$ is an odd $C^2$ function such that
    \begin{gather*}
        \|v\|_{L^2_a} + \| \partial_s v \|_{L^2_a} +\|H_av\|_{L^2_a}<\infty 
        \qquad \text{and} \qquad  
        \langle v,\phi_{k, a}\rangle_{L^2_a}=0
        \quad
        \text{for all } 0\leq k\leq K,
    \end{gather*}
    then
    \begin{equation*}\label{eq spectral gap intro}
        \langle v,H_av\rangle_{L^2_a}
        \leq
        (-K+\epsilon)\|v\|_{L^2_a}^2 .
    \end{equation*}
\end{theorem}

We also identify the geometric meaning of the lowest mode. The eigenfunction
$\phi_{0,a}$ corresponds to the infinitesimal change of scale
of the special Lagrangian desingularization. More precisely, if $\beta_a$ denotes the
potential associated with the scaling deformation, then
\begin{align*}
    \frac{d}{da}f_a
    =
    2a^{-1}\partial_s\beta_a .
\end{align*}
The following result shows that $\phi_{0,a}$ is close in $L^2_a$ to the
normalized scaling mode $a^{-2}\beta_a$. This identification will be used in
the companion paper to modulate the scale parameter $a$ while keeping the
remaining modes under control.

\begin{proposition}[Identification of the Scaling Mode]
\label{prop scaling mode intro}
    Let $n\geq 3$ and $K\in\mathbb{N}$. There exist
    $s_0^*=s_0^*(n,K)>0$ and, for each $0<s_0\leq s_0^*$,
    $a^*=a^*(n,K,s_0)>0$ with the following property.

    Suppose $0<a\leq a^*$, and let
    $\phi_{k,a}=\phi_{k,a,s_0,n}$, $0\leq k\leq K$, be the eigenfunctions
    constructed in Theorem~\ref{main thm intro}. Let $\beta_a$ be the potential for the scaling deformation as in
    Lemma~\ref{lem beta properties}. Then
    \begin{equation*}
        \|\phi_{0,a}-a^{-2}\beta_a\|_{L^2_a}
        \leq
        Ca,
    \end{equation*}
    for some constant $C=C(n,s_0)>0$. In particular,
    there exists $C' = C'(n,K, s_0) > 0$ such that 
    \begin{equation*}
        \left| \langle \phi_{0,a},a^{-2}\beta_a\rangle_{L^2_a} - \|\phi_{0,a}\|_{L^2_a}^2 \right| + \sum_{k=1}^K \left|  \langle \phi_{k,a},a^{-2}\beta_a\rangle_{L^2_a}\right| \le C' a.
    \end{equation*}
\end{proposition}

We begin in Section \ref{section Prelims} by recalling some background material on Lagrangian mean curvature flow and establishing notation.
Section \ref{section Construction of Eigenfunctions} comprises the bulk of the paper and obtains the proofs of Theorem \ref{main thm intro} and Proposition \ref{prop scaling mode intro}.
Specifically, Theorem \ref{main thm intro} appears as Theorems \ref{thm global eigenfunctions} and \ref{thm spectral gap}, and Proposition \ref{prop scaling mode intro} appears as Lemma \ref{lem L^2_a est for phi with beta}.
In Section \ref{Section Approx Flow Solutions}, we indicate how the spectral analysis of Theorem \ref{main thm intro} can be used to obtain Lagrangian mean curvature flow solutions that form finite-time singularities with precise blow-up rates on the second fundamental form and mean curvature that remains uniformly bounded.
The arguments in Section \ref{Section Approx Flow Solutions} do not constitute rigorous proofs, but instead will be made rigorous in a forthcoming companion paper~\cite{SS26II}.

Finally, we note almost all results in this paper hold for all dimensions $n \ge 3$.
This dimension restriction is essentially due to the different asymptotic behavior of the special Lagrangian desingularization $\overline L$ in dimensions $n\ge 3$ compared to dimension $n =2$ (see Remark \ref{rmk beta asymps when n=2}).

\addtocontents{toc}{\protect\setcounter{tocdepth}{0}} 
\subsection*{Acknowledgments}
The authors thank Albert Wood, Charles Collot, and Felix Schulze for helpful conversations.
The first named author is supported by a Leverhulme Trust Early Career Fellowship (ECF-2023-182).
The second named author is supported by the Taiwan NSTC grant 114-2115-M-008-012-MY3. 
The authors also thank the University of Warwick and the National Center for Theoretical Sciences for the support and hospitality during the authors' visits.
\addtocontents{toc}{\protect\setcounter{tocdepth}{2}}

\section{Preliminaries} \label{section Prelims}
\subsection{Lagrangian Mean Curvature Flow in $\mathbb{C}^{n}$}
Throughout, we take the ambient manifold to be $\mathbb{C}^{n}$ equipped with its
standard flat Calabi--Yau structure $(\overline{g},J,\omega,\Omega)$. In terms of
the complex coordinates $z_{j}=x_{j}+iy_{j}$, $j=1,\dots,n$, these are given by
\begin{align*}
\overline{g}
=
\sum_{j=1}^{n}(dx_{j}^{2}+dy_{j}^{2}),\qquad
J=i\,\cdot,\qquad
\omega=
\sum_{j=1}^{n}dx_{j}\wedge dy_{j},\qquad
\Omega=dz_{1}\wedge\cdots\wedge dz_{n}.
\end{align*}
We denote by $\lambda$ the Liouville form
\begin{align*}
    \lambda
    =
    \frac{1}{2}\sum_{j=1}^{n}\big(y_{j}dx_{j}-x_{j}dy_{j}\big),
\end{align*}
such that $\omega = -d\lambda$.

An $n$-dimensional immersed submanifold $L\subset\mathbb{C}^{n}$ is called
\emph{Lagrangian} if $\omega\big|_{L}=0$. We say that a one-parameter family of
immersed submanifolds $\{L(t)\}_{t\in[0, T)}$ satisfies the \emph{Lagrangian mean curvature flow} (LMCF) if each time-slice
$L(t)$ is Lagrangian and
\begin{align}
    \left(\frac{d}{dt}\mathbf{x}_{L(t)}\right)^{\perp}
    =
    \mathbf{H}_{L(t)},
    \qquad t\in(0,T),
\end{align}
where $\mathbf{x}_{L(t)}\in\mathbb{C}^{n}$ is the position vector and $\mathbf{H}_{L(t)}$ denotes the mean curvature vector of $L(t)$.

If $L$ is oriented and Lagrangian, then
\begin{align*}
    \Omega|_{L}=e^{i\theta_{L}}\,dV_{L}
\end{align*}
for some function $\theta_{L}:L\to\mathbb{R}/2\pi\mathbb{Z}$, called the
\emph{Lagrangian angle}. The Lagrangian angle is a potential for the mean
curvature vector, in the sense that
\begin{align*}
    \mathbf{H}_{L}=J\nabla\theta_{L}.
\end{align*}
It follows that a connected Lagrangian $L$ is minimal if and only if $\theta_{L}$ is constant $\overline{\theta}$. In this
case, we call $L$ a \emph{special Lagrangian (SL)} submanifold with phase $\overline{\theta}$.

We say a Lagrangian submanifold $L$ is \emph{zero-Maslov} if $\theta_{L}$ can be lifted to a real-valued function $\theta_{L}:L\to\mathbb{R}$. A zero-Maslov Lagrangian submanifold $L$ is \emph{almost-calibrated} if $\osc_{L}\theta_{L}<\pi$, where $\osc_{L}\theta_{L} := \sup_L \theta_L - \inf_L \theta_L$ denotes the oscillation. 
A Lagrangian submanifold $L$ is \emph{exact} if $\lambda\big|_{L} = d\beta_{L}$ for some $\beta_{L}:L\to\mathbb{R}$. In this case, the position vector $\mathbf{x}_{L}\in\mathbb{C}^{n}$ of $L$ satisfies $\mathbf{x}_{L}^{\perp} = 2J\nabla\beta_{L}$, where $(\cdot)^{\perp}$ denotes the orthogonal projection onto the normal space of $L$. In other words, if $L$ is exact, the normal part of the position vector of $L$ has a potential $2\beta_{L}$. 
Exactness, zero-Maslov, and almost-calibrated conditions are preserved under LMCF.

Suppose that a LMCF $L(t)$ develops a finite-time singularity at
$(P,T)\in\mathbb{C}^{n}\times\mathbb{R}$. That is,
\begin{align*}
    \limsup_{\substack{x\in L(t),\, x\to P\\ t\nearrow T}}
    |\mathbf{A}_{L(t)}(x)|=\infty,
\end{align*}
where $\mathbf{A}_{L(t)}$ denotes the second fundamental form of $L(t)$ in
$\mathbb{C}^{n}$. A coarse asymptotic structure of the singularity is captured
by its \emph{tangent flows}, which arise as subsequential limits of the
parabolically rescaled flows
\begin{align*}
    L^{i}(t)
    :=
    \lambda_{i}\big(L(T+t/\lambda_{i}^{2})-P\big),
    \qquad t<0,
\end{align*}
where $\lambda_{i}\to\infty$. Neves~\cite{Neves07} showed that if the LMCF has \emph{zero-Maslov class}, then
every tangent flow at a finite-time singularity is a union of special
Lagrangian cones. Equivalently, consider the \emph{rescaled flow}
\begin{align*}
    M(\tau)
    =
    e^{\tau/2}\big(L(T-e^{-\tau})-P\big),
\end{align*}
which satisfies
\begin{align}
    \left(\frac{d}{d\tau}\mathbf{x}_{M(\tau)}\right)^{\perp} = \mathbf{H}_{M(\tau)} + \frac{1}{2}\mathbf{x}_{M(\tau)}^{\perp},
\end{align}
then every subsequential limit of $M(\tau)$ as $\tau\to\infty$ is a union of SL cones. If, in addition,
$L(t)$ is \emph{exact}, then Neves~\cite{Neves07} showed that
every tangent flow is a union of SL cones of \emph{same phase}.

The tangent flow captures the geometry of the flow near the singularity at the
parabolic scale $\sqrt{T-t}$. If the tangent flow is itself singular, then one
must look at smaller, \emph{Type II scales} in order to detect the smooth geometry
forming near the singularity. We call any smooth limiting flow obtained at such
a scale a \emph{Type II model} for the singularity.

Assuming that the tangent flow is a SL cone, one
expects the Type II model to be modeled on an asymptotically conical SL submanifold. More precisely, one expects the existence of a SL submanifold $\overline{L}$ asymptotic to the tangent cone
$\mathcal{C}$ at infinity, together with a scale function
\begin{align*}
    a(t)=o(\sqrt{T-t})\qquad\text{as }t\nearrow T,
\end{align*}
such that, in a neighborhood of the singularity, the flow is asymptotic to the
rescaled model $a(t)\overline{L}$, which collapses to $\mathcal{C}$ as
$t\nearrow T$.

\subsection{Cohomogeneity-One SLs and LMCF} \label{subsect Cohom-1 SLs and LMCF}

Now suppose that a compact connected Lie group $G\leq SU(n)$ acts linearly on
$\mathbb{C}^{n}$, preserving the standard flat Calabi--Yau structure, and that
the generic $G$-orbits are $(n-1)$-dimensional. Such cohomogeneity-one constructions were studied by, for instance, Haskins~\cite{Haskins04},
Joyce~\cite{Joyce02}, and Madnick--Wood~\cite{MadnickWood25}. We thus direct the readers to those papers for further details. Under the cohomogeneity-one setting, if $\mathcal{C}\subset\mathbb{C}^{n}$ is a $G$-invariant SL cone
with phase $0$, then, for any $a>0$, the submanifold
\begin{align}
    L
    =
    \left\{
    \lambda\boldsymbol{\sigma}
    \::\:
    \lambda\in\mathbb{C},\ 
    \im(\lambda^{n})=a^{n},\
    \arg(\lambda)\in(0,\pi/n),\
    \boldsymbol{\sigma}\in\Sigma=\mathcal{C}\cap\mathbb{S}^{2n-1}
    \right\}
\end{align}
is a smooth embedded SL submanifold in $\mathbb{C}^{n}$ with
phase $0$, asymptotic to $\mathcal{C}\cup e^{i\pi/n}\mathcal{C}$.

The set
\begin{align*}
    \left\{
    \lambda\in\mathbb{C}
    \::\:
    \im(\lambda^{n})=a^{n},\
    \arg(\lambda)\in(0,\pi/n)
    \right\}
\end{align*}
is a connected smooth embedded curve in $\mathbb{C}$, called the profile curve
of $L$. It can be parametrized by
\begin{align*}
    \gamma_{a}(\theta)
    =
    a\,\gamma(\theta)
    =
    a\,\sin(n\theta)^{-1/n}e^{i\theta},
    \qquad \theta\in(0,\pi/n),
\end{align*}
and satisfies
\begin{equation}\label{eq: equiv. SL profile curve}
    \boldsymbol{\kappa}(\gamma_{a})
    -
    (n-1)\frac{\gamma_{a}^{\perp}}{|\gamma_{a}|^{2}}
    =
    0.
\end{equation}
Rotating the profile curve by $e^{i\phi}$ changes the phase of the corresponding
SL by $n\phi$. In particular, the rotated curve
\begin{align} \label{eq: overline gamma}
    \overline{\gamma}(\theta)
    =
    \sin(n\theta)^{-1/n}
    e^{i\left(\theta+\frac{\pi}{2}\left(1-\frac1n\right)\right)},
    \qquad \theta\in(0,\pi/n),
\end{align}
is symmetric with respect to the imaginary axis and gives rise to an
asymptotically conical SL $\overline{L}$ with phase
$(n-1)\frac \pi 2$.

\begin{lemma}\label{lem: profile function of Lawlor neck}
    For any $n \ge 2$ and $a>0$, the curve $\overline{\gamma}_{a}=a\overline{\gamma}$ given in \eqref{eq: overline gamma} can be
    written as the graph of an even function
    \begin{align*}
        f_{a}:\mathbb{R}\to\mathbb{R}_{>0}.
    \end{align*}
    Moreover, $f_{a}(\cdot) = a f_1(\cdot / a) $ satisfies
    \begin{equation}\label{eq: profile function minimality}
        \frac{(f_{a})_{xx}}{1+(f_{a})_{x}^{2}}
        +
        (n-1)\frac{x(f_{a})_{x}-f_{a}}{x^{2}+(f_{a})^{2}}
        =
        0,
    \end{equation}
    \begin{align*}
        \text{with} \qquad  
        f_{a}(0)=a,\quad
        f_{a}'(0)=0, \quad \text{and } 
        f_{a}''(x)>0
        \text{ for all }x\in\mathbb{R}.
    \end{align*}
    Furthermore, as $|x|\to\infty$, there holds\footnote{See Definition \ref{defn notation} below for the precise definition of the $O(\cdot; \cdot)$ notation.}
    \begin{equation}\label{eq: asymptotic of SL profile function}
        \left|
        \left(\frac{d}{dx}\right)^{k}
        \!\left(f_{a}(x)-\overline{c}_{0}|x|\right)
        \right|
        =
        O\!\left(a^{n}|x|^{1-n-k}; k\right)
        \qquad\text{for all }k\in\mathbb{N}\cup\{0\},
    \end{equation}
    \begin{align*}
        \text{where } \quad 
        \overline{c}_{0}
        =
        \tan\!\left(\frac{\pi}{2}-\frac{\pi}{2n}\right)
        =
        \cot\!\left(\frac{\pi}{2n}\right).
    \end{align*}
\end{lemma}

\begin{proof}
    This follows from a direct calculation using the parametrization of
    $\overline{\gamma}_{a}$.
\end{proof}

Neves~\cite{Neves07}, Groh--Schwarz--Smoczyk--Zehmisch ~\cite{GSSZ07}, and
Madnick--Wood~\cite{MadnickWood25} showed that, under the above
cohomogeneity-one ansatz, LMCF reduces to the following evolution equation for
the profile curves $\gamma(t)$ in $\mathbb{C}$:
\begin{equation}\label{eq: equiv. LMCF profile curve}
    \left(\frac{d}{dt}\gamma(t)\right)^{\perp}
    =
    \boldsymbol{\kappa}(\gamma(t))
    -
    (n-1)\frac{\gamma(t)^{\perp}}{|\gamma(t)|^{2}},
\end{equation}
where $\boldsymbol{\kappa}(\gamma(t))$ denotes the curvature vector of
$\gamma_{t}$.

\begin{remark}
    It is worth emphasizing that the reduced equations
    \eqref{eq: equiv. LMCF profile curve} and
    \eqref{eq: equiv. SL profile curve} are independent of the choice of the Lie
    group $G$. For example, when $G=SO(n)$ acts diagonally on $\mathbb{C}^{n}$,
    the orbits are copies of $\mathbb{S}^{n-1}$, and the resulting
    $SO(n)$-invariant SLs are diffeomorphic to
    $\mathbb{S}^{n-1}\times\mathbb{R}$. These are the
    \emph{Lagrangian catenoids}, which may be viewed as the $SO(n)$-invariant
    versions of the \emph{Lawlor necks}~\cite{Law89}.
\end{remark}

\subsection{Rescaled LMCF Near a SL Profile}\label{subsec: RLMCF near a SL profile}

Let $(\gamma(t))_{t\in I}$ be a solution of (\ref{eq: equiv. LMCF profile curve}), such that the associated equivariant LMCF $L(t)$ is almost-calibrated. Denote the image of $\gamma(t):\mathbb{R}\to\mathbb{C}$ by $\Gamma(t) = \gamma(t)(\mathbb{R})$.

Suppose that for each $t\in I$, $\Gamma(t)$ is an entire graph with profile function $f$, that is, $\Gamma(t) = \{x + if(x, t)\:|\:x\in\mathbb{R}\}$, then $f$ satisfies the quasi-linear equation
\begin{equation}\label{eq: LMCF of profile function}
    \partial_{t}f = \frac{f_{xx}}{1+f_{x}^{2}} + (n-1)\frac{xf_{x} - f}{x^{2}+f^{2}}.
\end{equation}

We will assume that $I = [0, T)$ for some $T>0$, and as $t\nearrow T$, $L(t)$ develops a finite-time singularity. By \cite{MadnickWood25}*{Theorem 6.2}, the singularity must occur at $(O, T)$, where $O\in\mathbb{C}^{n}$ is the origin. By \cite{MadnickWood25}*{Theorem 6.7}, the tangent flow is a union of $G$-invariant SL cones. We will assume the tangent flow at $(O, T)$ is $\mathcal{C}_{0}\cup\mathcal{C}_{1}$ given by the reflection-symmetric profile curve $C_{0} := e^{-i(\frac{\pi}{2}+\frac{\pi}{2n})}\mathbb{R}\cup e^{i(\frac{\pi}{2}+\frac{\pi}{2n})}\mathbb{R}$. It follows that the rescaled flow $\widetilde{\gamma}_{\tau} := e^{\tau/2}\gamma_{T - e^{-\tau}}$ converges to $C_{0}$ as $\tau\to \infty$ smoothly away from $O$. By the graphicality assumption on $\gamma(t)$, the rescaled flow $\widetilde{\gamma}_{\tau}$ is also graphical whose profile function satisfies the equation
\begin{align}\label{eq: rescaled flow of profile function}
    \partial_{\tau}h = \frac{h_{ss}}{1+h_{s}^{2}} + (n-1)\frac{sh_{s} - h}{s^{2}+h^{2}} - \frac{1}{2}(sh_{s} - h).
\end{align}



Equation~(\ref{eq: rescaled flow of profile function}) is a quasi-linear parabolic equation.
We now linearize at a SL profile and integrate (\ref{eq: rescaled flow of profile function}) into a fully non-linear equation.

Define differential operators $\mathcal{H}$ and $\mathcal{X}$ by
\begin{equation}
    \mathcal{H}[h] := \frac{h_{ss}}{1+h_{s}^{2}} + (n-1)\frac{sh_{s} - h}{s^{2}+h^{2}},\quad \mathcal{X}[h] := h-sh_{s}.
\end{equation}
Note that $\mathcal{H}[h] = 0$ if and only if $h$ is a SL profile function, and $\mathcal{X}[h] = 0$ if and only if $h$ is linear. 
Then the rescaled flow (\ref{eq: rescaled flow of profile function}) reads
\begin{equation}
    \partial_{\tau}h = \mathcal{H}[h] + \frac{1}{2}\mathcal{X}[h].
\end{equation}
\begin{lemma}[Potentials for Mean Curvature and Position]\label{lem: potentials of H and X}
    We have
    \begin{equation}
        \mathcal{H}[h] = \partial_{s}\Theta[h]\quad\mbox{and}\quad\mathcal{X}[h] = 2\partial_{s}\beta[h],
    \end{equation}
    where
    \begin{equation} \label{eqn defn Theta[h], beta[h]}
        \Theta[h] = \arctan(h_{s}) + (n-1)\arctan(s^{-1}h) \quad\mbox{and}\quad\beta[h] = \frac{1}{2}\int_{0}^{s}(h-\sigma h_{\sigma})\:\dd\sigma.
    \end{equation}
\end{lemma}
\begin{proof}
    This follows from direct computation.
\end{proof}

Assume now that $h$ is a solution of the rescaled Lagrangian mean curvature flow (\ref{eq: rescaled flow of profile function}) and $h$ is a perturbation of the SL profile $f_{a}$ by a function $U$, that is,
\begin{equation}\label{ansatz: perturb Lawlor neck by U}
    h(s, \tau) = f_{a(\tau)}(s) + U(s, \tau).
\end{equation}

\begin{lemma}[linearization of mean curvature at $f$]\label{lem: linearized mean curvature}
    Let $U:\mathbb{R}\to\mathbb{R}$ be a smooth function, and let $f_{\xi U} := f + \xi U$, $\xi\in(-\epsilon, \epsilon)$, where $f = f_{1}$ is the profile function of the SL $\overline{L}$ as in Lemma \ref{lem: profile function of Lawlor neck}. Then
    \begin{align*}
        \left. \frac{\dd}{\dd \xi}\right|_{\xi = 0}\mathcal{H}(f_{\xi U})= L_{s, 1}U + \frac{2(n-1)(sf_{s} - f)}{s^{2}+f^{2}}\left\{\frac{f_{s}U_{s}}{1+f_{s}^{2}} - \frac{fU}{s^{2}+f^{2}}\right\} - \frac{(n-1)U}{s^{2}+f^{2}},
    \end{align*}
    where $L_{s, 1}$ is the induced Laplacian on $\overline{L}$ given by
    \begin{equation}
        L_{s, 1}U = \frac{U_{ss}}{1+f_{s}^{2}} + (n-1)\frac{sU_{s}}{s^{2}+ f^{2}}.
    \end{equation}
\end{lemma}
\begin{proof}
    Direct computation shows
    \begin{align*}
         & \left. \frac{\dd}{\dd \xi} \right|_{\xi = 0}\left\{  \frac{(f_{\xi U})_{ss}}{1+(f_{\xi U})_{s}^{2}} + (n-1)\frac{s(f_{\xi U})_{s} - (f_{\xi U})}{s^{2} + (f_{\xi U})^{2}} \right\}\\
         & = \frac{U_{ss}}{1+f_{s}^{2}} - \frac{2f_{s}f_{ss}U_{s}}{(1+f_{s}^{2})^{2}} + (n-1)\left\{\frac{sU_{s} - U}{s^{2}+f^{2}} - \frac{2(sf_{s} - f)fU}{(s^{2}+f^{2})^{2}}\right\}\\
        &= L_{s, 1}U - \frac{2f_{s}f_{ss}U_{s}}{(1+f_{s}^{2})^{2}} - \frac{(n-1)U}{s^{2}+f^{2}}\left\{1 + \frac{2f(sf_{s} - f)}{s^{2}+f^{2}}\right\}\\
        & = L_{s, 1}U + \frac{2(n-1)f_{s}(sf_{s} - f)U_{s}}{(1+f_{s}^{2})(s^{2}+f^{2})} - \frac{(n-1)U}{s^{2}+f^{2}}\left\{1 + \frac{2f(sf_{s} - f)}{s^{2}+f^{2}}\right\}\\
        & = L_{s, 1}U + \frac{2(n-1)(sf_{s} - f)}{s^{2}+f^{2}}\left\{\frac{f_{s}U_{s}}{1+f_{s}^{2}} - \frac{f U}{s^{2}+f^{2}}\right\} - \frac{(n-1)U}{s^{2}+f^{2}},
    \end{align*}
    where in the third equality we use the equation of $f = f_{1}$.
\end{proof}

We now assume $U = u_{s}$ for some smooth potential $u:\mathbb{R}\to\mathbb{R}$ so that
\begin{equation}\label{ansatz: perturb Lawlor neck by der of u}
    h(s, \tau) = f_{a(\tau)}(s) + u_{s}(s, \tau).
\end{equation}

\begin{lemma}[Commutation Relations]\label{lem: commutator}
    Let $f = f_{1}$ be the SL profile function as in Lemma \ref{lem: profile function of Lawlor neck}. For any smooth function $u:\mathbb{R}\to\mathbb{R}$, we have
    \begin{align}
        (L_{s, 1}u)_{s} &= L_{s, 1}u_{s} + \frac{2(n-1)(sf_{s} - f)}{s^{2}+f^{2}}\left\{\frac{f_{s}u_{ss}}{1+f_{s}^{2}} - \frac{fu_{s}}{s^{2}+f^{2}}\right\} - \frac{(n-1)u_{s}}{s^{2}+f^{2}}, \text{ and}\\
        (\mathcal{X} u)_{s} &= \mathcal{X} u_{s} - u_{s}.
    \end{align}
\end{lemma}
\begin{proof}
    For the first equation, a direct computation shows
    \begin{align*}
        (L_{s, 1}u)_{s} &= \left(\frac{u_{ss}}{1+f_{s}^{2}} + (n-1)\frac{su_{s}}{s^{2}+f^{2}}\right)_{s}\\
        &=\frac{u_{sss}}{1+f_{s}^{2}} + (n-1)\frac{su_{ss}}{s^{2}+f^{2}} - \frac{2f_{s}f_{ss}}{(1+f_{s}^{2})^{2}}u_{ss}\\
        &\quad + \frac{(n-1)u_{s}}{s^{2}+f^{2}} - \frac{2(n-1)(s + ff_{s})}{(s^{2}+f^{2})^{2}}su_{s}\\
        & = L_{s, 1}u_{s} + \frac{2(n-1)(sf_{s} - f)}{s^{2}+f^{2}}\left\{\frac{f_{s}u_{ss}}{1+f_{s}^{2}} - \frac{fu_{s}}{s^{2}+f^{2}}\right\}\\
        &\quad + \frac{2(n-1)(sf_{s} - f)fu_{s}}{(s^{2}+f^{2})}- \frac{2(n-1)(s + ff_{s})}{(s^{2}+f^{2})^{2}}su_{s}+ \frac{(n-1)u_{s}}{s^{2}+f^{2}}\\
        & = L_{s, 1}u_{s} + \frac{2(n-1)(sf_{s} - f)}{s^{2}+f^{2}}\left\{\frac{f_{s}u_{ss}}{1+f_{s}^{2}} - \frac{fu_{s}}{s^{2}+f^{2}}\right\} - \frac{(n-1)u_{s}}{s^{2}+f^{2}},
    \end{align*}
    as desired. Note that in the third equality we use the equation of $f$ \eqref{eq: profile function minimality}.
    
    For the second equation, 
    \begin{equation*}
        (\mathcal{X} u)_{s} = (u - su_{s})_{s} = - su_{ss}
        = u_{s} - su_{ss} - u_{s}
        = \mathcal{X} u_{s} - u_{s}.
    \end{equation*}
\end{proof}

\begin{lemma}[Linearized Mean Curvature at the Level of Potential]\label{lem: linearization at the level of potential}
    Let $f_{\xi u_{s}}:= f + \xi u_{s}$, where $f = f_{1}$ is the SL profile function from Lemma \ref{lem: profile function of Lawlor neck}. Then
    \begin{align*}
        \left. \frac{d}{d \xi} \right|_{\xi = 0}\mathcal{H}[f_{\xi u_{s}}] = (L_{s, 1}u)_{s}.
    \end{align*}
\end{lemma}
\begin{proof}
    The result follows from Lemma~\ref{lem: linearized mean curvature} and Lemma~\ref{lem: commutator}.
\end{proof}

We now arrive at the equation of the potential $u$ of the perturbation given in (\ref{ansatz: perturb Lawlor neck by der of u}).
\begin{proposition} \label{prop: PDE for the potential}
    Let $h(s, \tau) = f_{a(\tau)}(s) + u_{s}(s, \tau)$. Then $h$ is a solution of (\ref{eq: rescaled flow of profile function}) if and only if
    \begin{equation}\label{eq: linearization at Lawlor at potential level with constant c}
        \partial_{\tau}u = (n-1)\frac{\pi}{2} + \left(1-2\tfrac{\partial_\tau a(\tau)}{a(\tau)}\right)\beta[f_{a(\tau)}] + H_{s, a}u + Q_{s, a}(u) + c
    \end{equation}
    for some $\tau$-dependent constant $c = c(\tau) \in\mathbb{R}$, where
    \begin{equation}
        H_{s, a}u = L_{s, a}u - \frac{s}{2}u_{s} + u,
        \qquad L_{s,a} u = \frac{u_{ss}}{1 + (f_a)_s^2} +(n-1) \frac{s u_s}{s^2 + f_a^2} 
        ,
    \end{equation}
    and
    \begin{multline}
        Q_{s, a}(u) = \Theta[f_{a}+u_{s}] - (n-1)\frac{\pi}{2} - L_{s, a}\\
        = \arctan\left((f_{a} + u_{s})_{s}\right) - \frac{u_{ss}}{1+(f_{a})_{s}^{2}} + (n-1)\left[\arctan\left(\frac{f_{a}+u_{s}}{s}\right) - \frac{su_{s}}{s^{2}+f_{a}^{2}}\right]-(n-1)\frac{\pi}{2}.
    \end{multline}
\end{proposition}
\begin{proof}
    By Lemma~\ref{lem: linearized mean curvature}, Lemma~\ref{lem: linearization at the level of potential} and the Taylor expansion we have
    \begin{equation}
        \mathcal{H}[f_{a}+u_{s}] = \mathcal{H}[f_{a}] + (L_{s, a}u)_{s} + \mathcal{Q}_{s, a}(u) = (L_{s, a}u)_{s} + \mathcal{Q}_{s, a}[u],
    \end{equation}
    where
    \begin{equation*}
        \mathcal{Q}_{s, a}[u] = \mathcal{H}[f_{a}+u_{s}]  - (L_{s, a}u)_{s}
        =\left(\Theta[f_{a}+u_{s}] - L_{s, a}u\right)_{s} = (\Theta[f_{a}] + Q_{s, a}(u))_{s},
    \end{equation*}
    and $Q_{s,a}(u) := \Theta[f_{a}+ u_{s}] - \Theta[f_{a}] - L_{s,a}u$.
    On the other hand, by Lemma~\ref{lem: commutator} we have
    \begin{equation*}
        \mathcal{X}[f_{a} + u_{s}] = \mathcal{X}[f_{a}] + \mathcal{X}[u_{s}]
        = (2\beta[f_{a}] + \mathcal{X}u + u)_{s}
        = (2\beta[f_{a}] + 2u - su_{s})_{s}.
    \end{equation*}
    Thus, $\partial_{\tau}h = \mathcal{H}[h] + \frac{1}{2}\mathcal{X}[h]$ implies that
    \begin{align*}
         \partial_{\tau}f_{a(\tau)} + (\partial_{\tau}u)_{s} &= \left(\Theta[f_{a}] + L_{s, a}u + Q_{s, a}(u) + \beta[f_{a}] + u - \frac{s}{2}u_{s} \right)_{s}.
    \end{align*}
    Since $f_{a(\tau)}(s) = a(\tau)f_{1}(\tfrac{s}{a(\tau)})$, taking a time-derivative yields
    \begin{align*}
        \partial_{\tau}f_{a(\tau)}(s) &= \partial_{\tau}\left(a(\tau)f_{1}(\tfrac{s}{a(\tau)})\right)\\
        &= a'(\tau)f_{1}(\tfrac{s}{a(\tau)}) - s(f_1)_{s}(\tfrac{s}{a(\tau)})a'(\tau)a(\tau)^{-1}\\
        &= \frac{a'(\tau)}{a(\tau)}\left(f_{a(\tau)}(s) - s(f_1)_{s}(\tfrac{s}{a(\tau)})\right)\\
        &= \frac{a'(\tau)}{a(\tau)}\mathcal{X}[f_{a(\tau)}](s)\\
        &= \left(2\frac{a'(\tau)}{a(\tau)}\beta[f_{a(\tau)}]\right)_{s}.
    \end{align*}
    Therefore,
    \begin{align*}
        (\partial_{\tau}u)_{s} = \left(\Theta[f_{a}] + \left(1 - 2\tfrac{a'(\tau)}{a(\tau)}\right)\beta[f_{a(\tau)}] + L_{s, a}u - \frac{s}{2}u_{s} + u + Q_{s, a}(u)\right)_{s}.
    \end{align*}
    Since $\Theta[f_{a}] = (n-1)\frac{\pi}{2}$ is a constant, this finishes the proof.
\end{proof}

It follows that to construct a solution to (\ref{eq: rescaled flow of profile function}) with $h = f_{a} + u_{s}$, it suffices to solve for $u$ which satisfies (\ref{eq: linearization at Lawlor at potential level with constant c}) for some $c\in\mathbb{R}$.

\begin{corollary} \label{cor: PDE for the potential}
    Suppose $u$ solves the equation
    \begin{equation}\label{eq: main equation}
        \partial_{\tau}u = \left(1 - 2\frac{\partial_\tau a(\tau)}{a(\tau)}\right)\beta[f_{a(\tau)}] + H_{s, a}u + Q_{s, a}(u).
    \end{equation}
    Then $h(s, \tau) := f_{a(\tau)}(s) + u_{s}(s, \tau)$ is a solution of (\ref{eq: rescaled flow of profile function}).
\end{corollary}
\begin{proof}
    Take $c = -(n-1)\frac{\pi}{2}$ in (\ref{eq: linearization at Lawlor at potential level with constant c}).
\end{proof}

We collect the properties of the zeroth order term in (\ref{eq: main equation}).
\begin{lemma} \label{lem beta properties}
    Let $n \ge 3$ and $a>0$. The function $\beta_{a} := \beta[f_{a}]$ has the following properties:
    \begin{itemize}
        \item[(i)] $\beta_{a}(s) = a^{2}\beta_{1}(\frac{s}{a})$.
        \item[(ii)] $2a^{-1}\beta_{a}$ is the potential of the infinitesimal scaling of $f_{a}$.
        \item[(iii)] $L_{s, a}\beta_{a} = 0$.
        \item[(iv)] $\beta_a(s)$ is a smooth, odd function of $s \in \R$. 
        \item[(v)] There is a positive constant  $\overline{\beta}\in (0, \infty)$  such that 
        \begin{equation}
            \lim_{s\to\pm\infty}\beta_{1}(s) = \pm\overline{\beta}, 
        \end{equation}
        and $\beta_{a}$ satisfies the estimate\footnote{See Definition \ref{defn notation} below for the precise definition of the $O(\cdot; \cdot)$ notation.}
        \begin{align}
            \left|\partial_{s}^{k}(\beta_{a}(s) \pm a^{2}\overline{\beta})\right| = O(a^{n}|s|^{2-n-k}; k)\quad\mbox{as}\quad s\to\pm\infty.
        \end{align}
        \item[(vi)] For $s > 0$, $\beta_a(s)$ is a positive, increasing, concave function of $s$.
    \end{itemize}
\end{lemma}
\begin{proof}
    The scaling property (i) follows from definition.
    
    Since $f_{a}(s) = af_{1}(\tfrac{s}{a})$ for any $a>0$, we have
    \begin{align*}
        \left. \frac{d}{d\xi}\right|_{\xi = 0}f_{a+\xi}(s) &= \left.\frac{\dd}{\dd\xi}\right|_{\xi = 0}(a+\xi)f_{1}(\tfrac{s}{a+\xi})\\
        & = f_{1}(\tfrac{s}{a}) - a^{-1}s(f_{1})_{s}(\tfrac{s}{a})\\
        & = a^{-1}\left(f_{a}(s) - s(f_{a})_{s}(s)\right)\\
        & = \left(2a^{-1}\beta_{a}(s)\right)_{s}.
    \end{align*}
    This proves (ii).

    To prove (iii), first we compute (for simplicity, $(\cdot)' := \frac{d}{d s}$ here)
    \begin{equation*}
        \beta'_{a} = f_{a} - sf'_{a}, \qquad 
        \beta''_{a} = -sf''_{a}.
    \end{equation*}
    Thus,
    \begin{equation*}
        L_{s, a}\beta_{a} = \frac{\beta''_{a}}{1+(f'_{a})^{2}} + (n-1)\frac{s\beta'_{a}}{s^{2}+f_{a}^{2}}
        =s\left[\frac{-f''_{a}}{1+(f'_{a})^{2}} + (n-1)\frac{f_{a} - sf'_{a}}{s^{2}+f_{a}^{2}}\right] = 0
    \end{equation*}
    by (\ref{eq: profile function minimality}). Hence, (iii) is proved.

    (iv) follows from definition \eqref{eqn defn Theta[h], beta[h]} of $\beta_a = \beta[f_a]$ and the fact that $f_a(s)$ is a smooth, even function of $s \in \R$ by Lemma \ref{lem: profile function of Lawlor neck}.

    The estimate (v) with some $\overline \beta \in \R$ follows from the estimates of the profile function in Lemma~\ref{lem: profile function of Lawlor neck}.

    Since $f_a$ is convex (Lemma \ref{lem: profile function of Lawlor neck}), $\beta_a'' = -s f_a'' < 0$ for $s > 0$.
    (v) (with some $\overline \beta \in \R$ possibly non-positive) shows $\lim_{s \to \infty} \beta_a' = 0$, and it follows that $\beta_a' > 0$ for $s >0$.
    (iv) implies $\beta_a(0) = 0$.
    Thus, $\beta_a(s) > 0$ for all $s > 0$ and $\lim_{s \to \infty} \beta_1 = \overline \beta >0$.
    This shows $\overline \beta >0$ in (v) and proves (vi).
\end{proof}

\begin{remark} \label{rmk beta asymps when n=2}
    In dimension $n =2$, $\beta[f_a]$ has different asymptotics.
    Indeed, for $n =2$, $f_1(s) - \overline c_0 s$ is asymptotic to $s^{-1}$ (Lemma \ref{lem: profile function of Lawlor neck}), so the integral \eqref{eqn defn Theta[h], beta[h]} defining $\beta[f_1]$ is asymptotic to $\ln s$.
\end{remark}

\subsection{Notation}
    For the remainder of the paper, we use the following notation unless otherwise noted.

    \begin{definition} \label{defn notation}
        For constants $a,b,c$ and functions $f = f(\vec x), g = g( \vec x),$ and $h = h(\vec x) \ge 0$ defined for all $\vec x \in \Omega$, 
        we write ``$f = g + O(h; a,b,c )$'' to mean there exists a constant $C$ (depending on $a,b,c$ and dimension $n$) such that 
            $$|f(\vec x) - g (\vec x )| \le C h(\vec x) \qquad \text{for all }  \, \vec x \in \Omega.$$
        We write ``$f = g + O(h; a,b,c)$ as $\vec x \to \infty$'' to mean there exists constants $C, R_0$ (depending on $a,b,c$ and dimension $n$) such that 
            $$|f(\vec x) - g (\vec x )| \le C h(\vec x) \qquad \text{for all } \, |\vec x| \ge R_0.$$
        Similarly, ``$f = g + O(h; a,b,c)$ as $\vec x \to \vec x_0$'' means there exists $C, \epsilon$ (depending on $a,b,c$ and dimension $n$) such that 
            $$|f(\vec x) - g(\vec x ) | \le C h(\vec x) \qquad \text{for all } \, |\vec x - \vec x_0| \le \epsilon.$$
        Analogous definitions apply for additional (or fewer) parameters $a,b,c,d,e, \dots$ or for $\vec x = (\vec y, \vec z) \to (\infty, \vec z_0)$.

        For nonnegative functions $f = f(\vec x) \ge 0 $ and $g = g(\vec x) \ge 0$, we use the notation ``$f \lesssim g$'' to mean there exists $C > 0$ such that $f(\vec x) \le C g (\vec x)$ for all $\vec x$.
        We write ``$f \lesssim_{a,b,c} g $'' to indicate that $f \lesssim g$ and the implicit constant $C = C(a,b,c)$ depends on $a,b,c$.
        We write ``$f \sim g$'' to indicate that $f \lesssim g$ and $g \lesssim f$.
        Analogous definitions apply for additional (or fewer) parameters.
    \end{definition}

    For any $a > 0$, $\overline \gamma_a : \R \to \C$ denotes the profile curve for the cohomogeneity-one special Lagrangian $\overline L_a$ asymptotic to $C = \R_{\ge 0} e^{i \left( \frac \pi 2 - \frac \pi {2n} \right)} \cup \R_{\ge 0} e^{i \left( \frac \pi 2 + \frac \pi {2n} \right)}$ described in Subsection \ref{subsect Cohom-1 SLs and LMCF}.
    It is scaled so that the image 
        $$\overline \Gamma_a := \overline \gamma_a (\R) \text{ contains } 0+ia \in \C.$$
    $f_a : \R \to \R$ denotes the profile function for $\overline \Gamma_a$ as in Lemma \ref{lem: profile function of Lawlor neck}, that is,
        $$\overline \Gamma_a = \{ s + i f_a(s) \, \colon \, s \in \R \}.$$
    $\overline c_0 := \tan\left( \frac \pi 2 - \frac \pi{2n} \right)$ denotes the asymptotic slope of $f_a(s)$ (Lemma \ref{lem: profile function of Lawlor neck}).

    We denote operators
    \begin{align}
        \label{eqn L_a defn}
        L_a u :={}& L_{s,a} u = \frac1{\sqrt{ ( 1 + (\partial_s f_a)^2 ) (s^2 + f_a^2)^{n-1} }} \frac \partial {\partial s} \left( \sqrt{ \frac{(s^2 + f_a^2)^{n-1} }{1 + (\partial_s f_a)^2 } } \frac{\partial u}{\partial s} \right)\\
        \nonumber
        ={}& \frac{u_{ss}}{1 + (f_a)_s^2} +(n-1) \frac{s u_s}{s^2 + f_a^2} \\
        \label{eqn H_a defn}
        H_a u :={}& H_{s,a} u = L_{s,a} u - \frac s2 \partial_s u + u , \\ 
        \label{eqn Q_a defn}
        Q_a(u) :={}& Q_{s,a}(u) = \arctan ( \partial_s f_a + u_{ss} ) + (n-1) \arctan \left( \frac {f_a+u_s}s \right) - L_{s, a} u .
    \end{align}

\subsubsection{Weighted Sobolev Spaces}

\begin{definition}  \label{defn Weigthed Sobolev spaces}
        For any $a > 0$, define measures $dV_a, d\mu_a$ on $\R$ by
        \begin{align}
            \label{eqn defn dV_a}
            dV_a &:= \sqrt{ ( 1 + (\partial_s f_a)^2 )( s^2 + f_a^2 )^{n-1} } \, ds \text{ and} \\
            \label{defn eqn mu_a}
            d \mu_a &:= e^{-F_a(s)} dV_a \\
            \label{eqn defn F_a}
            \text{where } F_a(s) &:= \frac{s^2}4 + \frac12 \int_0^s \tilde s ( \partial_s f_a)^2 d \tilde s.
        \end{align}

        For any $k \in \N$ and $p \in [1, \infty]$, we furthermore denote weighted Sobolev spaces with respect to $d \mu_a$ by
        \begin{equation}
            W^{k,p}_a := W^{k,p}( \R , d\mu_a ), \qquad
            H^k_a := W^{k, 2}(\R, d\mu_a ), \qquad
            L^2_a := L^2(\R, d\mu_a).
        \end{equation}
        where for example
        \begin{equation*}
            \| u \|_{H^k_a} = \left( \int_\R u^2 d \mu_a + \int_\R ( \partial_s u)^2 d \mu_a  \right)^{1/2}.
        \end{equation*}
\end{definition}
\begin{remark}
    Note that the derivatives used in the definition of the $W^{k,p}_a, H^k_a$ norms are simply $\partial_s$ and \emph{not} covariant derivatives associated to the special Lagrangian.
\end{remark}

\section{Construction of Eigenfunctions} \label{section Construction of Eigenfunctions}

In this section, which comprises the bulk of the paper, we obtain the proofs of the main results Theorem \ref{main thm intro} and Proposition \ref{prop scaling mode intro}.
Specifically, Theorem \ref{main thm intro} will appear below as Theorems \ref{thm global eigenfunctions} and \ref{thm spectral gap}, and Proposition \ref{prop scaling mode intro} will appear as Lemma \ref{lem L^2_a est for phi with beta}.

The first main result, Theorem \ref{thm global eigenfunctions}, constructs an arbitrary finite number of eigenfunctions $\{ \phi_{k,a} : \R \to \R \}_{0 \le k \le K}$ for the linear operators $H_a$ \eqref{eqn H_a defn} when $0 < a \ll 1$ is sufficiently small.
To obtain this result, we first construct interior eigenfunctions $\phi_{int}$ for the operator on the cohomogeneity-one special Lagrangian $\overline L_a$ (Subsection \ref{subsection Cohom-One SLs}) and exterior eigenfunctions $\phi_{ext}$ for the operator on the asymptotic SL cone $\mathcal C$ (Subsection \ref{Subsection Planes}).
Subsection \ref{Subsection Matching} then constructs the eigenfunctions $\phi_{k,a}$ by gluing the $\phi_{int}$ and $\phi_{ext}$.

The gluing construction yields precise pointwise and integral estimates for the eigenfunctions $\phi_{k,a}$.
Subsection \ref{Subsection Eigenfunction Estimates} records these estimates.
These estimates later yield the spectral gap result Theorem \ref{thm spectral gap} obtained in Subsection \ref{Subsection Spectral Gap}.

\subsection{Cohomogeneity-One Special Lagrangians} \label{subsection Cohom-One SLs}

This subsection's main result, Lemma \ref{lem interior eigenfunctions}, constructs the interior eigenfunctions $\phi_{int}$ for the operator $H_a = L_a - \frac s2 \partial_s + 1 $ on a compact domain of the SL $\overline L_a$.
The interior eigenfunctions $\phi_{int}$ comprise a key piece of the gluing construction for the global $H_a$-eigenfunctions $\phi_{k,a}$ described above.
Lemma \ref{lem interior eigenfunctions} constructs these interior eigenfunctions by perturbing a weighted sum of generalized kernel elements for the operator $L_a$.

\subsubsection{Generalized Kernel Elements}

Consider the Poisson equation
    $$L_{a=1} u = \psi$$
where $u =u(s)$ and $\psi = \psi(s)$.
By \eqref{eqn L_a defn}, this reduces to
    $$\frac{1}{K_1(s)} \frac{d}{ds} \left( \frac{1}{K_2(s)} \frac{ d u }{ds} \right) :=
    \frac1{\sqrt{ ( 1 + (\partial_s f_1)^2 ) (s^2 + f_1^2)^{n-1} }} \frac \partial {\partial s} \left( \sqrt{ \frac{(s^2 + f_1^2)^{n-1} }{1 + (\partial_s f_1)^2 } } \frac{\partial u}{\partial s} \right)
    = \psi(s).$$
It is then straightforward to verify that, for any smooth function $\psi : \R \to \R$,
\begin{equation} \label{eqn Lapl Inversion Formula}
    u(s) =
    \int_0^s \psi (  \sigma)  K_1 (  \sigma) 
    \int_{ \sigma}^s  K_2( \tilde \sigma) d \tilde \sigma  d  \sigma
    := \int_0^s \psi (  \sigma) K(  \sigma , s ) d  \sigma 
\end{equation}
\begin{gather} \label{eqn defn K(sigma, s)}
\begin{aligned}
    \text{with } K( \sigma , s) :={}& K_1 (  \sigma) 
    \int_{ \sigma}^s  K_2( \tilde \sigma) d \tilde \sigma \\
    ={}& \sqrt{(1 + \partial_s f_1(\sigma)^2 )(\sigma^2 + f_1(\sigma)^2 )^{n-1} } 
    \int_{\sigma}^s \sqrt{ \frac{1 + \partial_s f_1(\tilde \sigma)^2 } {(\tilde \sigma^2 + f_1(\tilde \sigma)^2 )^{n-1}}} d \tilde \sigma
\end{aligned}
\end{gather}
defines a smooth solution to $L_1 u = \psi$ on $\R$ such that additionally $u(0) = 0$ and $\left. \frac{du}{ds}\right|_{s=0} = 0$.

The following proposition uses the above inversion of $L_{a=1}$ to construct generalized kernel elements of the operator $L_{a=1}$.
\begin{proposition}[The Generalized Kernel of $L_{a=1}$] \label{prop generalized kernel}
    Let $n \ge  3$ and $a=1$.
    There exists a collection of smooth functions $\{ u_i : \R \to \R , \, u_i =u_i(s) \}_{i \in \N}$ such that
		$$u_0 = \beta_{a=1}, \qquad L_{a=1} u_0 = 0, \quad  \text{and} \quad 
		L_{a=1} u_{i+1} = u_{i} \quad \forall i \in \mathbb N.$$
    Moreover, the $u_i$ are odd functions of $s \in \R$, have $u_i(s) > 0$ for all $s > 0$, and have asymptotic expansions
	\begin{equation} \label{generalized kernel asymptotic expansion eqn}
		u_i(s) =  \pm C_i s^{2i} + O( s^{2i-1} ; n,i) 	\qquad (\text{as } s \to \pm \infty)
	\end{equation}
	where the $C_i = C_i(n,i) \ne 0$ are given by the inductive formula
    \begin{equation} \label{gend kernel asymp constant inductive formula}
		C_i = \frac{C_{i-1} ( 1 + \ol{c}_0^2 ) }{2i ( 2i +n - 2 ) } 		\qquad ( \forall i \ge 1)
    \end{equation}
    and $\ol{c}_0 = \ol{c}_0(n) = \cot \left( \frac \pi {2n} \right) $ is as in Lemma \ref{lem: profile function of Lawlor neck}.
    These asymptotic expansions extend to first- and second-order derivatives of $u_i$, namely, for all $i \ge 1$
    \begin{equation} \label{generalized kernel derivatives asymptotic expansion eqn}
        \partial_s u_i(s) = \pm 2i C_i s^{2i-1} + O ( s^{2i-2};n,i ), \,
        \partial_{ss} u_i(s) = \pm 2i(2i-1) C_i s^{2i-2} + O ( s^{2i-3};n,i ),
    \end{equation}
    and, for $i =0$,
    \begin{equation} \label{beta derivatives asymptotic expansion eqn}
        \partial_s u_0(s) = \partial_s \beta_1 =  O ( s^{1-n} ;n) \le  O ( s^{-2} ;n), \quad
        \partial_{ss} u_0(s) = \partial_{ss} \beta_1 = O ( s^{-n} ;n) \le O ( s^{-3};n ) 
    \end{equation}
    as $s \to \pm \infty$.
    The $u_i$ have the additional property that
    \begin{equation} \label{eqn defn Theta_i}
        - \frac s2 \partial_s u_i + u_i = (1 - i ) u_i + \Theta_i ,
		\quad \text{where} \quad  \Theta_i =  O( s^{2i - 1};n,i ) 
		\quad \text{as } s \to \pm \infty
    \end{equation}
  and $\Theta_i$ is a smooth, odd function of $s$.
\end{proposition}
\begin{proof}
    To simplify the notation, we omit the dependencies in the ``$O$'' and ``$\lesssim$'' notation and write, for example, simply $O (\, \cdot \, ;  i) = O (\, \cdot \, )$ throughout the proof.
    We also omit the subscripts ``$a=1$'' throughout and write simply $L = L_{a=1}$ for example.

    First, note that by Lemma \ref{lem beta properties} $u_0(s) := \beta(s)$ is a smooth, odd function of $s \in \R$ that satisfies $L u_0 = 0$.
    Following \eqref{eqn Lapl Inversion Formula}, define the $u_i = u_i (s)$ for $i\ge1$ inductively by
    \begin{equation} \label{proof generalized kernel, eqn 1}
        u_i(s) := \int_0^s u_{i-1}(\sigma) K(\sigma, s) d \sigma.
    \end{equation}
    By the discussion above and a straightforward induction argument, the functions $u_i : \R \to \R$ are well-defined, smooth, and satisfy $L u_i = u_{i-1}$ for all $i \ge 1$.

    (The $u_i$ are odd functions) \\
    We claim the $u_i$ are odd functions of $s$ for all $i \in \mathbb N$.
    This has already been shown for $u_0 = \beta$.
    Note that $f(s)$ is even (Lemma \ref{lem: profile function of Lawlor neck}), and so
        $$K(-\sigma, -s) = -K(\sigma , s)$$
    by \eqref{eqn defn K(sigma, s)}.
    Thus, a $\sigma$-substitution reveals
    $$u_{i+1}(-s) = \int_0^{-s} u_i(\sigma) K(\sigma, -s) d \sigma = - \int_0^s u_i (-\sigma) K (-\sigma, -s) d \sigma = \int_0^s u_i (-\sigma) K(\sigma, s) d \sigma.$$
    An induction argument then implies that the $u_i$ are odd functions of $s \in \R$ for all $i \in \mathbb N$.

    ($u_i(s) > 0$ for all $s > 0$) \\
    Note that $u_0 = \beta(s) >0 $ for all $s > 0$ by Lemma \ref{lem beta properties}.
    \eqref{eqn defn K(sigma, s)} implies $K(\sigma, s) >0$ for all $s, \sigma \in \R$ with $s > \sigma$.
    Using \eqref{proof generalized kernel, eqn 1}, a straightforward induction argument then implies that $u_i(s) > 0$ for all $s> 0$ and $i \in \N$.

    (Asymptotics at $s = \pm \infty$) \\
    Since the $u_i$ are odd functions of $s$, it suffices to determine the asymptotics at $s=+\infty$.
	We first obtain estimates for $K_1, K_2$.
	Observe that the asymptotics of $f=f_1$ and $\partial_s f_1 = f_s$ (Lemma \ref{lem: profile function of Lawlor neck}) imply
    \begin{multline}
        \label{K_2 asymps eqn 1}
        K_2(s) = \sqrt{ \frac{1 + f_s^2}{ ( s^2 + f^2 )^{n-1} } } 
	   = s^{1-n} \sqrt{ \frac{ 1 + \overline c_0^2 + O( s^{-n} ) }{ [ 1 + \overline c_0^2 + O ( s^{-n} ) ]^{n-1} } } \\
	   = s^{1-n} ( 1 + \overline c_0^2 )^{\frac{2-n}{2} } \left( 1 + O (s^{-n} ) \right)
	   \lesssim s^{1-n} .
    \end{multline}
    for large $s \gg 1$.
	Additionally, $0 <K_2(s) < C(n) < \infty$ for $s$ near $0$.
	It follows that
    \begin{equation} \label{K_2 asymps eqn 2}
        \frac{1}{ ( 1 + s)^{n-1} } \lesssim K_2(s) \lesssim \frac{ 1}{ ( 1 + s)^{n-1} }.
    \end{equation}
	An analogous argument for $K_1$ reveals
    \begin{gather}
        \label{K_1 asymps eqn 1}
        K_1(s) = s^{n-1} ( 1 + \overline c_0^2 )^{ \frac{n}{2} } \left( 1 + O (s^{-n} ) \right) 
		\quad  \text{as } s \to \infty, \\
        \label{K_1 asymps eqn 2}
		\text{and } ( 1 + s)^{n-1} \lesssim K_1 (s) \lesssim ( 1 + s)^{n-1}.
    \end{gather}

    We now use induction on $i$ to prove the claimed asymptotics of $u_i$.
    For $n \ge 3$, the asymptotics for $u_0 = \beta$ follow from Lemma \ref{lem beta properties}.
    Assume now the inductive hypothesis that the claimed asymptotics \eqref{generalized kernel asymptotic expansion eqn}--\eqref{generalized kernel derivatives asymptotic expansion eqn} hold for $u_{i-1}$ ($i \ge 1$).
    Since $u_i(s) = \int_0^s u_{i-1}(\sigma) K(\sigma, s) d \sigma$, differentiating and using \eqref{eqn defn K(sigma, s)} implies
    $$\frac{du_i}{ds}(s) = K_2(s) \int_0^s K_1(\sigma) u_{i-1}(\sigma) d \sigma.$$
    By \eqref{K_2 asymps eqn 1}, \eqref{K_1 asymps eqn 1}, and the inductive hypothesis,
    \begin{align*}
        \frac{du_i}{ds} 
        ={}& K_2(s) \int_0^s K_1(\sigma) u_{i-1}(\sigma) d \sigma \\
        ={}& s^{1-n} ( 1 + \overline c_0^2 ) ( 1 + O( s^{-n} ) ) \int_0^s \sigma^{n-1} ( 1 + O ( \sigma^{-n} ) ) [ C_{i-1} \sigma^{2i-2} + O ( \sigma^{2i-3} ) ] d \sigma \\
        ={}& s^{1-n} ( 1 + c_0^2 ) ( 1 + O( s^{-n} ) ) 
        \int_0^s C_{i-1} \sigma^{2i-2+n-1} + O ( \sigma^{2i-3 + n-1} ) d \sigma  \\
        ={}& \frac{C_{i-1}(1 + \overline c_0^2)}{2i - 2 +n} s^{2i-1} + O ( s^{2i-2} ) 
    \end{align*}
    as $s \to +\infty$.
    The claimed asymptotics for $u_i$, as well as \eqref{gend kernel asymp constant inductive formula}, now follow from integration.
    The asymptotics for $\partial_{ss} u_i$ can then be deduced from the relation $L u_i = u_{i-1}$, the asymptotics of $u_{i-1}$, $u_i$, and $\partial_s u_i$, and the fact that the asymptotics of $K_2(s)$ \eqref{K_2 asymps eqn 1} propagate to its derivative $\partial_s K_2$.

    (Asymptotics of $- \frac s2 \partial_s  u_i + u_i$)	\\
    The asymptotic expansion of $-\frac s2 \partial_s u_i + u_i$ \eqref{eqn defn Theta_i} can be deduced from the asymptotics of $u_i$ and $\partial_s u_i$, \eqref{generalized kernel asymptotic expansion eqn} and \eqref{generalized kernel derivatives asymptotic expansion eqn} respectively.    
    Since $u_i, s \partial_s u_i$ are odd, smooth functions of $s$, it follows that $\Theta_i$ is odd and smooth.
\end{proof}

\subsubsection{Interior Eigenfunctions} \label{subsubsection interior eigenfunctions}

In this subsection, we construct eigenfunctions for $H_{s,a} = L_{s,a} - \frac s2 \partial_s +1$ in a small region near $s = 0$.
Later, these will be matched with exterior eigenfunctions to form globally defined eigenfunctions for small enough $a > 0$.
For these later applications, it will be convenient to use the variable $\sigma$ instead of $s$ in this section.

\begin{definition}
    For $\kappa \in \R$ and $\Gamma > 0$, define weighted $C^2$-norms
    $$\| u \|_{X^\kappa_{\Gamma}} := \sup_{-\Gamma < \sigma < \Gamma } \sum_{i=0}^2 \frac{ | ( \langle \sigma \rangle \partial_\sigma )^i u | }{ \langle \sigma \rangle^\kappa }, 
    \qquad \text{where } \langle \sigma \rangle = \sqrt{ 1 + \sigma^2}.$$
    These define Banach spaces
        $$X^\kappa_\Gamma := \left\{ u \in C^2 ( (-\Gamma, \Gamma)) \, \colon \, \| u \|_{X^\kappa_\Gamma} < \infty \right\} \qquad \text{with norm } \| \cdot \|_{X^\kappa_\Gamma}.$$
\end{definition}

\begin{remark}
    The spaces $X^\kappa_\Gamma$ are equivalent to the usual spaces $C^2((-\Gamma, \Gamma))$.
    However, it will be more convenient to work with the spaces $X^\kappa_\Gamma$ to capture the asymptotic growth or decay at infinity.
\end{remark}

\begin{lem}[Interior Eigenfunctions] \label{lem interior eigenfunctions}
	Let $n \ge 3$ and $K \in \N$. There exists $0 < \overline s_0(n,K) \ll 1$ sufficiently small depending only on $n,K$ such that the following holds:
    
    For all $\tilde \lambda$ with $|\tilde \lambda | \le 1$ and all 
    $0 < a \le s_0 \le \overline s_0(n, K  ) \ll 1$,
    there exists a smooth, odd function $\phi_{int} \in C^\infty ((-s_0/a, s_0/a) \to \R )$, $\phi_{int} = \phi_{int; \tilde \lambda, a, s_0,K,n}(\sigma)$ such that 
    \begin{equation} \label{eqn interior eigval problem}
        \left(L_{\sigma, a=1} - a^2 \frac \sigma 2 \partial_\sigma + a^2 \right) \phi_{int} = a^2 ( 1 - K + \tilde \lambda ) \phi_{int}
    \end{equation}
    and $\phi_{int}$ has the form
    \begin{equation} \label{eqn interior eigenfunctions ansatz}
		\phi_{int} = \sum_{i=0}^K \hat C_{K, i} a^{2i} u_i 
		+ \tilde \lambda \sum_{i=0}^K a^{2(i+1)} \left( \hat C_{K, i} u_{i+1} + v_i \right) + a^2 w_K
	\end{equation}
    where $u_i = u_i(\sigma)$ are the generalized kernel elements from Proposition \ref{prop generalized kernel}, 
    the constants $\hat C_{K, i}$ satisfy the inductive formula
    \begin{equation} \label{eqn interior eigenfunctions, constants relation}
        \hat C_{K, i+1} = - \hat C_{K, i} ( K - i) , \qquad \hat C_{K, 0} = 1,
    \end{equation}
    and the correction terms $v_i , w_K$ are odd functions of $\sigma$ that satisfy the estimates
    \begin{gather} \label{eqn interior eigenfunctions, correction term ests}
    \begin{aligned}
        \| v_i \|_{X_{s_0/a}^{2i+2}} &\le C s_0^2, &
        \| \partial_a v_i \|_{X_{s_0/a}^{2i+3}} &\le C s_0, &
        \| \partial_{\tilde \lambda} v_i \|_{X^{2i+2}_{s_0/a}} &\le C s_0^2, \\
        \| w_K \|_{X_{s_0/a}^1} &\le C , &
        \| \partial_a w_K \|_{ X_{s_0/a}^{2} } &\le C s_0, &
        \| \partial_{\tilde \lambda} w_K \|_{X_{s_0/a}^{3}} &\le C a^2, 
    \end{aligned}
    \end{gather}
    for some $C = C (n, K) > 0$.
    Moreover, $\phi_{int}$ is a smooth function of $(\tilde \lambda, a, s_0, \sigma)$.
\end{lem}
\begin{proof}
    Write $\Lambda = - \frac \sigma 2 \partial_\sigma + 1$ and $L = L_{\sigma, a = 1}$ throughout.
    Using \eqref{eqn interior eigenfunctions ansatz} and Proposition \ref{prop generalized kernel}, a computation reveals 
    \begin{align*}
		&( L + a^2 \Lambda  ) \phi_{int}	
	   - a^2 ( 1 -K + \tilde \lambda ) \phi_{int} \\
		={}&  \sum_{i=0}^{K-1} a^{2i+2} u_i \left[ \hat C_{K,i+1} + ( K - i ) \hat C_{K, i}   \right] 
		+  \sum_{i=0}^K \hat C_{K, i} a^{2i+2} \Theta_i  \\
        &+ \tilde \lambda \sum_{i=0}^K \hat C_{K, i} a^{2i+2} a^2 ( \Lambda - 1 + K - \tilde \lambda ) u_{i+1} 
        + \tilde \lambda \sum_{i=0}^K a^{2i+2} ( L v_i + a^2 ( \Lambda - 1 + K - \tilde \lambda ) v_i ) \\
        &+ a^2 L w_K + a^4 ( \Lambda - 1 + K - \tilde \lambda ) w_K.
    \end{align*}
    We therefore obtain a solution of \eqref{eqn interior eigval problem} when
    \begin{align}
        \label{proof interior eigenfunctions, eqn 1}
		\hat C_{K, i+1} + ( K-i) \hat C_{K,i} &= 0 && \forall i \in \{ 0, \dots, K - 1\}, \\
        \label{proof interior eigenfunctions, eqn 2}
		L v_i + a^2 ( \Lambda-1 + K  - \tilde \lambda ) \left( \hat C_{K, i} u_{i+1} + v_i \right) &= 0
		&& \forall i \in \{ 0, \dots, K \}, \text{ and} \\
        \label{proof interior eigenfunctions, eqn 3}
		L w_K + a^2 ( \Lambda -1+ K  - \tilde \lambda ) w_K + \left( \sum_{i=0}^K \hat C_{K, i} a^{2i} \Theta_i \right) &= 0.
	\end{align}    
	The relation \eqref{eqn interior eigenfunctions, constants relation} ensures \eqref{proof interior eigenfunctions, eqn 1} is satisfied.
    The existence of $v_i, w_K$ solving \eqref{proof interior eigenfunctions, eqn 2}, \eqref{proof interior eigenfunctions, eqn 3} will be obtained by a fixed point argument, for suitable choices of the parameters $\tilde \lambda , a,s_0$. 

    First, we prove the following general claim that will be useful for the fixed point arguments.
    \begin{claim} \label{proof interior eigenfunctions, claim 1}
        Let $\Gamma > 0$ and $\kappa > -2$.
        For any $\psi \in C^0 ( (-\Gamma, \Gamma) )$, there exists a unique $u \in C^2((-\Gamma, \Gamma))$ such that $Lu = \psi$ on $(-\Gamma, \Gamma)$ and $u(0) = \partial_\sigma u(0) = 0$.
        Moreover, $u$ is given by
        \begin{equation} \tag{\ref{eqn Lapl Inversion Formula}}
            u(\sigma) = \int_0^\sigma \psi(\xi) K(\xi, \sigma) d \xi
        \end{equation}
        and satisfies the estimate
        \begin{equation} \label{proof interior eigenfunctions, eqn 4}
            \| u \|_{X_\Gamma^{\kappa + 2}} \le C \sup_{|\sigma | < \Gamma} \frac{ |\psi(\sigma)|}{\langle \sigma \rangle^\kappa}
            \qquad \text{for some } C = C(n, \kappa ) > 0.
        \end{equation}
    \end{claim}
    \begin{claimproof}
        The existence and uniqueness of $u$ follows from standard linear ODE theory.
        It remains to prove the estimate \eqref{proof interior eigenfunctions, eqn 4}.
        Throughout the proof of this claim, $C = C(n, \kappa) > 1$ denotes a constant depending only on $n$ and $\kappa$, but not $\Gamma$, which may change from line to line.
        By uniqueness, $u$ is equivalently given by 
        \begin{equation} \tag{\ref{eqn Lapl Inversion Formula}}
            u(\sigma) = \int_0^\sigma \psi(\xi) K( \xi , \sigma) d \xi
            = \int_0^\sigma \psi(\xi) K_1(\xi) \int_\xi^\sigma K_2(\rho) d \rho d \xi.
        \end{equation}
        and so
        \begin{equation} \label{proof interior eigenfunctions, eqn 5}
            \partial_\sigma u (\sigma) = K_2(\sigma) \int_0^\sigma \psi(\xi) K_1(\xi) d \xi.
        \end{equation}
        Recall \eqref{K_2 asymps eqn 2} and \eqref{K_1 asymps eqn 2} which give
        \begin{equation} \label{proof interior eigenfunctions, eqn 6}
            C^{-1} \langle \sigma \rangle^{n-1} \le K_1(\sigma) \le C \langle \sigma \rangle^{n-1} \quad \text{and} \quad
            C^{-1} \langle \sigma \rangle^{1-n} \le K_2 (\sigma) \le C \langle \sigma \rangle^{1-n}.
        \end{equation}
        Combining \eqref{proof interior eigenfunctions, eqn 5} and \eqref{proof interior eigenfunctions, eqn 6} then implies
        \begin{multline}\label{proof interior eigenfunctions, eqn 7} 
            \sup_{|\sigma | \le \Gamma } \frac{ | \langle \sigma \rangle  \partial_\sigma u | }{\langle \sigma \rangle^{\kappa + 2} }
            = \sup_{|\sigma | \le \Gamma} \frac{ |   \partial_\sigma u | }{\langle \sigma \rangle^{\kappa + 1} } 
            \le C \sup_{| \sigma | \le \Gamma } \langle \sigma \rangle^{-\kappa - n} \int_0^\sigma \langle \xi \rangle^{\kappa + n - 1} d \xi \cdot \sup_{|\xi | \le \Gamma } \frac{ |\psi(\xi)|}{\langle \xi \rangle^{\kappa} } \\
            \le C \sup_{|\sigma | \le \Gamma } \frac{ |\psi(\sigma)|}{\langle \sigma \rangle^{\kappa} }
        \end{multline}
        where we used the fact that $\kappa +n> -2+n > 0$ to estimate the integral.
        Integrating this derivative estimate and using that $\kappa > -2$ then implies 
        \begin{equation} \label{proof interior eigenfunctions, eqn 8} 
            \sup_{|\sigma | \le \Gamma } \frac{ | u |}{\langle \sigma \rangle^{\kappa + 2} } \le C \sup_{|\sigma | \le \Gamma } \frac{ | \langle \sigma \rangle  \partial_\sigma u | }{\langle \sigma \rangle^{\kappa + 2} }
            \le C \sup_{|\sigma | \le \Gamma } \frac{ |\psi(\sigma)|}{\langle \sigma \rangle^{\kappa} }.
        \end{equation}
        Finally, the second derivative estimate for $u$ follows from the fact that $Lu = \psi$, or equivalently $\partial_{\sigma \sigma} u = \psi K_1 K_2 + \frac{ \partial_\sigma K_2}{K_2} \partial_\sigma u$, and the estimates for $K_1, K_2$ and $\partial_\sigma u$ above.
        This completes the proof of \eqref{proof interior eigenfunctions, eqn 4} and the claim.
    \end{claimproof}

    Next, we prove the existence of a $w_K$ solving \eqref{proof interior eigenfunctions, eqn 3} using a fixed point argument.
    Throughout the remainder of the proof, $C = C(n,K) > 1$ denotes a constant depending only on $n,K$ which may change from line to line.
    Consider the operator
    \begin{gather}
        \mathcal F : X^1_{s_0/a} \to X^1_{s_0/a} \text{ given by}\\
        \mathcal F [u](\sigma) := -\int_0^\sigma K(\xi,\sigma) \left( \sum_{i=0}^K \hat C_{K, i} a^{2i} \Theta_i(\xi) + a^2 ( \Lambda -1+ K  - \tilde \lambda ) u(\xi) \right) d \xi  .
    \end{gather}
    If $0 < a \le s_0 \le 1$, then for any $\sigma \in (-s_0/a, s_0/a)$
    \begin{gather} \label{proof interior eigenfunctions, eqn 8.5} 
    \begin{aligned}
        \frac{\left| \sum_{i=0}^K \hat C_{K,i} a^{2i} \Theta_i(\sigma) \right| }{\langle \sigma \rangle^{-1}}
        &\le C \langle \sigma \rangle  \sum_{i=0}^K a^{2i} |\Theta_i(\sigma)| \\
        &\le C \langle \sigma \rangle  \sum_{i=0}^K a^{2i} \langle \sigma \rangle^{2i-1} && (\text{by } \eqref{eqn defn Theta_i}) \\
        &\le C \cancel{\langle \sigma \rangle}  \sum_{i=0}^K \cancel{a^{2i}}\left( \frac{s_0}{\cancel{a}} \right)^{2i} \cancel{\langle \sigma \rangle^{-1}} && \left( \text{since } |\sigma | \le \frac{s_0}a \text{ and } \frac{s_0}a \ge 1 \right) \\
        &\le C  && ( \text{since } s_0 \le 1 ).
    \end{aligned}    
    \end{gather}
    
    Additionally, for any $u \in X^1_{s_0/a}$ and any $\sigma \in (-s_0/a, s_0/a)$,
    \begin{gather} \label{proof interior eigenfunctions, eqn 9} 
    \begin{aligned}
        \frac{ |a^2 ( \Lambda -1 + K - \tilde \lambda ) u | }{\langle \sigma \rangle^{-1} }
        &\le C a^2 \langle \sigma \rangle \left( | \langle \sigma \rangle \partial_\sigma u | + | u| \right) && (\text{since $|\tilde \lambda | \le 1$} ) \\
        &\le C a^2 \langle \sigma \rangle^2 \| u \|_{X^1_{s_0/a}} \\
        &\le C (a^2 + s_0^2) \| u \|_{X^1_{s_0/a}}
        && (\text{since $|\sigma | \le s_0/a$} ) \\
        &\le C s_0^2 \| u \|_{X^1_{s_0/a}} && ( \text{since $a \le s_0$} ).
    \end{aligned}    
    \end{gather}

    These estimates and Claim \ref{proof interior eigenfunctions, claim 1} imply $\mathcal F : X^1_{s_0/a} \to X^1_{s_0/a}$ is a well-defined map when $0 <a \le  s_0 \le 1$.
    Additionally, the estimate \eqref{proof interior eigenfunctions, eqn 9} and Claim \ref{proof interior eigenfunctions, claim 1} show that $\mathcal F : X^1_{s_0/a} \to X^1_{s_0/a}$ is a contraction mapping when additionally $0 < s_0 \le  \overline s_0 ( n ,K)  \ll 1$ is sufficiently small depending only on $n,K$.
    Therefore, the contraction mapping theorem applies to give the existence of a fixed point $w_K \in X^1_{s_0/a}$ of $\mathcal F$, or equivalently a solution of \eqref{proof interior eigenfunctions, eqn 3}.
    By Claim \ref{proof interior eigenfunctions, claim 1} and estimates \eqref{proof interior eigenfunctions, eqn 8.5} and \eqref{proof interior eigenfunctions, eqn 9}, $\| w_K \|_{X^1_{s_0/a}} \le C$.
    Since $\Theta_i$ is a smooth, odd function, it follows from \eqref{proof interior eigenfunctions, eqn 3} that $w_K$ is also a smooth, odd function.
    
    By similar logic, for every $0 \le i \le K$, the existence of a smooth, odd function $v_i$ solving \eqref{proof interior eigenfunctions, eqn 2} with $\| v_i \|_{X^{2i+2}_{s_0/a}} \le C s_0^2$ can similarly be obtained via a contraction mapping argument for a suitably defined integral operator involving $K(\xi, \sigma)$ when $|\tilde \lambda | \le 1$ and $ 0 < a \le s_0 \le \overline s_0 ( n, K ) \ll 1$ is sufficiently small.

    The coefficients of the linear ODE system \eqref{proof interior eigenfunctions, eqn 2} and \eqref{proof interior eigenfunctions, eqn 3} for $(v_1, \dots, v_K, w_K)$ depend smoothly on the parameters $\tilde \lambda, a, s_0$,
    and the solutions $v_1, \dots, v_K, w_K$ as constructed above all vanish to first order at $\sigma = 0$.
    Therefore, $v_1, \dots , v_K, w_K$, and $\phi_{int}$ are smooth functions of $(\tilde \lambda, a, s_0, \sigma)$.
    Differentiating equations \eqref{proof interior eigenfunctions, eqn 2} and \eqref{proof interior eigenfunctions, eqn 3} with respect to $a$ and $\tilde \lambda$ gives similar differential equations for $\partial_a v_i, \partial_a w_K, \partial_{\tilde \lambda} v_i,$ and $\partial_{\tilde \lambda } w_K$.
    Repeatedly applying the estimate \eqref{proof interior eigenfunctions, eqn 4} from Claim \ref{proof interior eigenfunctions, claim 1} to these differential equations gives the remaining estimates in \eqref{eqn interior eigenfunctions, correction term ests}.
    The details are left as an exercise for the reader.
\end{proof}

\subsection{$G$-Invariant Special Lagrangian Cones} \label{Subsection Planes}

In this subsection, we consider the limiting case as $a \searrow 0$ and the SL limits to the asymptotic pair of cones $\mathcal C = \mathcal C_0 \cup \mathcal C_1$ given by the profile function $f_{a=0}(s) := \overline c_0 |s| $ away from $s = 0$.
Informally, as $a \searrow 0$, the operator $L_{s, a} u$ limits to
\begin{equation}
    L_{s, 0} u := \frac1{1 + \overline c_0^2} \left( \partial_{ss} u + \frac{n-1} s \partial_s u \right) \qquad \forall s > 0.
\end{equation}
We shall investigate the operator
    \begin{equation} \label{eqn H_0 defn}
        H_{s, 0} u := L_{s, 0} u - \frac s2 \partial_s u + u = \frac1{1 + \overline c_0^2} \left( \partial_{ss} u + \frac{n-1} s \partial_s u \right) - \frac s2 \partial_s u + u \qquad \forall s > 0.  
    \end{equation}
The next proposition provides eigenfunctions of $H_{s,0}$ that will be used in the construction of exterior $H_{s,a}$-eigenfunctions $\phi_{ext}$ in Lemma \ref{lem exterior eigenfunctions}.
Together with the interior eigenfunctions $\phi_{int}$ from Lemma \ref{lem interior eigenfunctions}, these exterior eigenfunctions $\phi_{ext}$ form the key ingredients for the gluing construction of the global $H_{s,a}$-eigenfunctions $\phi_{k,a}$ in Theorem \ref{thm global eigenfunctions}.

\begin{proposition} \label{prop plane eigenfunction}
    For any $k \in \mathbb N_{\ge 0}$, the functions $u_{\lambda_k} : (0, \infty) \to \R$ defined in terms of generalized Laguerre polynomials $L$ by
    \begin{equation} \label{eqn defn plane eigenfunction}
        u_{\lambda_k}(s) := L_{- \lambda_k + 1 }^{\left( \frac{n-2}2 \right)} \left( \frac{(1 + \overline c_0^2)s^2} 4 \right) = L_k^{\left( \frac{n-2}2 \right)} \left( \frac{(1 + \overline c_0^2)s^2} 4 \right), \qquad \lambda_k := 1 - k 
    \end{equation}
    are eigenfunctions of $H_{s,0}$ with eigenvalue $\lambda_k = 1-k$, that is,
        $$H_{s, 0} u_{\lambda_k} 
        = \frac{\partial_{ss} u_{\lambda_k}}{1 + \overline c_0^2} + \left( \frac{n-1}{(1+ \overline c_0^2) s} - \frac s2 \right) \partial_s u_{\lambda_k} + u_{\lambda_k} 
        = \lambda_k u_{\lambda_k} = (1-k) u_{\lambda_k}\qquad (\forall s >0).$$
    $u_{\lambda_k}$ has asymptotics
    \begin{gather} \label{eqn plane eigenfunction asymptotics}
        u_{\lambda_k}(s)  = 
        \left\{
        \begin{aligned}
            & \frac{(-1)^k \cdot ( 1 + \overline c_0^2)^k}{k!\cdot 4^k} s^{2k} + O(s^{2k-2}; n,k) \qquad & \text{as } s &\nearrow +\infty, \\
            & \frac{\Gamma\left( k + \frac n 2\right)}{k! \cdot \Gamma \left( \frac n2 \right)} + O(s^2 ; n,k) & \text{as } s &\searrow 0
        \end{aligned} \right. \\
        \label{eqn plane eigenfunction derivative asymptotics}
        \partial_s u_{\lambda_k}(s)  =
        \left\{
        \begin{aligned}
            & \frac{(-1)^k \cdot ( 1 + \overline c_0^2)^k}{k!\cdot 4^k} 2k s^{2k-1} + O(s^{2k-3}; n,k) \qquad & \text{as } s &\nearrow +\infty, \\
            & -\frac{\Gamma\left( k + \frac n 2\right)}{(k-1)! \cdot \Gamma \left( \frac n2 + 1 \right)} \frac{ (1 + \overline c_0^2)}2 s + O(s^3 ; n,k) & \text{as } s &\searrow 0
        \end{aligned} \right.
    \end{gather}
    where $\Gamma(\cdot)$ denotes the Gamma function.


    There exists another solution $v_{\lambda_k}(s)$ of 
    \begin{equation} \label{eqn defn v_lambda_k}
        H_{s,0} v_{\lambda_k} = (1-k) v_{\lambda_k}   \qquad (s >0)
    \end{equation}
    given by
    \begin{equation}
        v_{\lambda_k}(s) = u_{\lambda_k}(s) \int_1^s \frac1{u_{\lambda_k}(\sigma)} \sigma^{1-n} e^{ \frac{(1+ \overline c_0^2)\sigma^2}4} d \sigma .
    \end{equation}
    It has asymptotics
    \begin{gather} \label{eqn v_lambda_k asymptotics}
		v_{\lambda_k} (s) =
		\left\{ 
		\begin{aligned}
			&(-1)^k \frac{2 \cdot 4^k \cdot k!}{( 1+ \overline c_0^2)^{k+1}} s^{-2k - n} e^{\frac{(1+\overline c_0^2)s^2}4} + O \left(s^{-2k-n-2} e^{ \frac{(1+\overline c_0^2)s^2}4 }; n,k \right) ,	 &  \text{as } s &\nearrow + \infty\\
			&\frac1{(2-n) u_{\lambda_k}(0)}s^{2-n}  + O(s^{3-n}; n,k) , & \text{as } s&\searrow 0	
		\end{aligned} \right.
        \\
        \label{eqn v_lambda_k derivative asymptotics}
        \partial_s v_{\lambda_k}(s)  =
        \left\{ 
        \begin{aligned}
            & (-1)^k \frac{ 4^k \cdot k!}{( 1+ \overline c_0^2)^{k}} s^{-2k - n+1} e^{\frac{(1+\overline c_0^2)s^2}4} + O \left(s^{-2k-n-1} e^{ \frac{(1+\overline c_0^2)s^2}4 }; n,k \right) ,  & \text{as } s &\nearrow +\infty \\
            & \frac1{u_{\lambda_k}(0)} s^{1-n} + O(s^{3-n} ; n,k) , & \text{as } s &\searrow 0.
        \end{aligned} \right.
	\end{gather}
    $u_{\lambda_k}, v_{\lambda_k}$ are linearly independent and have Wronskian
    \begin{equation} \label{eqn defn plane Wronskian}
        W[u_{\lambda_k}, v_{\lambda_k} ] = s^{1-n} e^{ \frac{(1+\overline c_0^2) s^2}4}.
    \end{equation}
    
        
    Moreover, there exists a generalized eigenmode $\tilde u_{\lambda_k}: (0, \infty) \to \R$ satisfying
    \begin{gather} \label{eqn defn plane generalized eigenfunction 1}
        [H_{s, 0} - (1-k) ] \tilde u_{\lambda_k} = u_{\lambda_k} \qquad (\forall s > 0) , \\
        \label{eqn defn plane generalized eigenfunction 2}
        \text{or equivalently} \quad  
        \frac{\partial_{ss} \tilde u_{\lambda_k}}{1 + \overline c_0^2}  + \left( \frac{n-1}{(1+\overline c_0^2) s} - \frac s{2} \right) \partial_s \tilde u_{\lambda_k} + k \tilde u_{\lambda_k} = u_{\lambda_k} \qquad (\forall s > 0),
    \end{gather}    
    which has asymptotics
    \begin{gather} \label{eqn plane generalized eigenfunction asymptotics}
		\tilde u_{\lambda_k} (s) = 
		\left\{ 
		\begin{aligned}
			&(-1)^{k+1} \frac{2 \cdot ( 1+ \overline c_0^2)^{k-1}}{4^k \cdot k! } s^{2k} \ln (s) + O(s^{2k}; n,k),	\qquad &  \text{as } s &\nearrow + \infty\\
			&- \frac{\left( \int_0^\infty u^2 \sigma^{n-1} e^{- \frac{(1+\overline c_0^2)\sigma^2}4 } d\sigma \right)}{(2-n) u_{\lambda_k}(0)}  s^{2-n} 
            + O(s^{3-n}; n,k), & \text{as } s&\searrow 0	
		\end{aligned} \right.
        \\
        \label{eqn plane generalized eigenfunction derivative asymptotics}
        \partial_s \tilde u_{\lambda_k} (s) =
		\left\{ 
		\begin{aligned}
			&(-1)^{k+1} \frac{2 \cdot ( 1+ \overline c_0^2)^{k-1}}{4^k \cdot k! } 2k s^{2k-1} \ln (s) + O(s^{2k-1}; n,k),,	 &  \text{as } s &\nearrow + \infty\\
			&- \frac{\left( \int_0^\infty u^2 \sigma^{n-1} e^{- \frac{(1+\overline c_0^2)\sigma^2}4 } d\sigma \right)}{ u_{\lambda_k}(0)}  s^{1-n} 
            + O(s^{2-n}; n,k)  , & \text{as } s&\searrow 0	.
		\end{aligned} \right.
	\end{gather}

\end{proposition}
\begin{proof}
    The change of variables $r = \frac{(1+ \overline c_0^2) s^2}4$ transforms the operator
    \begin{gather}
        H_{s, 0} = \frac{\partial_{ss}}{1 + \overline c_0^2} + \left( \frac{n-1}{(1+\overline c_0^2)s} - \frac s2 \right) \partial_s + 1 \\
        \text{to } \quad \tilde H_{r, 0} :=  r \partial_{rr} + \left( \frac n2 - r \right) \partial_r + 1 .
    \end{gather}
    For any $k \in \mathbb N_{\ge 0}$, 
    \begin{equation} \label{proof plane eigenfunction, eqn 1}
        \tilde H_{r, 0} u = (1- k) u \quad \text{or equivalently} \quad 
        r \partial_{rr} u + \left( \frac n2 - r \right) \partial_r u +k u = 0
    \end{equation}
    is the Laguerre differential equation. It has a polynomial solution given by the generalized Laguerre polynomial
        $$u_{\lambda_k} = L_k^{\left( \frac{n-2}2 \right) } (r) = L_k^{\left( \frac{n-2}2 \right) } \left( \frac{(1+ \overline c_0^2) s^2}4 \right), $$
    which is known to have the asymptotics \eqref{eqn plane eigenfunction asymptotics}, \eqref{eqn plane eigenfunction derivative asymptotics} (see e.g. \cite{AS72}*{Ch. 22}).

    $u_{\lambda_k}$ solves $H_{s,0} u = (1-k) u$, which in divergence form becomes
        $$\frac d{ds} \left( s^{n-1} e^{- \frac{(1+\overline c_0^2)s^2}4 } \frac {du}{ds} \right) = - k (1+\overline c_0^2) u .$$
    The associated Wronskian $W$ satisfies
        $$\frac d{ds} \left( s^{n-1} e^{- \frac{(1+\overline c_0^2)s^2}4 } W \right) = 0 \qquad (\forall s > 0).$$
    It follows that
    \begin{equation} \tag{\ref{eqn defn v_lambda_k}}
        v_{\lambda_k}(s) = u_{\lambda_k}(s) \int_1^s \frac1{u_{\lambda_k}^2(\sigma)} \sigma^{1-n} e^{\frac{(1+\overline c_0^2) \sigma^2}4 } d \sigma
    \end{equation}
    gives a solution of $H_{s,0} v_{\lambda_k} = (1-k ) v_{\lambda_k}$ such that $u_{\lambda_k}, v_{\lambda_k}$ are linearly independent and have Wronskian
        $$W[u_{\lambda_k} , v_{\lambda_k} ] = u_{\lambda_k} v'_{\lambda_k} - u_{\lambda_k}' v_{\lambda_k} = s^{1-n} e^{ \frac{(1+\overline c_0^2)s^2}4}.$$
    The asymptotics of $v_{\lambda_k}$ can then be computed from \eqref{eqn defn v_lambda_k} and the asymptotics \eqref{eqn plane eigenfunction asymptotics}, \eqref{eqn plane eigenfunction derivative asymptotics}.

    Finally, for any $k \in \mathbb N_{\ge 0}$, the generalized eigenmode $\tilde u_{\lambda_k} : (0, \infty) \to \R$ solving \eqref{eqn defn plane generalized eigenfunction 1} can be obtained from the variation of parameters formula
    \begin{multline} \label{proof plane eigenfunction, eqn 2}
        \tilde u_{\lambda_k} (s)
        = - u_{\lambda_k}(s) \int_1^s \frac{ v_{\lambda_k} u_{\lambda_k}}{W} d \sigma + v_{\lambda_k} (s)\int_1^s \frac{u_{\lambda_k}^2}{W} d \sigma  
        - v_{\lambda_k} (s)\int_1^\infty \frac{u_{\lambda_k}^2}{W} d\sigma \\
        = - u_{\lambda_k}(s) \int_1^s u_{\lambda_k} v_{\lambda_k} \sigma^{n-1} e^{-\frac{(1+ \overline c_0^2)\sigma^2}4} d \sigma 
        - v_{\lambda_k}(s) \int_s^\infty u_{\lambda_k}^2 \sigma^{n-1} e^{-\frac{(1+ \overline c_0^2)\sigma^2}4} d \sigma .
    \end{multline}
    The asymptotics \eqref{eqn plane generalized eigenfunction asymptotics}, \eqref{eqn plane generalized eigenfunction derivative asymptotics} of $\tilde u_{\lambda_k}, \partial_s \tilde u_{\lambda_k}$ can then be deduced from \eqref{proof plane eigenfunction, eqn 2} and the asymptotics \eqref{eqn plane eigenfunction asymptotics}, \eqref{eqn plane eigenfunction derivative asymptotics}, \eqref{eqn v_lambda_k asymptotics}, \eqref{eqn v_lambda_k derivative asymptotics}.
\end{proof}

\subsubsection{Exterior Eigenfunctions} \label{subsubsection exterior eigenfunctions}

In this subsection, we construct eigenfunctions for $H_{s,a} = L_{s,a} - \frac s2 \partial_s + 1$ outside of a region near $s=0$.
These will be matched with the interior eigenfunctions from Subsection \ref{subsubsection interior eigenfunctions} to form globally defined eigenfunctions for small enough $a>0$.
Here, it will suffice to work with $s>0$ and perform an odd reflection for the matching argument later.

As usual, we consider the operators
\begin{gather}
    \tag{\ref{eqn H_a defn}}
    H_a = H_{s,a} = L_{s, a} - \frac s2 \partial_s + 1 \qquad ( a > 0)\\
    \tag{\ref{eqn H_0 defn}}
    \text{and } H_0 = \frac {\partial_{ss} }{1 + \overline c_0^2 } + \left( \frac{n-1}{(1 + \overline c_0^2)s} - \frac s2 \right) \partial_s + 1 .
\end{gather}
As in Subsection \ref{subsubsection interior eigenfunctions}, we begin by defining relevant weighted $C^2$-norms for the construction of the exterior eigenfunctions.

\begin{definition} \label{defn Y norms}
Given $b, b' \in \R$ and $0 < s_0 < 1$, define weighted $C^2$-norms
\begin{multline} \label{eqn Y^2 norm defn}
    \| u \|_{Y_{s_0}^{b,b'}} = \| u \|_{Y_{s_0}^{2; b,b'}}
    := \sup_{s_0 < s \le 1} \left(s^{-b} |u| + s^{-b+1} |\partial_s u| + s^{-b+2} |\partial_{ss} u| \right) \\
    + \sup_{ s \ge 1} \left(s^{-b'} (|u| + |\partial_{ss} u|) + s^{-b'+1} |\partial_s u|  \right)
\end{multline}
and its ``order 0'' analog
\begin{equation} \label{eqn Y^0 norm defn}
     \| u \|_{Y_{s_0}^{0; b,b'}}
    := \sup_{s_0 < s \le 1} \left(s^{-b} |u|  \right) 
    + \sup_{ s \ge 1} \left(s^{-b'} |u|  \right).
\end{equation}   

These norms naturally define Banach spaces
\begin{gather}
    Y^{2; b,b'}_{s_0} := \left \{ u \in C^2((s_0, \infty)) \, \colon \, \| u \|_{Y^{2; b, b'}_{s_0}} < \infty \right\}  \text{ and} \\
    Y^{0; b, b'}_{s_0} := \left \{ u \in C^0((s_0, \infty)) \, \colon \, \| u \|_{Y^{0; b, b'}_{s_0}} < \infty \right\} .
\end{gather}
\end{definition}

\begin{remark}
    Informally, $Y^{0; b, b'}_{s_0}$ captures functions $u$ that grow at most like $s^b$ near $0$ and at most like $s^{b'}$ near $\infty$.

    The $Y^{2; b,b'}_{s_0}$-norm involves (for $s \ge 1$) $s^{-b'}|\partial_{ss} u|$ but not $s^{-b'+2}|\partial_{ss} u|$ because
    solutions of say
        $$H_0 u = \psi$$
    have
        $$| \partial_{ss} u | \lesssim |s \partial_s u| + |u| + |\psi| \qquad (\text{for } s \gg 1),$$
    but 
    $$| s^2 \partial_{ss} u | \not \lesssim |s \partial_s u| + |u| + |\psi| \qquad (\text{for } s \gg 1)$$
    because of the presence of the drift term $- \frac s2 \partial_s$ in $H_0$.
\end{remark}

\begin{lem}[Exterior Eigenfunctions] \label{lem exterior eigenfunctions}
    Let $n \ge 3$, $K \in \mathbb N$, and $0 < s_0 < 1$.
    There exists $0 < \tilde \lambda^* (n, K, s_0) \ll 1$ and $0 < a^* (n, K, s_0) \ll 1$ sufficiently small depending only on $n,K, s_0$ such that the following holds:
    
    For any $0 \le | \tilde \lambda| < \tilde \lambda^*(n, K, s_0) \ll 1$,
    and $ 0 < a <  a^* (n, K,s_0) \ll 1$, 
    there exists a smooth function $\phi_{ext} : [s_0, \infty) \to \R , \phi_{ext} = \phi_{ext; \tilde \lambda , a, s_0, K,n} (s)$ such that 
    \begin{equation} \label{eqn exterior eigenfunction}
        (L_{s,a} - \frac s2 \partial_s + 1 )  \phi_{ext} = ( 1 - K + \tilde \lambda) \phi_{ext}    \qquad \text{for all } s \in (s_0, \infty)
    \end{equation}
    and such that $\phi_{ext}$ has the form 
    \begin{equation} \label{eqn exterior eigenfunctions ansatz}
        \phi_{ext} = u_{\lambda_K} + \tilde \lambda ( \tilde u_{\lambda_K} + \tilde v ) + \tilde w   
    \end{equation}
    where $u_{\lambda_K}$ ($\tilde u_{\lambda_K}$) is the (generalized) eigenfunction of $H_0$ with eigenvalue $1-K$ as in Proposition \ref{prop plane eigenfunction},
    and $\tilde v, \tilde w$ satisfy the following estimates: 
    for any $2K < b' \le  2K +1$ there exists $C = C(n, K, s_0, b') > 0$ such that 
    \begin{gather} \label{eqn exterior eigenfunctions, correction term ests}
        \begin{aligned}
        \| \tilde v \|_{Y^{2; 2-n, b'}_{s_0}} &\le C | \tilde \lambda |, &
        \partial_a \tilde v &= 0, &
        \| \partial_{\tilde \lambda} \tilde v \|_{Y^{2; 2-n, b'}_{s_0}} &\le C, 
        \\
        \| \tilde w \|_{Y^{2; 2-n, b'}_{s_0}} &\le C a^{n}, &
        \| \partial_a \tilde w \|_{Y^{2; 2-n, b'}_{s_0}} &\le C a^{n-1}, &
        \| \partial_{\tilde \lambda} \tilde w \|_{Y^{2; 2-n, b'}_{s_0}} &\le C a^{n} 
        \end{aligned}
    \end{gather}
    when additionally $0 < |\tilde \lambda | , a \ll  1$ are sufficiently small depending on $n,K,s_0$, and also $b'$.
    Moreover, $\phi_{ext}$ depends smoothly on $(\tilde \lambda, a, s)$.
\end{lem}
\begin{proof}
    Write $H_a = L_a - \frac s2 \partial_s + 1 $ as usual \eqref{eqn H_a defn}.
    Throughout the proof, we write $u_K = u_{\lambda_K}$, $v_K = v_{\lambda_K}$, and $\tilde u_K = \tilde u_{\lambda_K}$ to simplify the notation.
    Using \eqref{eqn exterior eigenfunctions ansatz} and Proposition \ref{prop plane eigenfunction}, a computation reveals
    \begin{multline} \label{proof exterior eigenfunctions, eqn 1}
        H_a \phi_{ext} - ( 1- K + \tilde \lambda ) \phi_{ext}
        = (H_a - H_0 ) \phi_{ext} 
        - \tilde \lambda^2 \tilde u_K
        + \tilde \lambda H_0 \tilde v - \tilde \lambda ( 1 - K + \tilde \lambda ) \tilde v \\
        + H_0 \tilde w - ( 1 - K + \tilde \lambda ) \tilde w.
    \end{multline}
    Thus, $\phi_{ext} = u_K + \tilde \lambda ( \tilde u_K + \tilde v) + \tilde w$ solves \eqref{eqn exterior eigenfunction} if
    \begin{align}
        \label{proof exterior eigenfunctions, eqn 2}
        H_0 \tilde v - ( 1 - K + \tilde \lambda ) \tilde v -  \tilde \lambda  \tilde u_K &=0 \text{ and} \\
        \label{proof exterior eigenfunctions, eqn 3}
        H_0 \tilde w - ( 1 - K + \tilde \lambda ) \tilde w + ( H_a - H_0) \phi_{ext} &= 0 .
    \end{align}
    The existence of $\tilde v, \tilde w$ solving \eqref{proof exterior eigenfunctions, eqn 2}, \eqref{proof exterior eigenfunctions, eqn 3} will be obtained by a fixed point argument, for suitable choices of the parameters $\tilde \lambda,a, s_0$.
    We first prove the following claim, which will be useful for the fixed point arguments that follow.
    \begin{claim} \label{proof exterior eigenfunctions, claim 1}
        Let $n \ge 3$ and $K \in \mathbb N$.
        For any $0 \le s_0 < 1$ and any $\psi \in C^0((s_0, \infty))$, there exists a solution $u \in C^2((s_0, \infty))$ of
        \begin{equation} \label{proof exterior eigenfunctions, eqn 5}
            H_0 u - ( 1 -K) u = \psi \qquad (\forall s > s_0 ) 
        \end{equation}
        given by
        \begin{multline} \label{proof exterior eigenfunctions, eqn 6}
            u (s) 
            = [(H_0 - (1-K))^{-1} \psi] (s)
            := -u_K(s) \int_1^s v_K (\sigma) \psi(\sigma) \sigma^{n-1} e^{- \frac{(1+ \overline c_0^2)\sigma}4 } d \sigma \\
            - v_K (s) \int_s^\infty u_K (\sigma ) \psi(\sigma) \sigma^{n-1} e^{- \frac{(1+ \overline c_0^2)\sigma}4 } d \sigma 
        \end{multline}
        where $u_K, v_K$ are as in Proposition \ref{prop plane eigenfunction}.

        Moreover, if $b, b' \in \R$ are such that
        \begin{equation} \label{proof exterior eigenfunctions, eqn 4}
            b \le 2-n, \qquad b \ne -n, \qquad \text{and} \qquad b' > 2K,
        \end{equation}
        then 
        \begin{equation} \label{proof exterior eigenfunctions, eqn 7}
           \psi \in Y^{0; b, b'}_{s_0} \qquad \implies \qquad u \in Y^{2; b, b'}_{s_0} \quad \text{and} \quad  \| u \|_{Y^{2; b, b'}_{s_0}} \le C \| \psi \|_{Y^{0; b, b'}_{s_0}}
        \end{equation}
        where $ C = C (n, K, b, b') > 0$ depends only on $n, K, b, b'$ (but is independent of $s_0$).
    \end{claim}
    \begin{claimproof}
        Recall from Proposition \ref{prop plane eigenfunction} that $u_K, v_K$ solve the homogeneous equation $(H_0 - (1-K) ) u = 0$.
        A straightforward computation then shows that, for any $\psi \in C^0((s_0, \infty))$, $u$ defined by \eqref{proof exterior eigenfunctions, eqn 6} is in $C^2((s_0, \infty))$ and satisfies $(H_0 - ( 1- K) ) u= 0$.

        The estimate \eqref{proof exterior eigenfunctions, eqn 7} follows from tedious but straightforward estimates using \eqref{proof exterior eigenfunctions, eqn 6} and the asymptotics of $u_K, v_K$ from Proposition \ref{prop plane eigenfunction}.
        For example, letting $C = (n, K, b, b') > 0$ be a constant depending only on $n, K, b,b'$ which may change from line to line, it follows that 
    \begin{align*}
        &\sup_{s \in (s_0, 1]} \left| s^{-b} v_K(s)  \int_s^\infty u_K(\sigma) \psi(\sigma) \sigma^{n-1} e^{- \frac{(1+ \overline c_0^2)\sigma^2}4} d \sigma \right|  
        \\
        \le{}& \sup_{s \in (s_0, 1]} C s^{2-n-b} \left( \int_s^1 |u_K| |\psi| \sigma^{-b} \sigma^{n-1+b} e^{- \frac{(1+ \overline c_0^2)\sigma^2}4} \right.
        \\ & \qquad \qquad \left. + \int_1^\infty |u_K| |\psi| \sigma^{-b'} \sigma^{n-1+b'} e^{- \frac{(1+ \overline c_0^2)\sigma^2}4} \right) 
            && (\text{by } \eqref{eqn v_lambda_k asymptotics}) 
        \\
        \le{}& \sup_{s \in (s_0, 1]} C s^{2-n-b}  \left( \int_s^1|u_K|  \sigma^{n-1+b} e^{- \frac{(1+ \overline c_0^2)\sigma^2}4} \right. \\
        & \qquad \qquad \left. + \int_1^\infty |u_K|  \sigma^{n-1+b'} e^{- \frac{(1+ \overline c_0^2)\sigma^2}4} \right) \| \psi \|_{Y^{0;b,b'}_{s_0}} 
        \\
        \le{}& \sup_{s \in (s_0, 1]} C s^{2-n-b}  \left( \int_s^1  \sigma^{n-1+b}  +    \int_1^\infty \sigma^{2K} \sigma^{n-1+b'} e^{- \frac{(1+ \overline c_0^2)\sigma^2}4} \right) \| \psi \|_{Y^{0;b,b'}_{s_0}} 
            && (\text{by } \eqref{eqn plane eigenfunction asymptotics}) 
        \\
        \le{}& \sup_{s \in (s_0, 1]} C s^{2-n-b}  \left( s^{n+ b} + 1  \right) \| \psi \|_{Y^{0;b,b'}_{s_0}} && (\text{since } b \ne -n)
        \\
        ={}&  \sup_{s \in (s_0, 1]}C  \left( s^2 + s^{2-n-b} \right) \| \psi \|_{Y^{0;b,b'}_{s_0}} 
        \\
        \le{}& C \| \psi \|_{Y^{0; b, b'}_{s_0}}.
    \end{align*}
    where the last inequality uses the fact that $b \le 2-n$ and $0 \le s_0 < 1$.
    Analogous estimates apply to estimate the other integral term on the right-hand side of \eqref{proof exterior eigenfunctions, eqn 6} and to estimate the supremum over $s \in[1,\infty)$.
    The details are left as an exercise for the reader.
    (Note that the term $v_K(s)  \int_s^\infty u_K(\sigma) \psi(\sigma) \sigma^{n-1} e^{- \frac{(1+ \overline c_0^2)\sigma^2}4} d \sigma$ in \eqref{proof exterior eigenfunctions, eqn 6} is the dominating term for $s_0 < s \le 1$,
    and the term $u_K(s) \int_1^s v_K (\sigma) \psi(\sigma) \sigma^{n-1} e^{- \frac{(1+ \overline c_0^2)\sigma}4 } d \sigma$ in \eqref{proof exterior eigenfunctions, eqn 6} is the dominating term as $s \nearrow +\infty$.)

    Differentiating \eqref{proof exterior eigenfunctions, eqn 6} gives
    \begin{multline}
        \partial_s u = - (\partial_s u_K) \int_1^s v_K(\sigma ) \psi(\sigma) \sigma^{n-1} e^{- \frac{(1+ \overline c_0^2)\sigma^2}4} d\sigma \\
        - (\partial_s v_K) \int_s^\infty u_K(\sigma) \psi(\sigma) \sigma^{n-1} e^{- \frac{(1+ \overline c_0^2)\sigma^2}4} d \sigma .
    \end{multline}
    This quantity can be estimated pointwise using the asymptotic estimates from Proposition \ref{prop plane eigenfunction}.
    Finally, $H_0 u - ( 1 - K) u = \psi$ implies
    $$\partial_{ss} u = ( 1 + \overline c_0^2) \left\{ \psi + ( 1- K ) u - \frac{ n-1}{(1 + \overline c_0^2) s } \partial_s u + \frac s2 \partial_s u - u\right\} ,$$
    and the pointwise estimates for $u, \partial_s u, \psi$ thus give pointwise estimates for $\partial_{ss} u$.    
    Estimate \eqref{proof exterior eigenfunctions, eqn 7} follows.
    The details are left as an exercise for the reader.
    \end{claimproof}

    \begin{claim} \label{proof exterior eigenfunctions, claim 2}
        For $s_0 = 0$, $b = 2-n$, any $b' > 2K$, and any $0 \le | \tilde \lambda | \le \tilde \lambda^* (n, K, b') \ll 1 $ sufficiently small (depending only on $n, K, b'$),
        the function
        \begin{equation} \label{proof exterior eigenfunctions, eqn 7.5}
        F = F_{\tilde \lambda}^{ b'}: Y^{0; 2-n, b'}_{0} \to Y^{0; 2-n, b'}_{0} \qquad \text{given by } F[u] = [H_0 - ( 1 - K)]^{-1} ( \tilde \lambda u + \tilde \lambda \tilde u_K) 
        \end{equation}
        is well-defined and has a unique fixed point $\tilde v = \tilde v_{\tilde \lambda, b'} \in Y^{0; 2-n ,b'}_{0}$.
        Moreover, the unique fixed point $\tilde v:(0, \infty) \to \R$ is smooth, solves \eqref{proof exterior eigenfunctions, eqn 2}, depends smoothly on $(\tilde \lambda, a, s)$, and satisfies estimates
        \begin{equation} \label{proof exterior eigenfunctions, eqn 7.7}
            \| \tilde v \|_{Y^{2; 2-n, b'}_{0}} \le C | \tilde \lambda |, \qquad 
            \partial_a \tilde v = 0, \qquad \text{and} \qquad 
            \| \partial_{\tilde \lambda} \tilde v \|_{Y^{2; 2-n, b'}_{0}} \le C
        \end{equation}
        for some $C = C(n, K, b') > 0$.
    \end{claim}
    \begin{claimproof}
    For $s_0 = 0$, $b = 2-n$, any $b' > 2K$, and any $0 \le | \tilde \lambda | \le 1$, consider the map
    $F = F^{b'}_{\tilde \lambda}$ given by \eqref{proof exterior eigenfunctions, eqn 7.5}.
    By the asymptotics of $\tilde u_K$ \eqref{eqn plane generalized eigenfunction asymptotics}, $\| \tilde u_K \|_{Y^{0; 2-n, b'}_{0}} \le C(n, K, b') <  \infty$ for all $b' > 2K$.
    Thus, Claim \ref{proof exterior eigenfunctions, claim 1} implies $F : Y^{0; 2-n, b'}_{0} \to Y^{0; 2-n, b'}_{0}$ is well-defined for any $b' > 2K$ and $|\tilde \lambda | \le 1$.
    Moreover, it follows from \eqref{proof exterior eigenfunctions, eqn 7} that $F : Y^{0; 2-n, b'}_{0} \to Y^{0; 2-n, b'}_{0}$ is a contraction mapping if also $|\tilde \lambda | \ll 1$ is sufficiently small depending only on $n, K, b'$.
    In this case, the contraction mapping theorem applies to give a unique fixed point $\tilde v  \in Y^{0; 2-n, b'}_{0}$ with $F[\tilde v ] = \tilde v$.
    It follows from Claim \ref{proof exterior eigenfunctions, claim 1} that $\tilde v \in Y^{2; 2-n, b'}_{0}$, $\tilde v$ solves \eqref{proof exterior eigenfunctions, eqn 2}, and 
    \begin{equation} \label{proof exterior eigenfunctions, eqn 8}
        \| \tilde v \|_{Y^{2; 2-n, b'}_{0}} \le C(n, K, b') | \tilde \lambda | 
    \end{equation}
    if $b' > 2K$ and $|\tilde \lambda | \ll 1$ is sufficiently small depending only on $n, K,b'$.
    Since $\tilde u_K$ is smooth and $\tilde v$ solves \eqref{proof exterior eigenfunctions, eqn 2}, $\tilde v$ is also smooth as a function of $s$.
    
    The coefficients of the linear ODE \eqref{proof exterior eigenfunctions, eqn 2} for $\tilde v$ depend smoothly on $\tilde \lambda $, and the construction of $\tilde v$ implies $\tilde v (1), \partial_s \tilde v(1)$ depend smoothly on $\tilde \lambda$.
    Hence, $\tilde v$ is a smooth function of $(\tilde \lambda, s)$.
    Since $F$ doesn't depend on $a$, $\tilde v$ is independent of $a$ and $\partial_a \tilde v = 0$.
    The estimate for $\partial_{\tilde \lambda} \tilde v$ \eqref{proof exterior eigenfunctions, eqn 7.7} follows from differentiating both sides of the equality
        $$\tilde v =  [ H_0 - ( 1-K)]^{-1} ( \tilde \lambda \tilde v + \tilde \lambda \tilde u_K)$$
    with respect to $\tilde \lambda$ and using \eqref{proof exterior eigenfunctions, eqn 7} and \eqref{proof exterior eigenfunctions, eqn 8}.
    \end{claimproof}

    \begin{claim} \label{proof exterior eigenfunctions, claim 3}
        For $s_0= 0$, $b = 2-n$, any $b' > 2K$, and any $0 \le | \tilde \lambda | \le \tilde \lambda^* (n, K, b') \ll 1$ sufficiently small (depending only on $n, K, b'$),
        let $\tilde v = \tilde v_{\tilde \lambda , b'}$ denote the unique fixed point of $F = F^{ b'}_{\tilde \lambda }$ from Claim \ref{proof exterior eigenfunctions, claim 2}.

        If $2K < b' \le b'_+$, then $\tilde v _{\tilde \lambda, b'_+} = \tilde v_{\tilde \lambda , b'} $ for all $0 \le |\tilde \lambda | \le \tilde \lambda^* (n, K, b') \ll 1$.
    \end{claim}
    \begin{claimproof}
        By definition of the norms, there's an inclusion $Y^{0; 2-n, b'}_{0} \subset Y^{0; 2-n, b'_+}_{0}$.
        It then follows from the definition of $F = F^{ b'}_{\tilde \lambda }$ \eqref{proof exterior eigenfunctions, eqn 7.5} that
        $F^{b'}_{\tilde \lambda} : Y^{0; 2-n, b'}_{0} \to Y^{0; 2-n, b'}_{0}$ is the restriction of $F^{b'_+}_{\tilde \lambda}: Y^{0; 2-n, b'_+}_{0} \to Y^{0; 2-n, b'_+}_{0}$, that is $F^{ b'}_{\tilde \lambda} = F^{b'_+}_{\tilde \lambda } |_{Y^{0; 2-n, b'}_{0}}$.
        Thus, $\tilde v_{\tilde \lambda , b'} $ is a fixed point of $F^{b'_+}_{\tilde \lambda }$.
        By uniqueness of the fixed point (Claim \ref{proof exterior eigenfunctions, claim 2}), $\tilde v _{\tilde \lambda, b'_+} = \tilde v_{\tilde \lambda , b'} .$
    \end{claimproof}

    By Claims \ref{proof exterior eigenfunctions, claim 2} and \ref{proof exterior eigenfunctions, claim 3}, the function $\tilde v = \tilde v_{\tilde \lambda, b'=2K+1}$, which is defined for any $0 \le |\tilde \lambda | \le \tilde \lambda^* (n, K, b'=2K+1) \ll 1$ sufficiently small depending only on $n$ and $K$, 
    is a smooth function $\tilde v : (0, \infty) \to \R$ solving \eqref{proof exterior eigenfunctions, eqn 2}.
    Moreover, for any $2K < b' \le 2K+1$, if also $0 \le |\tilde \lambda | < \tilde \lambda^* (n, K, b') \ll 1$ is sufficiently small depending only on $n, K,$ and $b'$,
    then $\tilde v = \tilde v_{\tilde \lambda , 2K+1} = \tilde v_{\tilde \lambda, b'}$ and Claim \ref{proof exterior eigenfunctions, claim 2} implies estimates \eqref{eqn exterior eigenfunctions, correction term ests} hold.

    \begin{claim} \label{proof exterior eigenfunctions, claim 4}
        For any $0 < s_0 < 1$, $b = 2-n$, and  $b' > 2K$,
        there exists $ 0 < \tilde \lambda^* (n, K, s_0, b') \ll 1$ and $ 0 < a^* (n, K, s_0,  b') \ll 1$ sufficiently small depending only on $n, K, s_0, b'$ such that the following holds:
        
        If $0 \le |\tilde \lambda | \le \tilde \lambda^* \ll 1$ and $ 0 < a \le a^* \ll 1$ and
        $\tilde v = \tilde v_{\tilde \lambda, 2K+1} = \tilde v_{\tilde \lambda , b'}$ as in Claims \ref{proof exterior eigenfunctions, claim 2} and \ref{proof exterior eigenfunctions, claim 3},
        then the function
        \begin{multline} \label{proof exterior eigenfunctions, eqn 9}
            G = G^{b'}_{\tilde \lambda,a, s_0}: Y^{2; 2-n, b'}_{s_0} \to Y^{2; 2-n, b'}_{s_0} , \\ \text{given by }
            G w = (H_0 - 1 + K )^{-1} \left[ \tilde \lambda w - ( H_a - H_0) w - (H_a - H_0) (u_K + \tilde \lambda (\tilde u_K + \tilde v) ) \right]
        \end{multline}
        is well-defined and has a unique fixed point $\tilde w = \tilde w_{\tilde \lambda, a, s_0, b'} \in Y^{2; 2-n, b'}_{s_0}$.
        Moreover, $\tilde w : (s_0, \infty) \to \R$ is smooth, solves \eqref{proof exterior eigenfunctions, eqn 3}, is a smooth function of $(\tilde \lambda, a, s)$, and satisfies the estimates
        \begin{equation} \label{proof exterior eigenfunctions, eqn 9.5}
            \| \tilde w \|_{Y^{2; 2-n, b'}_{s_0}} \le C a^n, \qquad \| \partial_a \tilde w \|_{Y^{2;2-n,b'}_{s_0}} \le C a^{n-1} , \qquad \text{and} \qquad \| \partial_{\tilde \lambda} \tilde w \|_{Y^{2; 2-n, b'}_{s_0}}\le C a^n
        \end{equation}
        for some $C = C(n, K,s_0,  b' ) > 0$.
    \end{claim}
    \begin{claimproof}
        First note that by definition of the norms $Y$ (Definition \ref{defn Y norms}),
        \begin{equation} \label{proof exterior eigenfunctions, eqn 10}
            \| \partial_{ss}u \|_{Y_{s_0}^{0; b, b'}} +\| \partial_s u \|_{Y_{s_0}^{0; b, b'}} \le \| u \|_{Y^{2; b+2, b'}_{s_0}}\le \frac 1 {s_0^2} \| u \|_{Y^{2; b, b'}_{s_0}} \lesssim_{s_0} \| u \|_{Y^{2; b, b'}_{s_0}}   .
        \end{equation}
        When $0 < s_0 < 1$, $b' > 2K$, and $|\tilde \lambda | \ll1 $ is sufficiently small depending only on $n, K, b'$, it follows that
        \begin{gather} \label{proof exterior eigenfunctions, eqn 11}
        \begin{aligned}
            &\| G w \|_{Y^{2; 2-n, b'}_{s_0}} \\
            \lesssim&{}_{n, K, b'} | \tilde \lambda | \|  w \|_{Y^{0; 2-n, b'}_{s_0}}
            + \| (H_a - H_0)  w \|_{Y^{0; 2-n, b'}_{s_0}} \\
            &+ \| (H_a - H_0 )(u_K + \tilde \lambda ( \tilde u_K + \tilde v) ) \|_{Y^{0; 2-n, b'}_{s_0}} 
            && (\text{by \eqref{proof exterior eigenfunctions, eqn 7}})\\
            \lesssim&{}_{n, s_0} | \tilde \lambda | \|  w \|_{Y^{2; 2-n, b'}_{s_0}}
            + \| a^n s^{-n} | \partial_{ss} w | + a^n s^{-n-1} | \partial_s w|  \|_{Y^{0; 2-n, b'}_{s_0}} \\
            &+ \| a^n s^{-n} | \partial_{ss} u_K | + a^n s^{-n-1} | \partial_s u_K| \|_{Y^{0; 2-n, b'}_{s_0}} \\
            &+ |\tilde \lambda | \| a^n s^{-n} | \partial_{ss} ( \tilde u_K + \tilde v) | + a^n s^{-n-1} | \partial_s ( \tilde u_K + \tilde v ) | \|_{Y^{0; 2-n, b'}_{s_0}}
            && (\text{by Lemma \ref{lem H_a - H_0 pointwise est}}) \\
            \lesssim&{}_{s_0} |\tilde \lambda | \| w \|_{Y^{2; 2-n, b'}_{s_0}}
            + a^n \| w\|_{Y^{2; 2-n, b'}_{s_0}} \\
            &+ a^n \| u_K \|_{Y^{2; 2-n, b'}_{s_0}}
            + |\tilde \lambda|   a^n \| \tilde u_K \|_{Y^{2; 2-n, b'}_{s_0}}  
            +|\tilde \lambda | a^n \| \tilde v \|_{Y^{2; 2-n, b'}_{s_0}} 
            && (\text{by \eqref{proof exterior eigenfunctions, eqn 10}}) \\
            \le{}& ( | \tilde \lambda | + a^n ) \| w \|_{Y^{2; 2-n, b'}_{s_0}} 
            +  C(n, K, s_0, b') \cdot a^n
        \end{aligned}
        \end{gather}
        where the last line uses the fact that $\| \tilde v \|_{Y^{2; 2-n, b'}_{s_0}} \le C(n, K, s_0, b') $ for $b' > 2K$ and $|\tilde \lambda | \ll 1$ sufficiently small depending only on $n, K, b'$ (see \eqref{eqn exterior eigenfunctions, correction term ests})
        and uses the fact that the asymptotics of $u_K, \tilde u_K$ from Proposition \ref{prop plane eigenfunction} implies $\| u_K \|_{Y^{2; 2-n, b'}_{s_0}} + \| \tilde u_K \|_{Y^{2; 2-n, b'}_{s_0}} \le C(n,K, b')$.
        In particular, $G : Y^{2; 2-n, b'}_{s_0} \to Y^{2 ; 2-n, b'}_{s_0}$ is well-defined for $0 < s_0 < 1$, $b' > 2K$, and $|\tilde \lambda | \ll 1$ sufficiently small depending only on $n, K, b'$.
        
        A similar estimate as above applies to show that for all $w_1, w_2 \in Y^{2; 2-n, b'}_{s_0}$
        \begin{gather}\label{proof exterior eigenfunctions, eqn 11.5}
        \begin{aligned} 
            \| G w_1 - G w_2 \|_{Y^{2; 2-n, b'}_{s_0}} &\le C(n, K, b',s_0) ( |\tilde \lambda | + a^n ) \| w_1 - w_2 \|_{Y^{2; 2-n, b'}_{s_0}} \\
            &< \frac12 \| w_1 - w_2 \|_{Y^{2; 2-n, b'}_{s_0}}
        \end{aligned} \end{gather}        
        when $0 < s_0 < 1$, $b' > 2K$, and $a, |\tilde \lambda | \ll 1$ are sufficiently small depending only on $n, K, s_0, b'$.
        In this case, the contraction mapping theorem applies to give the existence of a unique fixed point $\tilde w = \tilde w_{\tilde \lambda,a,  s_0,b'} \in Y^{2; 2-n, b'}_{s_0}$, $G \tilde w = \tilde w$.
        By the definition of $[H_0 - (1-K)]^{-1}$ \eqref{proof exterior eigenfunctions, eqn 6}, this fixed point $\tilde w :(s_0, \infty) \to \R$ solves \eqref{proof exterior eigenfunctions, eqn 3} and is smooth since $u_K, \tilde u_K, \tilde v$ are smooth.
        By estimate \eqref{proof exterior eigenfunctions, eqn 11},
        \begin{gather} \label{proof exterior eigenfunctions, eqn 12}
        \begin{aligned}
            \| \tilde w \|_{Y^{2; 2-n,b'}_{s_0}} 
            = \| G \tilde w \|_{Y^{2; 2-n,b'}_{s_0}}  
            \le C(n, K, s_0, b') ( | \tilde \lambda | + a^n ) \| \tilde w \|_{Y^{2; 2-n,b'}_{s_0}} + C(n, K, s_0,b') a^n \\
            \le \frac12 \| \tilde w \|_{Y^{2; 2-n,b'}_{s_0}} + C(n, K, s_0, b') a^n.
        \end{aligned}
        \end{gather}
        Subtracting $\frac12 \| \tilde w \|_{Y^{2; 2-n,b'}_{s_0}}$ from both sides proves
        $\| \tilde w \|_{Y^{2;2-n, b'}_{s_0}} \le C(n, K, s_0, b') \cdot a^n$.

        Because the coefficients of \eqref{proof exterior eigenfunctions, eqn 3} are smooth functions of $(\tilde \lambda, a)$, $\tilde v$ depends smoothly on $(\tilde \lambda , a,s)$, and $\tilde w(1), \partial_s\tilde w (1)$ depend smoothly on $(\tilde \lambda, a)$, it follows that $\tilde w$ depends smoothly on $(\tilde \lambda, a,s)$.
        Differentiating both sides of 
        \begin{equation}
            \tilde w = G \tilde w = (H_0 - 1 + K)^{-1} \left[ \tilde \lambda w - ( H_a - H_0) w - (H_a - H_0) (u_K + \tilde \lambda (\tilde u_K + \tilde v) ) \right]
        \end{equation}
        with respect to $\tilde \lambda$ and $a$ shows 
        \begin{equation} \label{proof exterior eigenfunctions, eqn 13}
            \partial_{\tilde \lambda} \tilde w = (H_0 - 1 +K)^{-1} \left[ \tilde w+ \tilde \lambda \partial_{\tilde \lambda} \tilde w - (H_a -H_0) \partial_{\tilde \lambda } \tilde w - (H_a - H_0 ) (\tilde u_K + \tilde v + \tilde \lambda \partial_{\tilde \lambda } \tilde v )  \right] 
        \end{equation} 
        \begin{multline} \label{proof exterior eigenfunctions, eqn 14}
            \text{and } \partial_a \tilde w = (H_0 - 1 + K)^{-1} \left[ \tilde \lambda \partial_a \tilde w - (\partial_a H_a) \tilde w - (H_a - H_0) \partial_a \tilde w \right.\\
             \left.- (\partial_a H_a) ( u_K + \tilde \lambda ( \tilde u_K + \tilde v)) \right] 
        \end{multline}
        where $\partial_a H_a$ is the operator defined by \eqref{proof partial_a H_a pointwise est, eqn 2}.
        Lemma \ref{lem partial_a H_a pointwise est} implies
        \begin{equation}
            \| (\partial_a H_a ) ( u_K + \tilde \lambda ( \tilde u_k + \tilde v) + \tilde w) \|_{Y^{0; 2-n, b'}_{s_0}} 
            \le C(n, K, s_0, b') a^{n-1} 
        \end{equation}
        and \eqref{proof exterior eigenfunctions, eqn 7.7} gives $\| \partial_{\tilde \lambda } \tilde v \|_{Y^{0; 2-n, b'}_{s_0}} \le C(n, K, s_0, b')$.
        The claimed estimates \eqref{proof exterior eigenfunctions, eqn 9.5} for $\partial_{\tilde \lambda} \tilde w$ and $\partial_a \tilde w$ then follow by applying similar logic as in the estimates \eqref{proof exterior eigenfunctions, eqn 11} and \eqref{proof exterior eigenfunctions, eqn 11.5} to equations \eqref{proof exterior eigenfunctions, eqn 13} and \eqref{proof exterior eigenfunctions, eqn 14}.
    \end{claimproof}

    A priori, the fixed point $\tilde w = \tilde w_{\tilde \lambda, a, s_0, b'}$ from Claim \ref{proof exterior eigenfunctions, claim 4} depends on $s_0 \in (0,1)$ and $b' \in (2K, \infty)$.
    However, similar logic as in Claim \ref{proof exterior eigenfunctions, claim 3} gives that, for any $0 < s_0 \le s_0^+ < 1$ and $2K < b' \le b'_+$, 
    \begin{equation}
        \tilde w_{\tilde \lambda, a, s_0, b'} |_{(s_0^+, \infty)} = \tilde w_{\tilde \lambda , a, s_0^+, b'_+}
    \end{equation}
    for all $0 \le |\tilde \lambda | \le \tilde \lambda^* (n, K, s_0, b') \ll 1$ and $0 < a \le a^* (n, K, s_0, b') \ll 1$.
    Thus, taking 
    \begin{equation*}
        \phi_{ext} = u_K + \tilde \lambda ( \tilde u_K + \tilde v) + \tilde w 
    \end{equation*}
    with $\tilde v = \tilde v_{\tilde \lambda , 2K+1} |_{[s_0, \infty)}$ and $\tilde w = \tilde w_{\tilde \lambda, a, \frac{s_0}2, 2K+1} |_{[s_0, \infty)}$ as in Claims \ref{proof exterior eigenfunctions, claim 2} and \ref{proof exterior eigenfunctions, claim 4} completes the proof.
\end{proof}

\subsection{Matching Interior and Exterior Eigenfunctions} \label{Subsection Matching}

In this subsection, we use a matching argument to construct global eigenfunctions of $H_{s,a}$ defined for $s \in \R$.
This is done in Theorem \ref{thm global eigenfunctions} by gluing the interior and exterior eigenfunctions constructed in Lemmas \ref{lem interior eigenfunctions} and \ref{lem exterior eigenfunctions} above.
Theorem \ref{thm global eigenfunctions} will complete a key portion of the paper's main result, Theorem \ref{main thm intro}, stated in the introduction.

\begin{remark}
    For the statements that follow, it will be convenient to use the simplified notation ``for all $0 < a < a^* (n,k,s_0) \ll 1$'' to mean ``for all positive $a$ sufficiently small depending only on $n,k,s_0$'' or, more specifically,
    ``there exists $0< a^* = a^* (n,k,s_0)<1$ depending only on $n,k,s_0$ such that for all $0 < a < a^*$.''
    See for example Lemma \ref{lem matching prelude} below.
\end{remark}

We first require the following two lemmas (Lemmas \ref{lem rewriting Laguerre poly} and \ref{lem matching prelude}), which ultimately show the interior and exterior eigenfunctions agree to leading order on their overlapping domains.

\begin{lem} \label{lem rewriting Laguerre poly}
    Let $n \ge 3$, $k \in \mathbb N$, and $P_k(s)$ be the polynomial in $s$ given in terms of the generalized Laguerre polynomial as $P_{k}(s) = L_k^{\left( \frac{n-2}2 \right) } \left( \frac {(1 + \overline c_0^2) s^2}{4} \right)$.
    Let $\hat C_{k,i}$ be the constants as in Lemma \ref{lem interior eigenfunctions} and let $u_i$ be the generalized kernel elements with asymptotics
    \begin{equation} \tag{\ref{generalized kernel asymptotic expansion eqn}}
        u_i(\sigma) = \pm C_i\sigma^{2i} + O ( \sigma^{2i-1}; n, i) \qquad \text{as } \sigma \to \pm \infty
    \end{equation}
    as in Proposition \ref{prop generalized kernel}.

    Then, there exists a constant $C = C(n,k) \ne 0$ such that for all $s_0 > 0$ and $0 < a < a^* (n, k , s_0) \ll 1$,
    \begin{align} 
        \label{lem rewriting Laguerre poly, eqn 1}
        \sum_{i=0}^k \hat C_{k,i} a^{2i} u_i (s/a) &= C P_k (s) + O(a; n, k,s_0) \cdot s^{2k-1}  && \forall s \ge s_0 , \text{ and}\\
        \label{lem rewriting Laguerre poly, eqn 2}
        \frac \partial{ \partial s} \left( \sum_{i=0}^k \hat C_{k,i} a^{2i} u_i (s/a) \right) &= C P_k' (s) + O(a; n, k,s_0) \cdot s^{2k-2}  && \forall s \ge s_0 .
    \end{align}
\end{lem}
\begin{proof}  
    The asymptotics \eqref{generalized kernel asymptotic expansion eqn} imply that for all $s \ge s_0$
    \begin{multline} \label{proof rewriting Laguerre poly, eqn 1}
        \sum_{i=0}^k \hat C_{k,i} a^{2i} u_i(s/a)
        = \sum_{i=0}^k \left[ \hat C_{k,i}  C_i s^{2i} + O(a s^{2i-1} ; n, i) \right]  \\
        = \sum_{i=0}^k \hat C_{k,i} C_i s^{2i} + O(a; n, k,s_0) \cdot s^{2k-1}.
    \end{multline}
    Similarly, by Proposition \ref{prop generalized kernel}, for all $s \ge s_0$
    \begin{equation} \label{proof rewriting Laguerre poly, eqn 2}
        \frac \partial {\partial s} \left( \sum_{i=0}^k \hat C_{k,i} a^{2i} u_i(s/a) \right) 
        = \sum_{i=0}^k \hat C_{k,i} 2i C_i  s^{2i-1}  + O ( a ; n, k, s_0) s^{2k- 2} .
    \end{equation}
    The proof will then follow from the following claim.
    \begin{claim} \label{claim rewriting Laguerre poly}
        There exists a constant $C = C(n,k) \ne  0$ such that 
        \begin{equation} \label{eqn rewriting Laguerre poly}
            \sum_{i=0}^k \hat C_{k,i} C_i s^{2i} = C P_k(s) \qquad (\forall s > 0 ).
        \end{equation}
    \end{claim}
    \begin{claimproof}
        $P_k (s) = L_k^{\left( \frac{n-2}2 \right)} \left( \frac{(1 + \ol{c}_0^2) s^2}4 \right) = \sum_{i=0}^k B_i s^{2i}$ is a polynomial of order $2k$ in $s$ whose coefficients satisfy the inductive formula
        $$B_{i+1} = - \frac{(k-i) ( 1 + \ol{c}_0^2)}{(2i+2)(2i+n)} B_i \qquad \forall 0 \le i < k.$$
        Indeed, the generalized Laguerre polynomial $u_k(r) = L_k^{\left( \frac{n-2}2 \right)} \left( \frac{r^2}4 \right) = \sum_{i = 0}^k \beta_i r^{2i}$ satisfies the differential equation
            $$u_{rr} + \left( \frac{n-1}r - \frac r2 \right) u_r + u = ( 1-k) u.$$
        After substituting in the polynomial expression for $u(r) = \sum_{i = 0}^k \beta_i r^{2i}$ into this differential equation, it follows that
            $$\beta_{i+1} = - \frac{(k-i)}{(2i+2)(2i+n)} \beta_i \qquad \forall 0 \le i < k.$$
        The inductive formula for $B_i$ then follows from the fact that $r^2 = ( 1 + \ol{c}_0^2 ) s^2$.

        By Proposition \ref{prop generalized kernel} and Lemma \ref{lem interior eigenfunctions}, $\sum_{i=0}^k \hat C_{k,i} C_i s^{2i} = \sum_{i=0}^k b_i s^{2i}$ is a polynomial of order $2k$ in $s$ whose coefficients $b_i = \hat C_{k,i} C_i$ satisfy the inductive formula 
        $$b_{i+1} = \hat C_{k, i+1} C_{i+1}
        = - (k-i) \hat C_{k, i} \frac{( 1 + \ol{c}_0^2) C_i }{(2i+2)(2i+n)}
        = - \frac{(k-i) ( 1 + \ol{c}_0^2)}{(2i+2)(2i+n)} b_i \qquad \forall 0 \le i < k.$$
        Since the coefficients $b_i, B_i$ satisfy the same inductive formula,
        it follows that, for some $C = C(n,k)$,
        \begin{equation}  
            \sum_{i=0}^k \hat C_{k,i} C_i s^{2i} = C P_k(s) \qquad (\forall s > 0).
        \end{equation} 
        Additionally, $C \ne 0$ because
            $$B_0 \ne 0 \quad \text{and} \quad b_0 = \hat C_{k,0} C_0 = 1 \cdot \lim_{s \to +\infty} \beta_1(s) > 0$$
        by Lemma \ref{lem beta properties}.
    \end{claimproof}
    The lemma now follows from combining Claim \ref{claim rewriting Laguerre poly} with \eqref{proof rewriting Laguerre poly, eqn 1} and \eqref{proof rewriting Laguerre poly, eqn 2}.
\end{proof}

\begin{lem} \label{lem matching prelude}
    Let $n \ge 3$, $k \in \mathbb N$, and
    $P_{k}(s)$ be the polynomial in $s$ given in terms of the generalized Laguerre polynomial as $P_{k}(s) = L_k^{\left( \frac{n-2}2 \right) } \left( \frac {(1 + \overline c_0^2) s^2}{4} \right)$.
    
    There exists $C = C(n,k) \ne 0$ such that 
    for all $0 < s_0 \le s_1 \le s_1^* ( n, k) \ll 1$, 
    $0 < a \le a^* (n, k, s_0, s_1) \ll 1$, and
    $|\tilde \lambda | \le \tilde \lambda^* ( n, k, s_0, s_1) \ll 1$,
    the following estimates hold for the interior eigenfunctions $\phi_{int,k} = \phi_{int; \tilde \lambda, a,s_1,k}(\sigma = s/a)$ from Lemma \ref{lem interior eigenfunctions} and the exterior eigenfunctions $\phi_{ext,k} = \phi_{ext; \tilde \lambda, a, s_0, k} (s)$ from Lemma \ref{lem exterior eigenfunctions}:
    \begin{align}
        \label{lem matching prelude, eqn 1}
        \phi_{ext,k}(s) &= P_k(s) + O ( |\tilde \lambda|; n,k,s_0) + O(a^n; n,k,s_0) ,\\
        \label{lem matching prelude, eqn 2}
        \partial_s \phi_{ext,k} (s) &= P'_k(s) + O ( |\tilde \lambda|; n,k,s_0) + O(a^n; n,k,s_0), \\
        \label{lem matching prelude, eqn 3}
        \phi_{int,k}(s/a) &= C P_k(s) + O ( |\tilde \lambda|; n,k,s_0) + O(a; n,k,s_0) , \text{ and}  \\
        \label{lem matching prelude, eqn 4}
        a^{-1} \partial_\sigma \phi_{int,k}(s/a) &= C P_k'(s) + O ( |\tilde \lambda|; n,k,s_0) + O(a; n,k,s_0),
    \end{align}
    for all $s \in [s_0, s_1]$.
\end{lem}
\begin{proof}
    For all 
        $$0 < s_0 < s_1 \le s_1^* (n,k) \ll 1, \, 0 < a \le a^* (n,k,s_0,s_1) \ll1 ,  \text{ and } |\tilde \lambda |\le \tilde \lambda^* ( n, k, s_0, s_1) \ll1,$$
    let $\phi_{int, k} = \phi_{int; \tilde \lambda, a, s_1, k}$ and $\phi_{ext, k} = \phi_{ext; \tilde \lambda , a , s_0, k}$ denote the interior and exterior eigenfunctions given by Lemmas \ref{lem interior eigenfunctions} and \ref{lem exterior eigenfunctions}.

    The estimates \eqref{lem matching prelude, eqn 1} and \eqref{lem matching prelude, eqn 2} follow immediately from Lemma \ref{lem exterior eigenfunctions}.
    
    For any $s \in [s_0, s_1]$,
    \begin{align*}
        \phi_{int, k} (s/a) 
        ={}& \sum_{i= 0}^k \hat C_{k,i} a^{2i} u_i(s/a) \\
            &+ \tilde \lambda \sum_{i = 0}^k a^{2(i+1)} \left( \hat C_{k, i} u_{i+1}(s/a) + v_i(s/a) \right) + a^2 w_k(s/a) \\
        ={}& \sum_{i= 0}^k \hat C_{k,i} a^{2i} u_i(s/a) 
            + O ( |\tilde \lambda| ; n, k ,s_0) + O ( a; n, k,s_0)  
        && ( \text{by Lemma \ref{lem interior eigenfunctions}}) \\ 
        ={}& C(n,k) P_k(s) + O (|\tilde \lambda | ; n, k,s_0) + O (a; n,k,s_0) 
        && ( \text{by Lemma \ref{lem rewriting Laguerre poly}} )
    \end{align*}
    This proves \eqref{lem matching prelude, eqn 3}.
    A similar argument applies to the derivative of $\phi_{int,k}$ to prove \eqref{lem matching prelude, eqn 4}.
\end{proof}

\begin{theorem} \label{thm global eigenfunctions}
    Let $n \ge 3$ and $K \in \mathbb N$.
    For all $0 < s_0 \le s_0^* (n,K) \ll 1$,
    $0 < a \le a^* (n,K, s_0 ) \ll 1 $,
    and $k \in \mathbb N $ with $0 \le k \le K$,
    there exists a unique $\tilde \lambda_k = \tilde \lambda_k ( a; n,k, s_0)$ with 
    \begin{equation} \label{thm global eigenfunctions, eqn 1}
        |\tilde \lambda_k (a) | \lesssim_{n,k,s_0} a  \quad \text{and} \quad 
        |\partial_a \tilde \lambda_k (a) | \lesssim_{n,k,s_0} 1
    \end{equation}
    such that $\phi_k = \phi_{k,a,s_0}$ defined by
    \begin{gather} \label{eqn global eigenfunction defn}
        \phi_k(s) = 
        \left\{ \begin{aligned}
            &\phi_{int; \tilde \lambda_k , a, s_0, k}\left( \frac s a \right), & \text{if } |s| < s_0,  \\
            &\frac{\phi_{int; \tilde \lambda_k, a, s_0, k} \left( \frac {s_0} a \right) }{\phi_{ext; \tilde \lambda_k ,a, s_0, k}(s_0)}\phi_{ext; \tilde \lambda_k, a , s_0, k}(s), & \text{if } |s| \ge s_0.
        \end{aligned}
        \right.
    \end{gather}
    (where $\phi_{int; \tilde \lambda_k, a, s_0, k}$ and $\phi_{ext; \tilde \lambda_k , a, s_0, k}$ are given by Lemmas \ref{lem interior eigenfunctions}, \ref{lem exterior eigenfunctions} respectively)
    is a smooth eigenfunction
        $$H_{s,a} \phi_k = \lambda_k \phi_k \qquad \text{with eigenvalue } \lambda_k = 1 - k + \tilde \lambda_k.$$
\end{theorem}
\begin{proof}
    Let $k \in \mathbb N$ with $0 \le k \le K$.
    Since $\phi_{int} =  \phi_{int; \tilde \lambda, a, s_0, k}(\sigma = s/a)$ and $\phi_{ext} =  \phi_{ext; \tilde \lambda , a, s_0, k}(s)$ solve the same ODE $H_{s,a} u = (1- k + \tilde \lambda) u$, $\phi_k$ defined by \eqref{eqn global eigenfunction defn} will be smooth if the expressions in the piecewise definition \eqref{eqn global eigenfunction defn} match to first order at $s = s_0$.
    They clearly match to zeroth order.
    To measure the first-order difference, define
    \begin{equation} \label{proof global eigenfunctions, eqn 1}
        \Phi[s_0](\tilde \lambda ,a) := \frac{a^{-1} \phi_{int}'(s_0/a)}{\phi_{int} (s_0/a)} - \frac{ \phi_{ext}'(s_0)}{\phi_{ext}(s_0)}.
    \end{equation}

    \begin{claim} \label{proof global eigenfunctions, claim 1}
        There exists a positive constant $C= C(n,k) > 0$ depending only on $n,k$ such that,
        for all $0 < s_0 \le s_0^* (n, k) \ll1$,
        $0 < a \le a^* (n, k, s_0) \ll1$,
        and $ |\tilde \lambda| \le \tilde \lambda^* ( n, k, s_0) \ll 1$,
        \begin{gather}
            \label{proof global eigenfunctions, claim 1, eqn 1}
            \Phi[s_0](\tilde \lambda, a) = O ( |\tilde \lambda|; n, k, s_0) + O ( a ; n, k, s_0 )   \text{ and}\\
            \label{proof global eigenfunctions, claim 1, eqn 2}
            \partial_{\tilde \lambda} \Phi = C(n,k) s_0^{1-n} + O ( s_0^{2-n} ;n,k) + O ( | \tilde \lambda |; n,k,s_0) + O( a ; n,k,s_0) .
        \end{gather}
    \end{claim}
    \begin{claimproof}
        Denote
        $$P(s_0) =  L_k^{\left( \frac{n-2}2 \right) } \left( \frac{( 1 + \ol{c}_0^2) s_0^2}4 \right) $$
        and assume $0 < s_0^* (n,k) \ll 1$ is taken sufficiently small so that $P(s_0) \ne 0$ for all $0 < s_0 \le s_0^* (n,k) \ll 1$.
        Then Lemma \ref{lem matching prelude} implies that for some $C_{n,k}\ne 0 $ depending only on $n,k$
        \begin{gather} \label{proof global eigenfunctions, proof claim 1, eqn 1} \begin{aligned}
            \Phi[s_0]( \tilde \lambda , a)  
            ={}& \frac{ C_{n,k} P'(s_0) + O ( | \tilde \lambda  | +a; n, k,s_0 )  }{C_{n,k} P(s_0) + O( |\tilde \lambda|+ a ; n,k,s_0)  }
            - \frac{  P'(s_0) + O ( | \tilde \lambda  |+ a^n; n,k,s_0 )  }{ P(s_0) + O( |\tilde \lambda|+ a^n; n,k,s_0)  } \\
            ={}& \frac{ O ( | \tilde \lambda|+a; n,k,s_0 )  }{P (s_0)^2 + O ( |\tilde \lambda |+ a; n,k,s_0 )  } \\
            ={}& O ( | \tilde \lambda|+ a ; n,k,s_0 ).
        \end{aligned} \end{gather}
        This proves \eqref{proof global eigenfunctions, claim 1, eqn 1}.

        Next, a straightforward computation using \eqref{proof global eigenfunctions, proof claim 1, eqn 1} reveals
        \begin{gather} \label{proof global eigenfunctions, proof claim 1, eqn 2}
        \begin{aligned}
            \partial_{\tilde \lambda} \Phi 
            ={}& \frac{a^{-1} \phi_{int}'(s_0/a)}{\phi_{int}(s_0/a)} 
            \left[ \frac{ \partial_{\tilde \lambda} \phi_{int}'(s_0/a)}{ \phi_{int}'(s_0/a) } - 
            \frac{ \partial_{\tilde \lambda} \phi_{int}(s_0/a)}{ \phi_{int}(s_0/a) } \right]\\
            & - \frac{ \phi_{ext}'(s_0)}{\phi_{ext}(s_0)} 
            \left[ \frac{ \partial_{\tilde \lambda} \phi_{ext}'(s_0)}{ \phi_{ext}'(s_0) } - 
            \frac{ \partial_{\tilde \lambda} \phi_{ext}(s_0)}{ \phi_{ext}(s_0) } \right] \\
            ={}& \left( \frac{ \phi_{ext}'(s_0)}{\phi_{ext}(s_0)} + O( |\tilde \lambda| + a  ; n, k, s_0)   \right)  \\
            &\cdot \left[ \frac{ \partial_{\tilde \lambda} \phi_{int}'(s_0/a)}{ \phi_{int}'(s_0/a) } - 
            \frac{ \partial_{\tilde \lambda} \phi_{int}(s_0/a)}{ \phi_{int}(s_0/a) } 
            -\frac{ \partial_{\tilde \lambda} \phi_{ext}'(s_0)}{ \phi_{ext}'(s_0) } + 
            \frac{ \partial_{\tilde \lambda} \phi_{ext}(s_0)}{ \phi_{ext}(s_0) } \right] .
        \end{aligned} \end{gather}
    We compute the asymptotics of terms in this expansion.
    For the remainder of the proof of this claim, all uses of the $O(\cdot)$ notation are in fact $O(\cdot; n,k,s_0)$ unless indicated otherwise.
    
    By Lemma \ref{lem matching prelude},
    \begin{gather} \label{proof global eigenfunctions, proof claim 1, eqn 3}
        \frac{  \phi_{ext}'(s_0) }{\phi_{ext}(s_0)} 
        = \frac{P'(s_0) + O( |\tilde \lambda |) + O ( a^{n} ) }{P(s_0) + O ( |\tilde \lambda | ) + O ( a^{n} ) } 
        =  \frac{P'(s_0)}{P(s_0)} + O ( |\tilde \lambda | ) + O ( a^{n} ) ,\\
        \label{proof global eigenfunctions, proof claim 1, eqn 3.5}
        \text{and so } \left| \frac{  \phi_{ext}'(s_0) }{\phi_{ext}(s_0)}  \right| \sim_{n,k} s_0 .
    \end{gather}
    Lemma \ref{lem exterior eigenfunctions} implies that at $s=s_0$
    \begin{gather}
        \label{proof global eigenfunctions, proof claim 1, eqn 3.6}
        \partial_{ \tilde \lambda} \phi_{ext} 
        = \tilde u_{\lambda_k} + \tilde v + \tilde \lambda \partial_{\tilde \lambda} \tilde v + \partial_{\tilde \lambda} w 
        = \tilde u_{\lambda_k}(s_0) + O ( | \tilde \lambda | ) + O ( a^{n} ) \text{ and} \\
        \label{proof global eigenfunctions, proof claim 1, eqn 3.7}
        \partial_{\tilde \lambda} \phi_{ext}' 
        = \tilde u_{\lambda_k}' + \tilde v' + \tilde \lambda \partial_{\tilde \lambda} \tilde v' + \partial_{\tilde \lambda} w' 
        = \tilde u_{\lambda_k}'(s_0) + O ( | \tilde \lambda | ) + O ( a^{n} ) .
    \end{gather}
    Hence, 
    \begin{gather} \label{proof global eigenfunctions, proof claim 1, eqn 4}
    \begin{aligned}
        -\frac{ \partial_{\tilde \lambda} \phi_{ext}'(s_0)}{ \phi_{ext}'(s_0) } + 
        \frac{ \partial_{\tilde \lambda} \phi_{ext}(s_0)}{ \phi_{ext}(s_0) } 
        &= - \frac{  \tilde u_{\lambda_k}'(s_0) + O ( | \tilde \lambda |) + O ( a^{n} ) }{ P'(s_0) + O ( | \tilde \lambda | ) + O ( a^{n} ) } 
        + \frac{ \tilde u_{\lambda_k}(s_0) + O ( | \tilde \lambda |) + O ( a^{n} ) }{ P(s_0) + O ( | \tilde \lambda | ) + O ( a^{n} ) }  \\
        &= - \frac{  \tilde u_{\lambda_k}'(s_0) }{ P'(s_0) } + \frac{ \tilde u_{\lambda_k} (s_0) }{ P(s_0) } + O ( | \tilde \lambda |) + O ( a^{n} ) \\
        &= - \frac{  \tilde u_{\lambda_k}'(s_0) }{ P'(s_0) } + O( s_0^{2-n} ; n,k) + O ( |\tilde \lambda | ) + O(a^n),
    \end{aligned} \end{gather}
    where the last line uses the asymptotics of $\tilde u_{\lambda_k}$ from Proposition \ref{prop plane eigenfunction}.
    In particular,
    \begin{equation} \label{proof global eigenfunctions, proof claim 1, eqn 5}
        \left| -\frac{ \partial_\lambda \phi_{ext}'(s_0)}{ \phi_{ext}'(s_0) } + 
        \frac{ \partial_\lambda \phi_{ext}(s_0)}{ \phi_{ext}(s_0) } \right| \sim_{n,k} s_0^{-n}
    \end{equation}
    by the asymptotics of $\tilde u_{\lambda_k}, \tilde u_{\lambda_k}'$ from Proposition \ref{prop plane eigenfunction}.

    Lemma \ref{lem interior eigenfunctions} implies that, at $\sigma = s_0/a$, we have
    \begin{align}
        \partial_{\tilde \lambda} \phi_{int,k}(s_0 /a) 
        ={}& \sum_{i=0}^k a^{2(i+1)} \left( \hat C_{k, i} u_{i+1} + v_i \right) 
        + \tilde \lambda \sum_{i=0}^k a^{2(i+1)} \partial_{\tilde \lambda} v_i 
        + a^2 \partial_{\tilde \lambda} w_k \\
        a^{-1} \partial_{\tilde \lambda} \phi_{int,k}'(s_0 /a) 
        ={}& \sum_{i=0}^k a^{2i+1} \left( \hat C_{k, i} u_{i+1}' + v_i' \right) 
        + \tilde \lambda \sum_{i=0}^k a^{2i+1} \partial_{\tilde \lambda} v_i' 
        + a \partial_{\tilde \lambda} w_k'. 
    \end{align}
    It follows from Lemma \ref{lem interior eigenfunctions} and Proposition \ref{prop generalized kernel} that, at $\sigma = s_0/a$,
    \begin{align} \label{proof global eigenfunctions, proof claim 1, eqn 6}
        \partial_{\tilde \lambda} \phi_{int,k}(s_0 /a) 
        ={}& \sum_{i=0}^k a^{2(i+1)} \left( \hat C_{k, i} u_{i+1} + v_i \right) 
        + \tilde \lambda \sum_{i=0}^k a^{2(i+1)} \partial_{\tilde \lambda} v_i 
        + a^2 \partial_{\tilde \lambda} w_k \\
        ={}& \sum_{i =0}^k a^{2(i+1) } \left( \hat C_{k,i} u_{i+1} (s_0/a) + O ( s_0^2 ; n,k) \left( \frac sa \right)^{2i+2}  \right) 
         + O ( |\tilde \lambda | ) + O ( a) \notag\\
        ={}& \sum_{i =0}^k a^{2(i+1) } \left( \hat C_{k,i} u_{i+1} (s_0/a)   \right) 
        + O(s_0^4; n, k) + O ( |\tilde \lambda | ) + O ( a) \notag\\
        ={}& \sum_{i =0}^k a^{2(i+1) } \left( \hat C_{k,i} \left[ C_i \left(\frac{s_0}a \right)^{2i+2} + O\left( \left(\frac {s_0} a \right)^{2i+1}; n,k \right) \right]   \right) \notag\\
        &+ O(s_0^4; n, k) + O ( |\tilde \lambda | ) + O ( a) \notag\\
        ={}& \sum_{i=0}^k \hat C_{k,i} C_i s_0^{2i+2} + O(s_0^4; n, k) + O ( |\tilde \lambda | ) + O ( a) \notag\\
        ={}& s_0^2 \tilde P ( s_0) + O ( s_0^4; n, k) + O ( | \tilde \lambda| ) + O ( a) \notag
    \end{align}
    where $\tilde P ( s_0) $ is the even polynomial of degree $2k$ with nonzero constant term given by
        $$\tilde P (s_0) = \sum_{i=0}^k \hat C_{k,i} C_i s_0^{2i}$$
    with constants $\hat C_{k,i}$ and $C_i$ as in Lemma \ref{lem interior eigenfunctions} and Proposition \ref{prop generalized kernel}, respectively.

    Similar computations apply to show that
    \begin{multline} \label{proof global eigenfunctions, proof claim 1, eqn 7}
        a^{-1} \partial_{\tilde \lambda} \phi_{int,k}'(s_0/a) 
        = 2 s_0 \tilde P(s_0) + s_0^2 \tilde P_0'(s_0) + O ( s_0^3; n,k) + O ( | \tilde \lambda | ) + O ( a ) \\
        = 2 s_0 \tilde P(s_0) + O ( s_0^3; n,k) + O ( | \tilde \lambda | ) + O ( a).   
    \end{multline}
    Using Lemma \ref{lem matching prelude}, \eqref{proof global eigenfunctions, proof claim 1, eqn 6}, and \eqref{proof global eigenfunctions, proof claim 1, eqn 7},
    it follows that
    \begin{align} \label{proof global eigenfunctions, proof claim 1, eqn 8}
        &\frac{a^{-1} \partial_{\tilde \lambda} \phi_{int}'(s_0/a)}{a^{-1} \phi_{int}'(s_0/a )}
        - \frac{\partial_{\tilde \lambda} \phi_{int}(s_0/a) }{ \phi_{int}(s_0/a)} \\
        ={}& \frac{2 s_0 \tilde P(s_0) + O ( s_0^3; n,k) + O ( | \tilde \lambda |+ a)}{C_{n,k} P'(s_0) + O ( |\tilde \lambda | + a ) }
        - \frac{s_0^2 \tilde P(s_0) + O ( s_0^4;n,k) + O ( |\tilde \lambda |+a )  }{ C_{n,k} P(s_0) + O ( |\tilde \lambda |+a  )  } \notag\\
        ={}&\left( 2C' s_0 + O( s_0^3;n,k) + O ( | \tilde \lambda |+a)  \right) \frac1{C P'(s_0) } \left( 1 + O ( |\tilde \lambda |+a )  \right) \notag\\
        &- \left( C' s_0^2 + O ( s_0^4;n,k) + O ( |\tilde \lambda| + a ) \right) \frac1{C P(s_0)} \left( 1 + O ( |\tilde \lambda|+a )  \right) \notag\\
        ={}&\left( C'' + O ( s_0^2;n,k ) + O ( |\tilde \lambda|+a )  \right) - \left( C''' s_0^2 + O ( s_0^4;n,k) + O ( |\tilde \lambda |+a )  \right) \notag\\
        ={}& C'' + O ( s_0^2;n,k ) + O ( |\tilde \lambda | +a) \notag
    \end{align}
    where $C = C_{n,k}, C', C'', C'''$ are all non-zero constants depending only on $n, k$.

    Inserting \eqref{proof global eigenfunctions, proof claim 1, eqn 3}, \eqref{proof global eigenfunctions, proof claim 1, eqn 4}, and \eqref{proof global eigenfunctions, proof claim 1, eqn 8} into \eqref{proof global eigenfunctions, proof claim 1, eqn 2} then gives that
    \begin{multline}
        \partial_{\tilde \lambda} \Phi 
        = \left( \frac{P'(s_0)}{P(s_0)} + O( |\tilde \lambda |+ a )  \right) 
         \cdot \left[ - \frac{ \tilde u_{\lambda_k}'(s_0)}{P'(s_0)} + O ( s_0^{2-n}; n,k) + O( |\tilde \lambda| + a) \right] \\
        = - \frac{ \tilde u_{\lambda_k}'(s_0) }{P(s_0)} + O ( s_0^{3-n} ; n,k) + O ( |\tilde \lambda | + a ) .
    \end{multline}
    Thus, by the asymptotics of $\tilde u_{\lambda_k}'$ and $P = u_{\lambda_k}$ near $0$ (Proposition \ref{prop plane eigenfunction}), it follows that 
    \begin{equation}
        \partial_{\tilde \lambda} \Phi = C(n,k) s_0^{1-n} + O ( s_0^{2-n} ;n,k) + O ( | \tilde \lambda| + a ) 
    \end{equation}
    for some positive constant $ C(n,k) > 0$ depending only on $n,k$.
    This completes the proof of the claim.
    \end{claimproof}

    Now, consider the domains
    \begin{equation}
        \Omega_{\tilde \lambda^* , a^* } := \{ (\tilde \lambda , a) \in \R^2  \colon - \tilde \lambda^* < \tilde \lambda < \tilde \lambda^* \text{ and } 0< a < a^* \} .
    \end{equation}
    By Claim \ref{proof global eigenfunctions, claim 1},
    for all $0 < s_0 \le s_0^*(n,k) \ll 1$, there exist $\tilde \lambda^* = \tilde \lambda^* (n,k, s_0) \ll 1$ and $a^* = a^* (n,k,s_0) \ll1 $ small such that
    \begin{enumerate}
        \item $\Phi[s_0] (\cdot, \cdot) $ is a well-defined smooth function on $\Omega = \Omega_{\tilde \lambda^*, a^*}$,
        \item there exists $C_{n,k,s_0}>0$ (depending only on $n,k,s_0$) such that 
            $$|\Phi[s_0] (0, a) |\le C_{n,k,s_0} a < C_{n,k,s_0} a^* \quad \text{for all } 0 < a < a^*,$$
            and
        \item there exists $C_{n,k,s_0}'>4$ (depending only on $n,k,s_0$) such that
            $$2 C_{n,k,s_0}' \ge \partial_{\tilde \lambda} \Phi \ge \frac12 C_{n,k,s_0}' \ge 2 \qquad \text{for all } (\tilde \lambda , a) \in \Omega.$$
    \end{enumerate}
    Therefore, the $\partial_{\tilde \lambda} \Phi$ bounds imply for all $0< a < a^*$,
    \begin{gather} \label{proof global eigenfunctions, eqn 2}
        \Phi[s_0]( \tilde \lambda^* /2, a) \ge \tilde \lambda^* + \Phi[s_0] (0, a) \ge \tilde \lambda^* - C_{n,k,s_0} a   \\
        \label{proof global eigenfunctions, eqn 3}
        \text{and similarly } \Phi[s_0] ( - \tilde \lambda^* /2 , a) \le -\tilde \lambda^* + C_{n,k,s_0} a.
    \end{gather}
    Setting $a^{**} = a^{**} (n,k,s_0) := \min \left\{ \frac{ \tilde \lambda^*}{C_{n,k,s_0}}, a^* \right\}$, it follows that
    \begin{equation}
        \Phi[s_0] (-\tilde \lambda^* /2, a)  < 0 <  \Phi[s_0] (\tilde \lambda^* /2, a) 
    \end{equation}
    for all $0 < a < a^{**}$.
    Therefore, the intermediate value theorem and the fact that $\partial_{\tilde \lambda} \Phi \ge 2 > 0$ implies that for all $0 < a < a^{**}$, there exists a unique $ \tilde \lambda_k(a) \in (-\tilde \lambda^*, \tilde \lambda^* )$ such that $\Phi[s_0] ( \tilde \lambda_k (a), a) = 0$.
    As described at the beginning of the proof, the associated $\phi_k = \phi_{k,a,s_0}$ is then a smooth eigenfunction of $H_{s,a}$ with eigenvalue $\lambda_k = 1 - k + \tilde \lambda_k(a)$.

    It remains to prove the estimates \eqref{thm global eigenfunctions, eqn 1}.
    Since $\frac12 C_{n,k,s_0}' \le \partial_{\tilde \lambda } \Phi \le 2 C_{n,k,s_0}$,
    similar logic used to obtain \eqref{proof global eigenfunctions, eqn 2}, \eqref{proof global eigenfunctions, eqn 3} implies 
    \begin{equation} \label{proof global eigenfunctions, eqn 4}
        \frac12 C_{n,k,s_0}' \tilde \lambda - C_{n,k,s_0} a \le \Phi[s_0] (\tilde \lambda , a) \le 2 C_{n,k,s_0}' \tilde \lambda + C_{n,k,s_0} a.
    \end{equation}
    Evaluating these inequalities \eqref{proof global eigenfunctions, eqn 4} at $(\tilde \lambda_k(a), a)$ then shows that
    \begin{equation} \label{proof global eigenfunctions, eqn 4.5}
        | \tilde \lambda_k(a) | \lesssim_{n,k,s_0} a.
    \end{equation}

    \begin{claim} \label{proof global eigenfunctions, claim 2}
        For all $0 < s_0 \le s_0^* (n, k) \ll1$,
        $0 < a \le a^* (n, k, s_0) \ll1$,
        and $ |\tilde \lambda| \le \tilde \lambda^* ( n, k, s_0) \ll 1$,
        \begin{equation} \label{proof global eigenfunctions, claim 2, eqn 1}
            |\partial_{a} \Phi[s_0](\tilde \lambda, a)| = O(1; n,k,s_0) + O( |\tilde \lambda | a^{-1} ; n,k,s_0) .
        \end{equation}
    \end{claim}
    \begin{claimproof}
        By similar logic used to deduce \eqref{proof global eigenfunctions, proof claim 1, eqn 2},
        \begin{gather} \label{proof global eigenfunctions, proof claim 2, eqn 1}
        \begin{aligned}
            \partial_{\tilde \lambda} \Phi 
            ={}& \left( \frac{ \phi_{ext}'(s_0)}{\phi_{ext}(s_0)} + O( |\tilde \lambda| + a  ; n, k, s_0)   \right)  \\
            &\cdot \left[ \frac{ \partial_a \partial_s|_{s=s_0} [ \phi_{int}(s/a) ]}{ \partial_s|_{s=s_0} [ \phi_{int}(s/a) ] } - 
            \frac{ \partial_a [\phi_{int}(s_0/a)]}{ \phi_{int}(s_0/a) } 
            -\frac{ \partial_a \phi_{ext}'(s_0)}{ \phi_{ext}'(s_0) } + 
            \frac{ \partial_a \phi_{ext}(s_0)}{ \phi_{ext}(s_0) } \right]  .
        \end{aligned} \end{gather}
        We compute the asymptotics of terms in this expansion.

        By Lemma \ref{lem matching prelude},
        \begin{equation} \label{proof global eigenfunctions, proof claim 2, eqn 1.5}
            \frac{ \phi_{ext}'(s_0)}{\phi_{ext}(s_0)} + O( |\tilde \lambda| + a  ; n, k, s_0) 
            = O(1;n,k,s_0) 
        \end{equation}

        By Lemma \ref{lem exterior eigenfunctions}, 
        \begin{align}
            \label{proof global eigenfunctions, proof claim 2, eqn 2}
            \partial_a \phi_{ext}(s_0) &= \partial_a \tilde w (s_0) = O(a^{n-1}; n,k,s_0)  \text{ and}\\
            \label{proof global eigenfunctions, proof claim 2, eqn 3}
            \partial_a \phi_{ext}'(s_0) &= \partial_a \tilde w' (s_0) = O(a^{n-1}; n,k,s_0).
        \end{align}
        Thus, using also Lemma \ref{lem matching prelude},
        \begin{equation} \label{proof global eigenfunctions, proof claim 2, eqn 4}
            -\frac{ \partial_a \phi_{ext}'(s_0)}{ \phi_{ext}'(s_0) } + 
            \frac{ \partial_a \phi_{ext}(s_0)}{ \phi_{ext}(s_0) }
            = O(a^{n-1}; n,k,s_0) .
        \end{equation}

        As in Proposition \ref{prop generalized kernel}, write
        \begin{gather}
            \label{proof global eigenfunctions, proof claim 2, eqn 5}
            u_i (s) = \pm C_i s^{2i} + R_i (s) \\
            \label{proof global eigenfunctions, proof claim 2, eqn 6}
            \text{where } R_i (s) = O(s^{2i-1} ; n,i) \text{ and } R_i'(s) = O(s^{2i-1}; n,i) \text{ as } s \to +\infty.
        \end{gather}
        Then, by Lemma \ref{lem interior eigenfunctions},
        \begin{gather} \label{proof global eigenfunctions, proof claim 2, eqn 7}
        \begin{aligned}
            \phi_{int}(s/a) 
            ={}& \sum_{i = 0}^k \hat C_{k,i} a^{2i} u_i (s/a) 
            + \tilde \lambda \sum_{i =0}^k a^{2(i+1)} [ \hat C_{k,i} u_{i+1} (s/a) + v_i(s/a)] + a^2 w_k \\
            ={}& \sum_{i=0}^k \hat C_{k,i}  C_i s^{2i} + \sum_{i=0}^k \hat C_{k,i} a^{2i} R_i (s/a) \\
            & + \tilde \lambda \sum_{i=0}^k  \hat C_{k,i} C_{i+1}  s^{2i + 2} 
            + \tilde \lambda \sum_{i=0}^k a^{2(i+1)} [ \hat C_{k,i} R_{i+1} (s/a) + v_i (s/a) ] \\
            & + a^2 w_k (s/a).
        \end{aligned} \end{gather}
        Thus, 
        \begin{gather}  \label{proof global eigenfunctions, proof claim 2, eqn 8}
        \begin{aligned}
            \partial_a [ \phi_{int}(s/a)] 
            ={}& \sum_{i=0}^k \hat C_{k,i} ( 2i a^{2i-1} R_i (s/a) - s a^{2i-2} R_i'(s/a) ) \\
            &+ \tilde \lambda \sum_{i =0}^k (2i+2) a^{2i+1} [ \hat C_{k,i} R_{i+1} (s/a) + v_i (s/a)] \\
            &+ \tilde \lambda \sum_{i = 0}^k a^{2i+2} \left( \hat C_{k,i} R'_{i+1} (s/a) (-s/a^2) + \partial_a v_i (s/a) + v_i'(s/a)(-s/a^2) \right) \\
            &+ 2a w_k (s/a) + a^2 \partial_a w_k (s/a) + a^2 w_k'(s/a) (-s/a^2).
        \end{aligned} \end{gather}
        Inserting the estimates \eqref{proof global eigenfunctions, proof claim 2, eqn 6} and the estimates for $v_i, w_k$ from Lemma \ref{lem interior eigenfunctions} into \eqref{proof global eigenfunctions, proof claim 2, eqn 8},
        it follows that at $s=s_0$
        \begin{gather} \label{proof global eigenfunctions, proof claim 2, eqn 9}
        \begin{aligned}
            &\partial_a [\phi_{int} (s_0/a) ] \\
            ={}& 
            \sum_{i=0}^k  \left[ a^{2i-1} O \left( \frac{s_0^{2i-1}}{a^{2i-1}} ; n,k \right) + a^{2i} \cdot \frac{s_0}{a^2} \cdot O \left( \frac{s_0^{2i-2}}{a^{2i-2}} ; n,k \right)    \right] \\
            & + \tilde \lambda \sum_{i =0}^k  a^{2i+1} \left[ O \left( \frac{s_0^{2i+1}}{a^{2i+1}} ; n,k \right)  + s_0^2\cdot  O \left( \frac{s_0^{2i+2}}{a^{2i+2}}; n,k \right) \right] \\
            &+\tilde \lambda \sum_{i = 0}^k a^{2i+2} \left[   \frac{s_0}{a^2} \cdot  O \left(  \frac{s_0^{2i}}{a^{2i}} ; n,i \right)  + O \left( s_0 \frac{s_0^{2i+3}}{a^{2i+3}} ; n,k \right)  +  \frac{s_0}{a^2} \cdot O \left( s_0 \frac{s_0^{2i+1}}{a^{2i+1}}  ; n,k \right)  \right]
            \\
            &+ a \cdot O \left( \frac{s_0}{a} ; n,k \right) + a^2 \cdot O \left( s_0 \frac{s_0^2}{a^2} ; n,k \right) + a^2 \cdot \frac{s_0}{a^2} \cdot  O(1;n,k)
            \\
            ={}& O(1; n,k,s_0) + O(|\tilde \lambda| a^{-1} ; n,k,s_0) .
        \end{aligned} \end{gather}
        An analogous computation applies to show that
        \begin{equation} \label{proof global eigenfunctions, proof claim 2, eqn 10}
            \partial_a \partial_s |_{s=s_0} [ \phi_{int} (s/a)]
            = \partial_a \left[ \frac1a \phi_{int}' (s_0/a) \right]
            = O(1; n,k,s_0) + O(|\tilde \lambda| a^{-1} ; n,k,s_0) .
        \end{equation}
        Thus, by \eqref{proof global eigenfunctions, proof claim 2, eqn 9}, \eqref{proof global eigenfunctions, proof claim 2, eqn 10}, and Lemma \ref{lem matching prelude},
        \begin{equation} \label{proof global eigenfunctions, proof claim 2, eqn 11}
            \frac{ \partial_a \partial_s|_{s=s_0} [ \phi_{int}(s/a) ]}{ \partial_s|_{s=s_0} [ \phi_{int}(s/a) ] } - 
            \frac{ \partial_a [\phi_{int}(s_0/a)]}{ \phi_{int}(s_0/a) } 
            = O(1; n,k,s_0) + O(|\tilde \lambda| a^{-1} ; n,k,s_0) .
        \end{equation}

        Inserting estimates \eqref{proof global eigenfunctions, proof claim 2, eqn 1.5}, 
        \eqref{proof global eigenfunctions, proof claim 2, eqn 4}, and \eqref{proof global eigenfunctions, proof claim 2, eqn 11} into \eqref{proof global eigenfunctions, proof claim 2, eqn 1} completes the proof of the claim.
    \end{claimproof}

    Now, differentiating the identity $\Phi[s_0] (\tilde \lambda_k(a), a) \equiv 0$ implies
    \begin{align*}
        &\partial_a \tilde \lambda_k (a) \\
        ={}& - \frac{ \partial_a \Phi[s_0](\tilde \lambda_k(a), a)}{ \partial_{\tilde \lambda} \Phi[s_0] (\tilde \lambda_k (a), a) } \\
        ={}& \frac{ O(1; n, k, s_0) + O ( |\tilde \lambda_k(a) | a^{-1} ; n,k,s_0) }{C(n,k)s_0^{1-n} + O (s_0^{2-n}; n,k) + O ( |\tilde \lambda_k(a)| + a ; n,k, s_0) }
        && \text{(by Claims \ref{proof global eigenfunctions, claim 1}, \ref{proof global eigenfunctions, claim 2})} \\
        ={}& O(1; n,k,s_0) 
        && \text{(by \eqref{proof global eigenfunctions, eqn 4.5})}.
    \end{align*}
    This proves the remaining estimate in \eqref{thm global eigenfunctions, eqn 1} and thereby completes the proof of the theorem for any $k \in \mathbb N$ with $0 \le k \le K$.
    The theorem then follows by taking minima or maxima over $0 \le k \le K$ as needed.
\end{proof}

\subsection{Estimates for the Matched Eigenfunctions} \label{Subsection Eigenfunction Estimates}

The advantage of Theorem \ref{thm global eigenfunctions}'s gluing construction is that it implies precise estimates for the global eigenfunctions.
In this subsection, we collect and record these pointwise and integral estimates.

\subsubsection{Pointwise Estimates}
\begin{lem} \label{lem tilde phi_k,a pointwise ests}
    Let $n\ge 3$ and $K \in \mathbb N$.
    For all $0 < s_0 \le s_0^* (n,K) \ll 1$,
    $\delta \in (0,1)$,
    $0 < a \le a^* (n,K, s_0 , \delta) \ll 1 $,
    and $k \in \mathbb N $ with $0 \le k \le K$,
    the following holds:
    
    Let $\phi_{k,a} = \phi_{k,a,s_0}(s)$ denote the $H_a$-eigenfunction constructed in Theorem \ref{thm global eigenfunctions}.
    Define $\tilde \phi_{k,a}$ such that
    \begin{equation} \label{eqn defn tilde phi}
        \phi_{k,a}(s) = \sum_{i=0}^k \hat C_{k,i} a^{2i} u_i(s/a) + \tilde \phi_{k,a}(s)
    \end{equation}
    where $\hat C_{k,i}$ and $u_i$ are as in Lemma \ref{lem interior eigenfunctions}.

    Then
    \begin{align}
        \label{lem tilde phi_k,a pointwise ests, eqn 1}
        \tilde \phi_{k,a} (s) &= O ( a; n, k,s_0, \delta) \left( 1 + |s| \right)^{2k+ \delta} && \forall s \in \R, \text{ and} \\
        \label{lem tilde phi_k,a pointwise ests, eqn 1.5}
        \partial_s \tilde \phi_{k,a} (s) &= O (a; n, k, s_0, \delta) \left( 1 + |s| \right)^{2k-1+\delta} && \forall s \in \R.
    \end{align}

    In particular,
    \begin{align}
        \label{lem tilde phi_k,a pointwise ests, eqn 2}
        a^{-2} \beta_a(s) - \phi_{0,a} (s) &=   O ( a; n,s_0, \delta) \left( 1 + |s| \right)^{ \delta} && \forall s \in \R, \text{ and} \\
        \label{lem tilde phi_k,a pointwise ests, eqn 2.5}
        \partial_s \left( a^{-2} \beta_a -  \phi_{0,a}  \right) &= O (a; n, s_0, \delta) \left( 1 + |s| \right)^{-1+\delta} && \forall s \in \R.
    \end{align}
\end{lem}
\begin{proof}
    For $|s| \le s_0$,
    \begin{multline}
        \tilde \phi_{k,a}(s)
        = \phi_{int, k}(s/a) - \sum_{i=0}^k \hat C_{k,i} a^{2i} u_i(s/a)  
        = \tilde \lambda_k \sum_{i=0}^k a^{2(i+1)} ( \hat C_{k,i} u_{i+1} + v_i ) + a^2 w_k \\
        = O (a; n, k, s_0)
    \end{multline}
    by the estimates of Lemma \ref{lem interior eigenfunctions} and Theorem \ref{thm global eigenfunctions}.
    Similar logic applies to show 
    \begin{equation}
        \partial_s \tilde \phi_{k,a} = O ( a; n, k, s_0) \qquad \forall |s| \le s_0.
    \end{equation}

    Recall from Theorem \ref{thm global eigenfunctions} that $|\tilde \lambda_k(a) | \lesssim_{n,k,s_0} a .$
    It follows that, for $s \ge s_0$,
    \begin{align*}
        &\tilde \phi_{k,a} (s) \\
        ={}& \phi_{k,a}(s) - \sum_{i=0}^k \hat C_{k,i} a^{2i} u_i (s/a) \\
        ={}& \frac{\phi_{int} (s_0/a) }{\phi_{ext}(s_0)} \phi_{ext}(s) 
        - \check C(n,k) P_k(s) + O ( a; n, k, s_0) s^{2k-1} 
        && ( \text{Lemma \ref{lem rewriting Laguerre poly}} ) \\
        ={}& \frac{ \check C P_k(s_0) + O ( a; n, k, s_0) }{P_k(s_0) + O (a; n, k, s_0) } \phi_{ext}(s) 
        - \check C P_k(s) + O ( a; n, k, s_0) s^{2k-1} 
        && (\text{Lemma \ref{lem matching prelude}} )\\
        ={}& \check C \left( \phi_{ext}(s) - P_k(s) \right)
        + O(a; n,k,s_0) \phi_{ext} + O(a;n,k,s_0) s^{2k-1} \\
        ={}& O(a; n, k, s_0, \delta) s^{2k+\delta} 
        && ( \text{Lemma \ref{lem exterior eigenfunctions}} ).
    \end{align*}
    Similar logic applies to $s$-derivatives to show that 
    \begin{equation*}
        \partial_s \tilde \phi_{k,a} (s) = O(a;n, k, s_0, \delta) s^{2k-1+\delta} \qquad \forall s \ge s_0.
    \end{equation*}
    This proves \eqref{lem tilde phi_k,a pointwise ests, eqn 1} and \eqref{lem tilde phi_k,a pointwise ests, eqn 1.5}.

    \eqref{lem tilde phi_k,a pointwise ests, eqn 2} and \eqref{lem tilde phi_k,a pointwise ests, eqn 2.5} then follow from the fact that 
    \begin{equation}
        \hat C_{0, 0} u_0(s/a) = \beta_1(s/a) = a^{-2} \beta_a(s)
    \end{equation}
    (see Proposition \ref{prop generalized kernel} and Lemma \ref{lem beta properties}).
\end{proof}

Before proceeding, we note the following consequence of Lemma \ref{lem tilde phi_k,a pointwise ests}.
\begin{lem} \label{lem phi_k,a - phi_0,a pointwise ests}
    Let $n\ge 3$ and $K \in \mathbb N$.
    For all $0 < s_0 \le s_0^* (n,K) \ll 1$,
    $\delta \in (0,1)$,
    $0 < a \le a^* (n,K, s_0 , \delta) \ll 1 $,
    and $k \in \mathbb N $ with $0 \le k \le K$,
    the following holds:
    
    Let $\phi_{k,a} = \phi_{k,a,s_0}(s)$ denote the $H_a$-eigenfunction constructed in Theorem \ref{thm global eigenfunctions} and $\tilde \phi_{k,a}$ be as in \eqref{eqn defn tilde phi}.

    Then there exists $C = C(n,k)>0$ and $C' = C'(n,k,s_0, \delta) > 0$ such that 
    \begin{align}
        \label{lem phi_k,a - phi_0,a pointwise ests, eqn 1}
        | \phi_{k,a} (s) - \phi_{0,a}(s) | 
        &\le C \left( a^2 + |s|^2 + |s|^{2k} \right) + C' a \left( 1 + |s| \right)^{2k+\delta} && \forall s \in \R, \text{ and} \\
        \label{lem phi_k,a - phi_0,a pointwise ests, eqn 2}
        \left| \partial_s\left( \phi_{k,a} (s) - \phi_{0,a}(s) \right) \right| 
        &\le C \left( a + |s| + |s|^{2k-1} \right) + C' a \left( 1 + |s| \right)^{2k-1+\delta} && \forall s \in \R.
    \end{align}
\end{lem}
\begin{proof}
    Observe the cancellation
    \begin{gather} \label{proof phi_k,a - phi_0,a pointwise ests, eqn 1} \begin{aligned}
        \phi_{k,a}(s) - \phi_{0,a}(s)
        ={}& \sum_{i=0}^k \hat C_{k,i} a^{2i} u_i(s/a) + \tilde \phi_{k,a}(s) 
        - \hat C_{0,0} u_0 (s/a) - \tilde \phi_{0,a}(s) \\
        ={}& \sum_{i=1}^k \hat C_{k,i} a^{2i} u_i(s/a) + \tilde \phi_{k,a}(s) - \tilde \phi_{0,a}(s)
    \end{aligned} \end{gather}
    since $\hat C_{j,0} =1$ for all $j$ by Lemma \ref{lem interior eigenfunctions}.
    By Lemma \ref{lem tilde phi_k,a pointwise ests},
    \begin{equation} \label{proof phi_k,a - phi_0,a pointwise ests, eqn 2}
        | \tilde \phi_{k,a} (s) - \tilde \phi_{0,a}(s) |
        \le O(a; n, k, s_0, \delta) \left( 1 + |s| \right)^{2k+\delta} 
        \qquad \forall s \in \R.
    \end{equation}
    By Proposition \ref{prop generalized kernel},
    \begin{equation} \label{proof phi_k,a - phi_0,a pointwise ests, eqn 3}
        \left| \sum_{i=1}^k \hat C_{k,a} a^{2i} u_i(s/a) \right| 
        \le C \sum_{i=1}^k a^{2i} \langle s/a \rangle^{2i}  
        \le C \left( a^2 + s^2 + s^{2k} \right) \qquad \forall s \in \R
    \end{equation}
    where $C = C(n,k) > 0$ denotes a constant that depends only on $n,k$ and may change from line to line.
    Using estimates \eqref{proof phi_k,a - phi_0,a pointwise ests, eqn 2} and \eqref{proof phi_k,a - phi_0,a pointwise ests, eqn 3} in \eqref{proof phi_k,a - phi_0,a pointwise ests, eqn 1} thereby proves \eqref{lem phi_k,a - phi_0,a pointwise ests, eqn 1}.
    Taking derivatives of \eqref{proof phi_k,a - phi_0,a pointwise ests, eqn 1} and applying analogous estimates proves \eqref{lem phi_k,a - phi_0,a pointwise ests, eqn 2}.
\end{proof}

\begin{lem}[$\partial_a \phi_{k,a}$ Estimates] \label{lem partial_a phi_a ests}
    Let $n\ge 3$ and $K \in \mathbb N$.
    For all $0 < s_0 \le s_0^* (n,K) \ll 1$,
    $\delta \in (0,1)$,
    $0 < a \le a^* (n,K, s_0 , \delta) \ll 1 $,
    and $k \in \mathbb N $ with $0 \le k \le K$,
    the following holds:

    If $\phi_{k,a} = \phi_{k,a,s_0}(s)$ denotes the $H_a$-eigenfunction constructed in Theorem \ref{thm global eigenfunctions},
    then there exists $C = C(n,k,s_0)$ and $C_\delta = C(n,k,s_0, \delta)$ such that
    \begin{align}
        \label{lem partial_a phi_a ests, eqn 1}
        | \partial_a \phi_{k,a} (s) - \partial_a \phi_{0,a} (s) | &\le C + C_\delta (1+|s|)^{2k+\delta} && ( \forall s \in \R),  \\
        \label{lem partial_a phi_a ests, eqn 1.5}
        | \partial_s \partial_a \phi_{k,a} (s) - \partial_s \partial_a \phi_{0,a} (s) | &\le C + C_\delta (1+|s|)^{2k-1+\delta} && ( \forall s \in \R),   \\
        \label{lem partial_a phi_a ests, eqn 2} 
        |\partial_a \phi_{k,a}(s) | &\le \frac{Cs a^{-2}}{(1+(s/a))^{n-1}}+ C + C_\delta (1+|s|)^{2k+\delta}
        && ( \forall s \in \R),   \text{ and}\\
        \label{lem partial_a phi_a ests, eqn 2.5}
        |\partial_s \partial_a \phi_{k,a}(s) | &\le \frac{C a^{-2}}{(1+(s/a))^{n-1}}+ C + C_\delta(1 +  |s|)^{2k-1+\delta}
        && ( \forall s \in \R) .
    \end{align}

    In particular,
    \begin{equation} \label{lem partial_a phi_a ests, eqn 3}
        \| \partial_a \phi_{k,a} - \partial_a \phi_{0,a} \|_{H^1_a} \le C(n,k, s_0) .
    \end{equation}
\end{lem}
\begin{proof}
    Fix $0 \le k \le K$.
    Throughout the proof, $C = C(n,k, s_0)$ denotes a constant that depends on $n, k, s_0$ only and which may change from line to line.
    Similarly, $C_\delta = C_\delta(n,k,s_0, \delta)$ denotes a constant that depends on $n,k,s_0, \delta$ only and which may change from line to line.

    (Estimates for $|s| \ge s_0$)
    
    For $s \ge s_0$, 
    \begin{align} \label{proof partial_a phi_a ests, eqn 1}
        \partial_a \phi_{k,a}
        ={}& \partial_a \left( \frac {\phi_{int} (s_0/a) }{\phi_{ext}(s_0)} \phi_{ext}(s) \right) \\
        ={}& \frac{\phi_{int}(s_0/a)}{\phi_{ext}(s_0)} \partial_a \phi_{ext}(s)
        + \frac{\phi_{int}(s_0/a)}{\phi_{ext}(s_0)} \partial_{\tilde \lambda} \phi_{ext}(s) \frac{d \tilde \lambda}{da} \notag\\
        &+ \phi_{ext}(s) \frac{\partial_a [\phi_{int}(s_0/a)] \phi_{ext}(s_0) - \phi_{int}(s_0/a) \partial_a \phi_{ext}(s_0) }{\phi_{ext}^2(s_0)} \notag\\
        &+ \phi_{ext}(s) \frac{\partial_{\tilde \lambda} \phi_{int}(s_0/a) \phi_{ext}(s_0) - \phi_{int}(s_0/a) \partial_{\tilde \lambda} \phi_{ext}(s_0) }{\phi_{ext}^2(s_0)}  \frac {d \tilde \lambda}{da} \notag\\
        :={}& (I) + (II) + (III) + (IV).\notag
    \end{align}

    \begin{align} \label{proof partial_a phi_a ests, eqn 2}
        |(I)|
        ={}& \left| \frac{\phi_{int}(s_0/a)}{\phi_{ext}(s_0)} \partial_a \phi_{ext}(s) \right| \\
        ={}& \left| \frac{ C(n,k) P_k(s_0) + O ( a; n, k, s_0) }{P_k(s_0) + O ( a; n, k, s_0) }  \partial_a \tilde w \right| 
        && (\text{Lemmas \ref{lem exterior eigenfunctions}, \ref{lem matching prelude}} ) \notag\\
        \le{}& C_\delta a^{n-1} s^{2k+\delta} 
        && (\text{Lemma \ref{lem exterior eigenfunctions}}). \notag
    \end{align}

    \begin{align} \label{proof partial_a phi_a ests, eqn 3}
        |(II)| 
        ={}& \left| \frac{\phi_{int}(s_0/a)}{\phi_{ext}(s_0)} \partial_{\tilde \lambda} \phi_{ext}(s) \frac{d \tilde \lambda}{da} \right| \\
        \le{}& C \left| \frac{C(n,k) P_k(s_0) + O(a; n,k,s_0) }{P_k(s_0) + O(a; n,k,s_0) }  \partial_{\tilde \lambda} \phi_{ext} \right| 
        && ( \text{Lemma \ref{lem matching prelude}, Theorem \ref{thm global eigenfunctions}}) \notag\\
        \le{}& C | \tilde u_{\lambda_k} + \tilde v + \tilde \lambda  \partial_{\tilde \lambda} \tilde v  + \partial_{\tilde \lambda} \tilde w | && ( \text{Lemma \ref{lem exterior eigenfunctions}}) \notag\\
        \le{}& C_\delta s^{2k+\delta} && ( \text{Lemma \ref{lem exterior eigenfunctions}} ).\notag
    \end{align}

    \begin{align} \label{proof partial_a phi_a ests, eqn 4} 
        |(III)|  
        ={}& \left| \phi_{ext}(s) \frac{\partial_a [\phi_{int}(s_0/a)] \phi_{ext}(s_0) - \phi_{int}(s_0/a) \partial_a \phi_{ext}(s_0) }{\phi_{ext}^2(s_0)} \right|  \\
        ={}& \left| \phi_{ext}(s) \frac{ O(1)(P_k(s_0) + O(a) ) - ( C(n,k) P_k(s_0) + O(a) ) O(a^{n-1} )  }{(P_k (s_0) + O(a) )^2}\right| \notag\\
        &   ( \text{by Lemma \ref{lem matching prelude}, \eqref{proof global eigenfunctions, proof claim 2, eqn 2}, \eqref{proof global eigenfunctions, proof claim 2, eqn 9}} ) \notag\\
        \le{}& C|\phi_{ext}(s)| \notag\\
        \le{}& C s^{2k} + C_\delta a s^{2k+\delta} .\notag
    \end{align}
    where the last inequality follows from the estimates in Lemma \ref{lem exterior eigenfunctions} and we have written $O(\cdot) = O(\cdot ; n,k,s_0)$ above for brevity.

    Similarly, writing $O(\cdot) = O(\cdot; n,k,s_0)$ for brevity,
    \begin{align} \label{proof partial_a phi_a ests, eqn 5}
        |(IV)|
        ={}& \left| \phi_{ext}(s) \frac{\partial_{\tilde \lambda} \phi_{int}(s_0/a) \phi_{ext}(s_0) - \phi_{int}(s_0/a) \partial_{\tilde \lambda} \phi_{ext}(s_0) }{\phi_{ext}^2(s_0)}  \frac {d \tilde \lambda}{da} \right| \\
        \le{}& C| \phi_{ext}(s) |\left|\frac{ O(1) ( P_k(s_0) + O(a) )- ( C(n,k) P_k(s_0) + O(a)) O(1) }{(P_k(s_0) +O(a))^2} \right| \notag\\
        & (\text{by Lemma \ref{lem matching prelude}, \eqref{proof global eigenfunctions, proof claim 1, eqn 3.6}, \eqref{proof global eigenfunctions, proof claim 1, eqn 6}}) \notag\\
        \le{}& C s^{2k} + C_\delta s^{2k+\delta}\notag
    \end{align}
    where the last inequality follows from the estimates in Lemma \ref{lem exterior eigenfunctions}.

    Combining \eqref{proof partial_a phi_a ests, eqn 1}, \eqref{proof partial_a phi_a ests, eqn 2}, \eqref{proof partial_a phi_a ests, eqn 3}, \eqref{proof partial_a phi_a ests, eqn 4}, and \eqref{proof partial_a phi_a ests, eqn 5} gives that 
    \begin{equation} \label{proof partial_a phi_a ests, eqn 6}
        |\partial_a \phi_{k,a} (s) | 
        \le C_\delta s^{2k + \delta} = C(n,k,s_0, \delta) s^{2k+\delta} 
        \qquad \forall s \ge s_0.
    \end{equation}

    For $s \ge s_0$, $\partial_s \partial_a \phi_{k,a}$ is the same expression as \eqref{proof partial_a phi_a ests, eqn 1} but with $\phi_{ext}(s)$ replaced by $\partial_s \phi_{ext}(s)$ throughout.
    Similar estimates then give that
    \begin{equation} \label{proof partial_a phi_a ests, eqn 6.1}
        |\partial_s \partial_a \phi_{k,a} (s) | 
        \le C_\delta s^{2k-1 + \delta} = C(n,k,s_0, \delta) s^{2k-1+\delta} 
        \qquad \forall s \ge s_0.
    \end{equation}

    In particular,
    \begin{align} 
        \label{proof partial_a phi_a ests, eqn 6.5}
        |\partial_a \phi_{k,a}(s) - \partial_a \phi_{0,a}(s)| 
        \le | \partial_a \phi_{k,a} (s) | + | \partial_a \phi_{0,a} (s)| 
        \le C_\delta (s^{2k+\delta} + s^\delta ) 
        \le C_\delta s^{2k+ \delta} \\
        \label{proof partial_a phi_a ests, eqn 6.6}
        \text{and } |\partial_s \partial_a \phi_{k,a} (s) - \partial_s \partial_a \phi_{0,a} (s) | \le C_\delta s^{2k-1+\delta}
    \end{align}
    for all $s \ge s_0$.
    
    (Estimates for $|s| \le s_0$)

    We first note the following cancellation:
    For $0 \le s \le s_0$,
    \begin{align} \label{proof partial_a phi_a ests, eqn 7}
        &\phi_{k,a}(s) - \phi_{0,a}(s) \\
        ={}& \phi_{int,k}(s/a) - \phi_{int,0}(s/a) \notag\\
        ={}& \left( \sum_{i=0}^k \hat C_{k,i} a^{2i} u_i \right) - \hat C_{0,0} a^0 u_0 
         + \tilde \lambda \left( \sum_{i=0}^k a^{2i+2} (\hat C_{k,i} u_{i+1} + v_{k,i} )\right) - \tilde \lambda a^2 ( \hat C_{0,0} u_1 + v_{0,0} ) \notag\\
        &+a^2 w_k - a^2 w_0 \notag\\
        ={}& \sum_{i=1}^k \hat C_{k,i} a^{2i} u_i(s/a) + \tilde \lambda \sum_{i=1}^k \hat C_{k,i}a^{2i+2}   u_{i+1}(s/a) 
        + \tilde \lambda \left( \sum_{i=0}^k a^{2i+2} v_{k,i}(s/a) \right) \notag\\
        &- \tilde \lambda a^2 v_{0,0} (s/a)
        + a^2 w_k(s/a) - a^2 w_0(s/a) \notag
    \end{align}
    since $\hat C_{k,0} = \hat C_{0,0} = 1$ for all $k$ (Lemma \ref{lem interior eigenfunctions}).

    By Proposition \ref{prop generalized kernel}, we can write
    \begin{align}
        \label{proof partial_a phi_a ests, eqn 8}
        u_i(s) = \pm C_i s^{2i} + R_i(s) \\ 
        \label{proof partial_a phi_a ests, eqn 9}
        \text{where } | \partial_s^j R_i(s) | \le C(n,i,j) \langle s \rangle^{2i-1-j} \qquad \forall s \in \R, \, j \in \{0,1,2\}.
    \end{align}
    Using \eqref{proof partial_a phi_a ests, eqn 7}, it follows that for $0 \le s \le s_0$ (cf. \eqref{proof global eigenfunctions, proof claim 2, eqn 7}, \eqref{proof global eigenfunctions, proof claim 2, eqn 8})
    \begin{align} \label{proof partial_a phi_a ests, eqn 10}
        &\partial_a [ \phi_{k,a}(s) - \phi_{ 0,a}(s) ]  
        \\
        ={}& \partial_a \phi_{int,k} (s/a) - \frac s{a^2}  \phi_{int,k}'(s/a) + \frac{ d \tilde \lambda}{da} \partial_{\tilde \lambda} \phi_{int,k} (s/a ) \notag\\
        &-  \partial_a \phi_{int,0} (s/a) + \frac s{a^2}  \phi_{int,0}'(s/a) - \frac{ d \tilde \lambda}{da} \partial_{\tilde \lambda} \phi_{int,0} (s/a )
        \notag\\
        ={}& \sum_{i=1}^k \left[ \hat C_{k,i}  2i a^{2i-1} R_i (s/a) - \hat C_{k,i} s a^{2i-2} R_i'(s/a)  \right] \notag\\
        &+ \tilde \lambda \sum_{i=1}^k \left[ \hat C_{k,i} (2i+2) a^{2i+1} R_{i+1}(s/a) - \hat C_{k,i}s a^{2i} R_{i+1}'(s/a)\right] \notag\\
        &+ \frac{d \tilde \lambda}{da} \sum_{i=1}^k \hat C_{k,i} a^{2i+2} u_{i+1}(s/a)\notag\\
        &+  \tilde \lambda \sum_{i=0}^k \left[ (2i+2) a^{2i+1} v_{k,i} (s/a) + a^{2i+2} (\partial_a v_{k,i}) (s/a) - s a^{2i} v_{k,i}'(s/a)\right] \notag\\
        &+  \sum_{i=0}^k \left[ \frac{d \tilde \lambda}{da} a^{2i+2} v_{k,i}(s/a) + \tilde \lambda a^{2i+2} \frac{d \tilde \lambda}{da} (\partial_{\tilde \lambda} v_{k,i} )(s/a) \right] \notag\\
        &- \tilde \lambda 2 a v_{0,0} (s/a) - \tilde \lambda a^2 (\partial_a v_{0,0} )(s/a) + \tilde \lambda s v_{0,0}'(s/a) \notag\\
        &- \frac{d \tilde \lambda}{da} a^2 v_{0,0} (s/a) - \frac{d \tilde \lambda }{da } \tilde \lambda a^2 (\partial_{\tilde \lambda} v_{0,0}) (s/a)\notag\\
        &+ 2a w_k (s/a) + a^2 (\partial_a w_k) (s/a) - s w_k' (s/a) + a^2 \frac{d \tilde \lambda}{da} (\partial_{\tilde \lambda } w_k )(s/a) \notag\\
        &- 2a w_0 (s/a) - a^2 (\partial_a w_0) (s/a) + s w_0' (s/a) - a^2 \frac{d \tilde \lambda}{da} (\partial_{\tilde \lambda } w_0 )(s/a) \notag\\
        :={}& (I) + (II) + (III) + (IV) +(V) + (VI) +(VII) + (VIII) + (IX).\notag
    \end{align}

    \begin{align} \label{proof partial_a phi_a ests, eqn 11}
        |(I)| 
        ={}& \left| \sum_{i=1}^k \left[ \hat C_{k,i}  2i a^{2i-1} R_i (s/a) - \hat C_{k,i} s a^{2i-2} R_i'(s/a)  \right] \right| \\
        \le{}& \sum_{i=1}^k C(n,k,i)  \left( a^{2i-1} \langle s/a \rangle^{2i-1} + s a^{2i-2} \langle s/a \rangle^{2i - 2} \right)
        && ( \text{by \eqref{proof partial_a phi_a ests, eqn 9}} )\notag\\
        \le{}& C = C(n,k,s_0) 
        && ( |s| \le s_0) \notag
    \end{align}

    \begin{align} \label{proof partial_a phi_a ests, eqn 12}
        &|(II)|\\
        ={}& \left| \tilde \lambda \sum_{i=1}^k \left[ \hat C_{k,i} (2i+2) a^{2i+1} R_{i+1}(s/a) - \hat C_{k,i}s a^{2i} R_{i+1}'(s/a)\right] \right| \notag\\
        \le{}& C a \sum_{i=1}^k \left( a^{2i+1} \langle s/a \rangle^{2i+1} + sa^{2i} \langle s/a\rangle^{2i} \right) 
        && ( \text{Theorem \ref{thm global eigenfunctions}, \eqref{proof partial_a phi_a ests, eqn 9}} ) \notag\\
        \le{}& C a  && ( |s| \le s_0) .\notag
    \end{align}

    The remaining terms can be estimated similarly using also the estimates for $v_{k,i}$ and $w_k$ from Lemma \ref{lem interior eigenfunctions}. 
    We briefly sketch the remaining estimates.
    \begin{equation} \label{proof partial_a phi_a ests, eqn 13}
        |(III)|
        = \left| \frac{d \tilde \lambda}{da} \sum_{i=1}^k \hat C_{k,i} a^{2i+2} u_{i+1}(s/a) \right|  
        \le C \sum_{i=1}^k a^{2i+2} \langle s/a \rangle^{2i+2} 
        \le C ,
    \end{equation}
    \begin{multline} \label{proof partial_a phi_a ests, eqn 14}
        |(IV)| 
        =  \left| \tilde \lambda \sum_{i=0}^k \left[ (2i+2) a^{2i+1} v_{k,i} (s/a) + a^{2i+2} (\partial_a v_{k,i}) (s/a) - s a^{2i} v_{k,i}'(s/a)\right] \right| \\ 
        \le C a \sum_{i=0}^k \left[ a^{2i+1} \langle s/a \rangle^{2i+2} + a^{2i+2} \langle s/a \rangle^{2i+3} + sa^{2i} \langle s/a \rangle^{2i+1} \right] 
        \le C,
    \end{multline}
    \begin{multline} \label{proof partial_a phi_a ests, eqn 15}
        |(V)| 
        = \left| \sum_{i=0}^k \left[ \frac{d \tilde \lambda}{da} a^{2i+2} v_{k,i}(s/a) + \tilde \lambda a^{2i+2} \frac{d \tilde \lambda}{da} (\partial_{\tilde \lambda} v_{k,i} )(s/a) \right] \right| \\
        \le C \sum_{i=0}^k \left[a^{2i+2} \langle s/a \rangle^{2i+2} + a^{2i+3} \langle s/a \rangle^{2i+2} \right]
        \le C ,
    \end{multline}
    \begin{equation} \label{proof partial_a phi_a ests, eqn 16}
        |(VI)|
        = \left |- \tilde \lambda 2 a v_{0,0} (s/a) - \tilde \lambda a^2 (\partial_a v_{0,0} )(s/a) + \tilde \lambda s v_{0,0}'(s/a) \right| 
        \le C ,
    \end{equation}
    \begin{equation} \label{proof partial_a phi_a ests, eqn 17}
        |(VII)|
        = \left| - \frac{d \tilde \lambda}{da} a^2 v_{0,0} (s/a) - \frac{d \tilde \lambda }{da } \tilde \lambda a^2 (\partial_{\tilde \lambda} v_{0,0}) (s/a)\right| 
        \le C,
    \end{equation}
    \begin{multline} \label{proof partial_a phi_a ests, eqn 18}
        |(VIII)|
        = \left| 2a w_k (s/a) + a^2 (\partial_a w_k) (s/a) - s w_k' (s/a) + a^2 \frac{d \tilde \lambda}{da} (\partial_{\tilde \lambda } w_k )(s/a) \right|  \\
        \le C \left( a \langle s/a \rangle + a^2 \langle s/a \rangle^2 + s + a^2 \cdot a^2 \langle s/a \rangle^3 \right) 
        \le C ,
    \end{multline}
    and
    \begin{equation} \label{proof partial_a phi_a ests, eqn 19}
        |(IX)|
        = \left| - 2a w_0 (s/a) - a^2 (\partial_a w_0) (s/a) + s w_0' (s/a) - a^2 \frac{d \tilde \lambda}{da} (\partial_{\tilde \lambda } w_0 )(s/a) \right| 
        \le C.
    \end{equation}

    Inserting estimates \eqref{proof partial_a phi_a ests, eqn 11}--\eqref{proof partial_a phi_a ests, eqn 19} into \eqref{proof partial_a phi_a ests, eqn 10} gives that 
    \begin{equation} \label{proof partial_a phi_a ests, eqn 20}
        |\partial_a ( \phi_{k,a} (s) - \phi_{0,a}(s)) | \le C = C(n,k,s_0) \qquad \forall |s| \le s_0.
    \end{equation}

    Additionally, for $0 \le s \le s_0$,
    \begin{multline} \label{proof partial_a phi_a ests, eqn 21}
        \phi_{0,a} (s)  = \phi_{int, 0} (s/a)
        =  u_0(s/a) + \tilde \lambda a^2 u_1(s/a) + \tilde \lambda a^2 v_{0,0}(s/a) + a^2 w_0(s/a) \\
        = C_0 + R_0(s/a) + \tilde \lambda C_1 s^2 + \tilde \lambda a^2 R_1(s/a)
        + \tilde \lambda a^2 v_{0,0}(s/a) + a^2 w_0(s/a) 
    \end{multline}
    since $\hat C_{0,0} =1$ (Lemma \ref{lem interior eigenfunctions}).
    Thus, for $0 \le s \le s_0$,
    \begin{align} \label{proof partial_a phi_a ests, eqn 22}
        \partial_a [\phi_{0,a} (s)]
        ={}& - sa^{-2} R_0'(s/a) + \frac {d \tilde \lambda }{da} C_1 s^2  \\
        &+ \frac{d\tilde \lambda}{da} a^2 R_1(s/a) + 2a \tilde \lambda  R_1(s/a) - \tilde \lambda s R_1'(s/a) \notag\\
        &+  \frac{d\tilde \lambda}{da} a^2 v_{0,0}(s/a) + 2a \tilde \lambda  v_{0,0}(s/a) - \tilde \lambda s v_{0,0}'(s/a) \notag\\
        &+ 2 a w_0(s/a) - s w_0'(s/a).\notag
    \end{align}
    Similar estimates as in \eqref{proof partial_a phi_a ests, eqn 11}--\eqref{proof partial_a phi_a ests, eqn 19} above then show that for $0 \le s \le s_0$,
    \begin{equation} \label{proof partial_a phi_a ests, eqn 23}
        |\partial_a [ \phi_{0,a}(s)] | 
        \le | sa^{-2} R_0'(s/a) | + C
        \le C sa^{-2} \langle s/a \rangle^{1-n} + C 
        \le \frac{Cs a^{-2}}{(1+(s/a))^{n-1}} + C
    \end{equation}
    where the penultimate inequality uses \eqref{beta derivatives asymptotic expansion eqn}.

    Taking $\partial_s$ of \eqref{proof partial_a phi_a ests, eqn 10} yields, for $0 \le s \le s_0$,
    \begin{align} \label{proof partial_a phi_a ests, eqn 24}
        &\partial_s \partial_a [ \phi_{k,a}(s) - \phi_{ 0,a}(s) ]  
        \\
        ={}& \sum_{i=1}^k \left[ \hat C_{k,i}  2i a^{2i-2} R_i' (s/a) - \hat C_{k,i}  a^{2i-2} R_i'(s/a)  - \hat C_{k,i} s a^{2i-3} R_i''(s/a)\right] \notag\\
        &+ \tilde \lambda \sum_{i=1}^k \left[ \hat C_{k,i} (2i+2) a^{2i} R_{i+1}'(s/a) - \hat C_{k,i} a^{2i} R_{i+1}'(s/a) - \hat C_{k,i}s a^{2i-1} R_{i+1}''(s/a) \right] \notag\\
        &+ \frac{d \tilde \lambda}{da} \sum_{i=1}^k \hat C_{k,i} a^{2i+1} u_{i+1}'(s/a)\notag\\
        &+  \tilde \lambda \sum_{i=0}^k \left[ (2i+2) a^{2i} v_{k,i}' (s/a) + a^{2i+1} (\partial_a v_{k,i}') (s/a) -  a^{2i} v_{k,i}'(s/a) - s a^{2i-1} v_{k,i}''(s/a) \right] \notag\\
        &+  \frac{d \tilde \lambda}{da} \sum_{i=0}^k \left[  a^{2i+1} v_{k,i}'(s/a) + \tilde \lambda a^{2i+1}  (\partial_{\tilde \lambda} v_{k,i}' )(s/a) \right] \notag\\
        &- \tilde \lambda 2  v_{0,0}' (s/a) - \tilde \lambda a (\partial_a v_{0,0}' )(s/a) + \tilde \lambda  v_{0,0}'(s/a) + \tilde \lambda \frac sa v_{0,0}''(s/a) \notag\\
        &- \frac{d \tilde \lambda}{da} a v_{0,0}' (s/a) - \frac{d \tilde \lambda }{da } \tilde \lambda a (\partial_{\tilde \lambda} v_{0,0}') (s/a)\notag\\
        &+ 2 w_k' (s/a) + a (\partial_a w_k') (s/a) -  w_k' (s/a) - \frac sa w_k'' (s/a) + a \frac{d \tilde \lambda}{da} (\partial_{\tilde \lambda } w_k' )(s/a) \notag\\
        &- 2 w_0' (s/a) - a (\partial_a w_0') (s/a) +  w_0' (s/a) + \frac sa w_0'' (s/a) - a \frac{d \tilde \lambda}{da} (\partial_{\tilde \lambda } w_0' )(s/a) \notag\\
        :={}& (I') + (II') + (III') + (IV') + (V') + (VI') + (VII') + (VIII') + (IX').\notag
    \end{align}
    All of these quantities can be estimated by similar logic as in \eqref{proof partial_a phi_a ests, eqn 11}--\eqref{proof partial_a phi_a ests, eqn 19} above, except for the terms
    \begin{equation} \label{proof partial_a phi_a ests, eqn 25}
        \hat C_{k,1} s a^{-1} R_1''(s/a) , \quad
        \tilde \lambda s a^{-1} v_{k,0}''(s/a) , \quad
        \tilde \lambda \frac sa v_{0,0}''(s/a) , \quad
        \frac sa w_k'' (s/a), \quad \text{and} \quad
        \frac sa w_0''(s/a)
    \end{equation}
    contributed from lines $(I'), (IV'), (VI'), (VIII')$, and $(IX')$, respectively.
    These terms are instead estimated as follows:
    \begin{gather} \label{proof partial_a phi_a ests, eqn 26} \begin{aligned}
        &\left| \hat C_{k,1} s a^{-1} R_1''(s/a) \right|  + \left| \frac sa w_k'' (s/a) \right| + \left|  \frac sa w_0''(s/a)  \right| \\
        \le{}& C \frac sa \langle s/a \rangle^{-1} 
        && ( \text{by \eqref{proof partial_a phi_a ests, eqn 9} and Lemma \ref{lem interior eigenfunctions})} \\
        \le{}& C \frac{ (s/a) }{1 + (s/a) } 
        \le C.
    \end{aligned} \end{gather} 
    \begin{equation} \label{proof partial_a phi_a ests, eqn 27}
        \left| \tilde \lambda s a^{-1} v_{k,0}''(s/a) \right| + \left| \tilde \lambda \frac sa v_{0,0}''(s/a) \right|
        \le C s \langle s/a \rangle^0 \le C  \qquad (\forall 0 \le s \le s_0)
    \end{equation}
    by the fact that $|\tilde \lambda | \le C a$ (Theorem \ref{thm global eigenfunctions}) and the estimates from Lemma \ref{lem interior eigenfunctions}.
    In summary,
    \begin{equation} \label{proof partial_a phi_a ests, eqn 28}
        |\partial_s\partial_a ( \phi_{k,a} (s) - \phi_{0,a}(s)) | \le C = C(n,k,s_0) \qquad \forall |s| \le s_0.
    \end{equation}

    Additionally, taking $\partial_s$ of \eqref{proof partial_a phi_a ests, eqn 22} gives, for $0 \le s \le s_0$, 
    \begin{gather} \label{proof partial_a phi_a ests, eqn 29} \begin{aligned}
        \partial_s \partial_a [\phi_{0,a} (s)]
        ={}& - a^{-2} R_0'(s/a) - sa^{-3} R_0''(s/a) + \frac {d \tilde \lambda }{da} 2 C_1 s  \\
        &+ \frac{d\tilde \lambda}{da} a R_1'(s/a) + 2 \tilde \lambda  R_1'(s/a) - \tilde \lambda  R_1'(s/a) - \tilde \lambda \frac sa R_1''(s/a) \\
        &+  \frac{d\tilde \lambda}{da} a v_{0,0}'(s/a) + 2 \tilde \lambda  v_{0,0}'(s/a) - \tilde \lambda  v_{0,0}'(s/a) - \tilde \lambda \frac sa v_{0,0}''(s/a) \\
        &+ 2  w_0'(s/a) -  w_0'(s/a) - \frac sa w_0''(s/a).
    \end{aligned} \end{gather}
    By similar logic as in \eqref{proof partial_a phi_a ests, eqn 23},
    it follows that for all $0 \le s \le s_0$
    \begin{multline} \label{proof partial_a phi_a ests, eqn 30}
        |\partial_s \partial_a [\phi_{0,a} (s)] | 
        \le C \left( a^{-2} \langle s/a \rangle^{1-n} + a^{-2} \langle s/a \rangle^{1-n} \frac{ (s/a)}{\langle s/a \rangle} + 1 + \frac{(s/a)}{\langle s/a \rangle} \right)\\
        \le C a^{-2} \langle s/a \rangle^{1-n} + C 
    \end{multline}
    since $\frac{(s/a)}{\langle s/a \rangle} \le C \frac{(s/a)}{1 + (s/a)} \le C$.
    
    Combining \eqref{proof partial_a phi_a ests, eqn 6.5} and \eqref{proof partial_a phi_a ests, eqn 20} proves \eqref{lem partial_a phi_a ests, eqn 1}.
    Combining \eqref{proof partial_a phi_a ests, eqn 6.6} and \eqref{proof partial_a phi_a ests, eqn 28} proves \eqref{lem partial_a phi_a ests, eqn 1.5}.
    Combining \eqref{proof partial_a phi_a ests, eqn 6}, \eqref{proof partial_a phi_a ests, eqn 20}, and \eqref{proof partial_a phi_a ests, eqn 23} proves \eqref{lem partial_a phi_a ests, eqn 2}.    
    Combining \eqref{proof partial_a phi_a ests, eqn 6.1}, \eqref{proof partial_a phi_a ests, eqn 28}, and \eqref{proof partial_a phi_a ests, eqn 30} proves \eqref{lem partial_a phi_a ests, eqn 2.5}.
    \eqref{lem partial_a phi_a ests, eqn 3} then follows from \eqref{lem partial_a phi_a ests, eqn 1} and \eqref{lem partial_a phi_a ests, eqn 1.5} with $\delta =1$ and Proposition \ref{prop vol element bounds}.
\end{proof}

\subsubsection{Integral Estimates}

\begin{lem} \label{lem L^2_a est for global eigenfunctions}
    Let $n\ge 3$ and $K \in \mathbb N$.
    For all $0 < s_0 \le s_0^* (n,K) \ll 1$,
    $0 < a \le a^* (n,K, s_0 ) \ll 1 $,
    and $k \in \mathbb N $ with $0 \le k \le K$,
    the following holds:

    If $\phi_{k,a} = \phi_{k,a,s_0}(s)$ denotes the $H_a$-eigenfunction constructed in Theorem \ref{thm global eigenfunctions},
    then there exists $C(n,k) > 0$ such that 
    \begin{equation}
        0 < C(n,k)^{-1} \le \| \phi_{k,a} \|_{L^2_a(\R)} \le C(n,k) < \infty.
    \end{equation}
\end{lem}
\begin{proof}
    Fix $k \in \mathbb N$ with $0 \le k \le K$.
    Let $\chi_{ |s| < s_0 }$ denote the characteristic function for $s \in (-s_0, s_0)$ and similarly define $\chi_{|s| \ge s_0}$.
    
    By Lemmas \ref{lem rewriting Laguerre poly} and \ref{lem tilde phi_k,a pointwise ests}, there holds the pointwise estimate for all $s \in \R$
    \begin{gather} \label{proof L^2_a bound for global eigfuncs, eqn 1} \begin{aligned}
        \phi_{k,a}(s) ={}& \sum_{i=0}^k \hat C_{k,i} a^{2i} u_i (s/a) + \tilde \phi_{k,a}(s)  \\
        ={}& \chi_{|s| < s_0} \cdot \left(\sum_{i=0}^k \hat C_{k,i} a^{2i} u_i (s/a)\right)
        + \chi_{|s| \ge s_0} \cdot \left( C_k P_{k}(s) + O (a; n,k,s_0) |s|^{2k-1} \right)  \\
        &+ O(a; n,k,s_0) ( 1 + |s|^{2k+1} ) 
    \end{aligned} \end{gather}
    where $C_k = C_k(n,k) > 0$ is a fixed constant depending on $n,k$ that does \emph{not} vary from line to line.
    By the asymptotics in Proposition \ref{prop generalized kernel},
    \begin{equation} \label{proof L^2_a bound for global eigfuncs, eqn 2}
        \left| \chi_{|s| < s_0} \left(\sum_{i=0}^k \hat C_{k,i} a^{2i} u_i (s/a)\right) \right| \le C = C(n,k) .
    \end{equation}
    
    It follows that 
    \begin{gather} \label{proof L^2_a bound for global eigfuncs, eqn 3} \begin{aligned}
        &\| \phi_{k,a} - \chi_{|s| \ge s_0} C_k P_k \|_{L^2_a} \\
        \le{}& \left\| \chi_{|s| < s_0} \left(\sum_{i=0}^k \hat C_{k,i} a^{2i} u_i (s/a)\right)\right\|_{L^2_a}  
        + O(a; n,k, s_0) \| 1 + |s|^{2k+1} \|_{L^2_a}
        && (\text{by \eqref{proof L^2_a bound for global eigfuncs, eqn 1}}) \\
        \le{}& C \| \chi_{|s| < s_0 } \|_{L^2_a} + O(a;n,k,s_0) \| 1 + |s|^{2k+1} \|_{L^2_a}
        && ( \text{by \eqref{proof L^2_a bound for global eigfuncs, eqn 2}} ) \\
        \le{}& O(s_0; n,k) + O(a;n,k,s_0) 
    \end{aligned} \end{gather}
    where the last line follows from the estimates in Propositions \ref{prop exp -F_a bounds} and \ref{prop vol element bounds}.
    Using Propositions \ref{prop exp -F_a bounds} and \ref{prop vol element bounds},
    it can moreover be shown that there exists $C_k'= C_k'(n,k) > 0$ such that 
    \begin{equation} \label{proof L^2_a bound for global eigfuncs, eqn 4}
        0 < C_k'^{-1} C_k \le \| \chi_{|s| \ge s_0 } C_k P_k \|_{L^2_a} \le C_k' C_k< \infty
    \end{equation}
    for all $0 < s_0 \le s_0^*(n,K) \ll 1$ and all $0 < a \le a^* (n,K, s_0) \ll 1$.
    Combining \eqref{proof L^2_a bound for global eigfuncs, eqn 3} and \eqref{proof L^2_a bound for global eigfuncs, eqn 4} then yields the statement of the lemma.
\end{proof}

\begin{lem} \label{lem L^2_a est for global eigenfunctions + weight}
    Let $n\ge 3$ and $K \in \mathbb N$.
    For all $0 < s_0 \le s_0^* (n,K) \ll 1$,
    $0 < a \le a^* (n,K, s_0 ) \ll 1 $,
    and $k \in \mathbb N $ with $0 \le k \le K$,
    the following holds:

    If $\phi_{k,a} = \phi_{k,a,s_0}(s)$ denotes the $H_a$-eigenfunction constructed in Theorem \ref{thm global eigenfunctions},
    then there exists $C = C(n,k,s_0) > 0$ such that 
    \begin{equation}
        \left\| \frac{\phi_{k,a}(s)}{1 + (|s|/a)^{n+1}} \right\|_{L^2_a} \le C a^{\frac n2}.
    \end{equation}
\end{lem}
\begin{proof}
    Fix $k$ with $0 \le k \le K$.
    Throughout the proof, $C = C(n,k,s_0) > 0$ denotes a positive constant depending only on $n,k,s_0$ and which may change from line to line.

    We estimate the $L^2_a$-norm in regions where $|s| \le s_0$ and $|s| > s_0$.
    First,
    \begin{gather} \label{proof L^2_a est for global eigenfunctions + weight, eqn 1} \begin{aligned}
        &\int_0^{s_0} \frac{ \phi_{k,a}^2 (s)}{( 1 + (s/a)^{n+1} )^2} d \mu_a \\
        \le{}& C \int_0^{s_0} \phi_{int; k, a}^2(s/a) \cdot \langle s/a \rangle^{-2n-2} dV_a(s) 
        && ( \text{Proposition \ref{prop exp -F_a bounds}} ) \\
        ={}& C \int_0^{s_0} \left( \sum_{i=0}^k \hat C_{k,i} a^{2i} u_i(s/a) + \tilde \phi_{k,a} (s)\right)^2 \langle s/a \rangle^{-2n-2} dV_a(s) \\
        \le{}& C \int_0^{s_0} \left(\sum_{i=0}^k a^{4i} \langle s/a\rangle^{4i} + a^2  \right) \langle s/a \rangle^{-2n-2} dV_a(s) \\
        & (\text{by Proposition \ref{prop generalized kernel} and Lemma \ref{lem tilde phi_k,a pointwise ests}}) \\
        ={}& C a^n \int_0^{s_0/a}  \left(\sum_{i=0}^k a^{4i} \langle \sigma \rangle^{4i} + a^2  \right) \langle \sigma \rangle^{-2n-2} dV_{a=1}(\sigma)
        &&  (\text{substituting $\sigma = s/a$} )\\
        \le{}& C a^n \int_0^{s_0/a}  \left(\sum_{i=0}^k a^{4i} \langle \sigma \rangle^{4i} + a^2  \right) \langle \sigma \rangle^{-2n-2} \langle \sigma \rangle^{n-1} d \sigma 
        && (\text{by Lemma \ref{lem: profile function of Lawlor neck}}) \\
        \le{}& C a^n \sum_{i=0}^k a^{4i} \int_0^{s_0/a} \langle \sigma \rangle^{4i-n-3} d \sigma 
        + Ca^{n+2} \int_0^{\infty} \langle \sigma \rangle^{-n -3} d\sigma \\
        \le{}& C a^n \sum_{i=0}^k a^{4i} \int_0^{s_0/a} \langle \sigma \rangle^{4i-n-3} d \sigma 
        + C a^{n+2} && ( n \ge 3) .
    \end{aligned} \end{gather}
    When $0 \le i \le k$ and $4i - n- 3 < -1$, 
    \begin{equation} \label{proof L^2_a est for global eigenfunctions + weight, eqn 2}
        Ca^n a^{4i} \int_0^{s_0/a} \langle \sigma \rangle^{4i-n-3} d \sigma 
        \le C a^{n+4i}  \le C a^n.
    \end{equation}
    When $0 \le i \le k$ and $4i - n - 3 \ge -1$,
    \begin{equation} \label{proof L^2_a est for global eigenfunctions + weight, eqn 3}
        C a^n a^{4i} \int_0^{s_0/a} \langle \sigma \rangle^{4i-n-3} d \sigma 
        \le C a^{n+4i} \langle s_0/a \rangle^{4i-n-3+2} 
        \le C a^{n+4i} + C a^{n+4i-4i+n+1 } 
        \le C a^n.
    \end{equation}
    Using \eqref{proof L^2_a est for global eigenfunctions + weight, eqn 2} and \eqref{proof L^2_a est for global eigenfunctions + weight, eqn 3} in \eqref{proof L^2_a est for global eigenfunctions + weight, eqn 1} then gives
    \begin{equation} \label{proof L^2_a est for global eigenfunctions + weight, eqn 4}
        \int_0^{s_0} \frac{ \phi_{k,a}^2 (s)}{( 1 + (s/a)^{n+1} )^2} d \mu_a \le C a^n.
    \end{equation}

    Next, note that there exists a fixed constant $C_k = C_k(n,k) > 0$ such that, for all $|s| \ge s_0$ 
    \begin{equation}\label{proof L^2_a est for global eigenfunctions + weight, eqn 5}
        \phi_{k,a}(s) = 
        C_k P_k(s) + O ( a; n,k,s_0) ( 1+ s^{2k+1} )
    \end{equation}
    by Lemma \ref{lem tilde phi_k,a pointwise ests} with $\delta = 1$ and Lemma \ref{lem rewriting Laguerre poly}.
    It follows that 
    \begin{gather} \label{proof L^2_a est for global eigenfunctions + weight, eqn 6} \begin{aligned}
        &\int_{s_0}^\infty \frac{ \phi_{k,a}^2 (s)}{( 1 + (s/a)^{n+1} )^2} d \mu_a \\
        ={}& \int_{s_0}^\infty \frac{\left( C_k P_k(s) + O(a; n,k,s_0) ( 1 + s^{2k+1} )\right)^2}{(1 + (s/a)^{n+1})^2} d \mu_a 
        && ( \text{by \eqref{proof L^2_a est for global eigenfunctions + weight, eqn 5}} )\\
        \le{}& C \int_{s_0}^\infty \frac{ s^{4k} + a^2 s^{4k+2}}{(s/a)^{2n+2}} d \mu_a\\
        \le{}& C a^{2n+2} \int_{s_0}^\infty s^{4k+2} d \mu_a \\
        \le{}& C a^{2n+2} && (\text{by \eqref{eqn L^2_a est for polynomials}}).
    \end{aligned} \end{gather}
    The estimates \eqref{proof L^2_a est for global eigenfunctions + weight, eqn 4} and \eqref{proof L^2_a est for global eigenfunctions + weight, eqn 6} and the fact that $\phi_{k,a}$ is an odd function thereby complete the proof of the lemma.
\end{proof}

\begin{lem} \label{lem L^2_a est for partial_a phi}
    Let $n\ge 3$ and $K \in \mathbb N$.
    For all $0 < s_0 \le s_0^* (n,K) \ll 1$,
    $0 < a \le a^* (n,K, s_0 ) \ll 1 $,
    and $k \in \mathbb N $ with $0 \le k \le K$,
    the following holds:

    If $\phi_{k,a} = \phi_{k,a,s_0}(s)$ denotes the $H_a$-eigenfunction constructed in Theorem \ref{thm global eigenfunctions},
    then there exists $C = C(n,k,s_0) >0$ such that
    \begin{equation}
        \| \partial_a \phi_{k,a} \|_{L^2_a} \le C .
    \end{equation}
\end{lem}
\begin{proof}
    Throughout, $C = C(n,k,s_0 ) > 0$ denotes a positive constant which depends only on $n,k,s_0$ and which may change from line to line.

    By Lemma \ref{lem partial_a phi_a ests} with $\delta = 1$, there holds the pointwise estimate
    \begin{equation} \label{proof L^2_a est for partial_a phi, eqn 1}
        | \partial_a \phi_{k,a} (s) | 
        \le \frac{C |s| a^{-2} }{1 + (|s|/a)^{n-1}} + C (1 + s^{2k+1} ) 
        \qquad \forall s > 0.
    \end{equation}
    By Proposition \ref{prop vol element bounds},
    \begin{equation} \label{proof L^2_a est for partial_a phi, eqn 2}
        \| 1 + |s|^{2k+1} \|_{L^2_a} \le C        ,
    \end{equation}
    so it suffices to estimate the $L^2_a$-norm of $\frac{C |s| a^{-2} }{1 + (s/a)^{n-1}}$.

    We have the estimate
    \begin{gather} \label{proof L^2_a est for partial_a phi, eqn 3} \begin{aligned}
        &\int_a^\infty \left( \frac {Cs a^{-2} }{1 +(s/a)^{n-1}} \right)^2 d \mu_a \\
        \le{}& C \int_a^\infty \left( \frac{ s a^{-2} }{(s/a)^{n-1}} \right)^2 e^{-\frac{s^2}4} s^{n-1} ds
        && ( \text{by Propositions \ref{prop exp -F_a bounds}, \ref{prop vol element bounds}} )\\
        ={}& C a^{ 2n -6} \int_a^\infty s^{-n+3  } e^{-\frac{s^2}4} ds \\
        \le{}& C a^{n-3} \int_a^\infty e^{-\frac{s^2}4 } ds && ( n \ge 3)\\
        \le{}& C a^{n-3} \\
        \le{}& C && ( n \ge 3).
    \end{aligned} \end{gather}
    
    Next, for $s > 0$,
    \begin{equation} \label{proof L^2_a est for partial_a phi, eqn 4}
        \frac{C sa^{-2} }{1 + (s/a)^{n-1} } 
        \le Ca^{-1}\frac{ (s/a) }{1 + (s/a)^{n-1} } \le C a^{-1} .
    \end{equation}
    It follows that 
    \begin{gather} \label{proof L^2_a est for partial_a phi, eqn 5} \begin{aligned}
        \int_0^a \left( \frac{Cs a^{-2} }{1 + (s/a)^{n-1}} \right)^2 d \mu_a 
        \le{}& C a^{-2} \int_0^a d \mu_a 
        && ( \text{by \eqref{proof L^2_a est for partial_a phi, eqn 4}} ) \\
        \le{}& C a^{-2} \int_0^a \rho_a(s)^{n-1} d s 
        && ( \text{Propositions \ref{prop exp -F_a bounds}, \ref{prop vol element bounds}} )\\
        \le{}& C a^{-2+n} \\
        \le{}& C && ( n \ge 3).
    \end{aligned} \end{gather}

    Combining \eqref{proof L^2_a est for partial_a phi, eqn 1}, \eqref{proof L^2_a est for partial_a phi, eqn 2}, \eqref{proof L^2_a est for partial_a phi, eqn 3}, and \eqref{proof L^2_a est for partial_a phi, eqn 5} completes the proof.
\end{proof} 

\begin{lem} \label{lem L^2_a est for phi with beta}
    Let $n\ge 3$ and $K \in \mathbb N$.
    For all $0 < s_0 \le s_0^* (n,K) \ll 1$,
    $0 < a \le a^* (n,K, s_0 ) \ll 1 $,
    and $k \in \mathbb N $ with $0 \le k \le K$,
    the following holds:

    If $\phi_{k,a} = \phi_{k,a,s_0}(s)$ denotes the $H_a$-eigenfunction constructed in Theorem \ref{thm global eigenfunctions}
    and $\beta_a (s) $ is as in Lemma \ref{lem beta properties},
    then
    \begin{equation} \label{eqn phi_0 and beta L^2_a close}
         \| \phi_{0,a} - a^{-2} \beta_a \|_{L^2_a } \le C a \qquad
         \text{for some } C = C(n,s_0) > 0.
    \end{equation}

    In particular,
    \begin{align}
        \label{lem L^2_a est for phi with beta, eqn 1}
        \langle \phi_{0,a} , a^{-2} \beta_a \rangle_{L^2_a} &= \| \phi_{0,a} \|_{L^2_a}^2 + O(a; n, s_0)  \text{ and}\\
        \label{lem L^2_a est for phi with beta, eqn 2}
        \langle \phi_{k,a} , a^{-2} \beta_a \rangle_{L^2_a} &= O(a; n,k,  s_0)  \text{ for } k \ne 0.
    \end{align}
\end{lem}
\begin{proof}
    By Lemma \ref{lem tilde phi_k,a pointwise ests} with $\delta = 1$, there holds the pointwise estimate
    \begin{equation} 
        \phi_{0,a} (s) =  a^{-2} \beta_a(s)  + O (a; n,s_0) ( 1 + |s| )  \qquad \forall s \in \R.
    \end{equation}
    Thus,
    \begin{align*}
        \| \phi_{0,a} - a^{-2} \beta_a \|_{L^2_a} 
        = O(a; n,s_0) \| 1 + |s| \|_{L^2_a} \le C(n,s_0) a  = O(a;n,s_0) .
    \end{align*}
    by Proposition \ref{prop vol element bounds}.
    This proves \eqref{eqn phi_0 and beta L^2_a close}.

    The claimed estimate \eqref{lem L^2_a est for phi with beta, eqn 1} now follows.
    Indeed,
    \begin{multline}
        \left| \langle \phi_{0,a} , a^{-2} \beta_a \rangle_{L^2_a} - \| \phi_{0,a} \|^2_{L^2_a} \right| 
        = \left| \langle \phi_{0,a} ,  a^{-2} \beta_a - \phi_{0,a} \rangle_{L^2_a} \right| \\
        \le \| \phi_{0,a} \|_{L^2_a} \| a^{-2} \beta_a - \phi_{0,a} \| 
        \le C(n,s_0) a
    \end{multline}
    by Lemma \ref{lem L^2_a est for global eigenfunctions} and \eqref{eqn phi_0 and beta L^2_a close}.

    Next, note that for $1 \le k \le K$,
    \begin{gather*}
        ( 1 -k + \tilde \lambda _k )  \langle \phi_{k,a} , \phi_{0,a} \rangle_{L^2_a}
        = \langle H_a \phi_{k,a} , \phi_{0,a}  \rangle_{L^2_a} 
        = \langle \phi_{k,a} , H_a \phi_{0,a} \rangle_{L^2_a} 
        = (1 + \tilde \lambda_k ) \langle \phi_{k,a} , \phi_{0,a} \rangle_{L^2_a} \\
        \implies \langle \phi_{k,a} , \phi_{0,a} \rangle_{L^2_a} = 0.
    \end{gather*}
    where the first line uses Proposition \ref{prop H_a symmetric}.
    Thus, for $1 \le k \le K$,
    \begin{multline}
        | \langle \phi_{k,a}, a^{-2} \beta_a \rangle_{L^2_a} | 
        \le \left| \cancel{\langle \phi_{k,a} , \phi_{0,a} \rangle_{L^2_a}} \right| + \left| \langle \phi_{k,a} , a^{-2} \beta_a - \phi_{0,a} \rangle_{L^2_a} \right| \\
        \le \| \phi_{k,a} \|_{L^2_a} \| a^{-2} \beta_a - \phi_{0,a} \|_{L^2_a} 
        \le C(n,k,s_0) a
    \end{multline}
    where the last inequality uses Lemma \ref{lem L^2_a est for global eigenfunctions} and \eqref{eqn phi_0 and beta L^2_a close}.
    This proves \eqref{lem L^2_a est for phi with beta, eqn 2} and completes the proof.
\end{proof}

\subsection{Spectral Gap} \label{Subsection Spectral Gap}

Theorem \ref{thm global eigenfunctions} above shows there exist eigenfunctions $\phi_{k,a}$ of $H_{s,a}$ that can be constructed by gluing interior and exterior eigenfunctions.
However, it does not show that this gluing construction obtains \emph{all} eigenfunctions of $H_{s,a}$.

The aim of this subsection is to prove Theorem \ref{thm spectral gap}, which says that, when $s_0$ and $a$ are suitably small, $\{ \phi_k \}_{k =0}^K$ are in fact the first $K$ eigenfunctions and there is a spectral gap for the operator $H_{s,a}$ acting on odd functions on $\R$.
Together with Theorem \ref{thm global eigenfunctions}, Theorem \ref{thm spectral gap} will complete the proof of the paper's main result Theorem \ref{main thm intro} stated in the introduction.
Broadly, the proof of this spectral gap result proceeds by using the last subsection's estimates to prove $\phi_k$ has exactly $k$ zeros on $(0,\infty)$, and then applying the Sturmian Comparison Theorem.

\begin{lem} \label{lem phi_int positive}
    Let $n \ge 3$, $K \in \mathbb N$, and $C_0 > 0$.
    For all $0 < s_0 \le s_0^* (n,K) \ll 1$,
    $0 < a \le a^* (n,K, C_0, s_0) \ll 1$, and
    $\tilde \lambda \in \R$ with $|\tilde \lambda| \le C_0 a$,
    the interior eigenfunction $\phi_{int; K} = \phi_{int; \tilde \lambda , a, s_0, K}$ given by Lemma \ref{lem interior eigenfunctions} satisfies
    \begin{equation}
        \phi_{int; K}(\sigma) > 0 \qquad \forall \sigma \in \left( 0 , \frac{s_0}a \right)
    \end{equation}
\end{lem}
\begin{proof}
    Throughout the proof, we write $\phi = \phi_{int; K} = \phi_{int; \tilde \lambda, a, s_0, K}$ for simplicity.
    Throughout the proof, $C = C(n,K)> 0$ denotes a positive constant depending only on $n,K$ that may change from line to line.

    There exists $\Gamma = \Gamma(n) > 1$ such that 
    \begin{equation} \label{proof phi_int positive, eqn 1}
        u_0 (\sigma) =\beta_1(\sigma) \ge \frac12 \overline \beta > 0 \qquad \forall \sigma \ge \Gamma
    \end{equation}
    where $\overline \beta = \overline \beta(n) >0$ is defined as in Lemma \ref{lem beta properties}.
    By the asymptotics in Proposition \ref{prop generalized kernel}, the generalized kernel elements $u_i$ satisfy
    \begin{equation} \label{proof phi_int positive, eqn 2}
        |u_i | + |\sigma \partial_\sigma u_i | \le C \sigma^{2i}  \qquad \forall \sigma \ge \Gamma, \, \forall 0 \le i \le K+1.
    \end{equation}

    Using \eqref{proof phi_int positive, eqn 1}, \eqref{proof phi_int positive, eqn 2}, $|\tilde \lambda| \le C_0 a$, and Lemma \ref{lem interior eigenfunctions}, it follows that for all $ 0 < a \le a^* (n, K, C_0, s_0 ) \ll 1$ and $\Gamma \le \sigma < \frac{s_0}a $, 
    \begin{gather}  \label{proof phi_int positive, eqn 3}
    \begin{aligned}
        &\phi(\sigma) \\
        ={}& \sum_{i=0}^K \hat C_{K, i} a^{2i} u_i 
		+ \tilde \lambda \sum_{i=0}^K a^{2(i+1)} \left( \hat C_{K, i} u_{i+1} + v_i \right) + a^2 w_K \\
        \ge{}& u_0 (\sigma) -  \sum_{i=1}^K C a^{2i} |u_i (\sigma) | 
        - C_0 a  \sum_{i=1}^{K+1} Ca^{2i} |u_{i}(\sigma)| 
        - C_0 a \sum_{i=0}^K a^{2i+2} |v_i(\sigma) | - a^2 |w_K(\sigma) | \\
        \ge{}& \frac12 \overline \beta  -  \sum_{i=1}^K C s^{2i}_0
        - C_0 a  \sum_{i=1}^{K+1} Cs^{2i}_0
        - C_0 a \sum_{i=0}^K C s_0^{2i+2} - C a s_0 \\
        \ge{}& \frac12 \overline \beta - C \sum_{i=1}^{2K+2} s_0^i
    \end{aligned} \end{gather}
    where the last line uses the estimate $C_0 a < 1$ for all $0 < a \le a^* (n, K, C_0, s_0) \ll 1$.
    It follows that if $0 < s_0 \le s_0^* (n,K) \ll 1$ is sufficiently small, then 
    \begin{equation} \label{proof phi_int positive, eqn 4}
        \phi(\sigma) > \frac14 \overline \beta > 0 \qquad \forall 0 < a \le a^* (n, K, C_0, s_0 ) \ll1 , \, \forall \sigma \in \left[ \Gamma,  \frac{s_0}a  \right).
    \end{equation}

    By smoothness of the $u_i$, Lemma \ref{lem interior eigenfunctions}, and the fact that $\Gamma = \Gamma(n)$ depends only on $n$, it follows that 
    \begin{equation} \label{proof phi_int positive, eqn 5}
        \sum_{i=0}^{K+1} |\partial_\sigma u_i (\sigma) | + \sum_{i=0}^K|\partial_\sigma v_i |  + |\partial_\sigma w_K | \le C  \qquad \forall 0 \le \sigma \le \Gamma  .
    \end{equation}
    By Lemma \ref{lem beta properties} and the fact that $\Gamma = \Gamma(n)$ depednds only on $n$, there exists $c = c(n) > 0$ such that 
    \begin{equation} \label{proof phi_int positive, eqn 6}
        \partial_\sigma u_0 (\sigma) = \partial_\sigma \beta_1 (\sigma) \ge c > 0 \qquad \forall \sigma \in [0, \Gamma].
    \end{equation}
    Using \eqref{proof phi_int positive, eqn 5}, \eqref{proof phi_int positive, eqn 6}, and $C_0 a < 1$ for all $0 < a \le a^* (n, K, C_0, s_0) \ll 1$,
    it follows that for all $0 < a \le a^* (n, K, C_0, s_0) \ll 1$ and all $\sigma \in [0, \Gamma]$,
    \begin{gather} \label{proof phi_int positive, eqn 7} \begin{aligned}
        &\partial_\sigma \phi_{int} \\
        ={}& \sum_{i=0}^K \hat C_{K,i} a^{2i} \partial_\sigma u_i
        + \tilde \lambda \sum_{i=0}^K a^{2i+2} \left( \hat C_{K,i} \partial_\sigma u_{i+1} + \partial_\sigma v_i \right) 
        + a^2 \partial_\sigma w_K \\
        \ge{}& \partial_\sigma u_0 - \sum_{i=1}^K C a^{2i} |\partial_\sigma u_i| 
        - C_0 a\sum_{i=0}^K C a^{2i+2} |\partial_\sigma u_{i+1} | 
        - C_0 a \sum_{i=0}^K a^{2i+2} |\partial_\sigma v_i | 
        - a^2 |\partial_\sigma w_K | \\
        \ge{}& c - C a^2  \\
        \ge{}& \frac12 c > 0.
    \end{aligned} \end{gather}
    Since $\phi_{int}$ is an odd function, $\phi_{int}(0) = 0$ and the derivative estimate \eqref{proof phi_int positive, eqn 7} implies
    \begin{equation} \label{proof phi_int positive, eqn 8}
        \phi_{int} (\sigma) > 0 \qquad \forall 0 < a \le a^* (n, K, C_0, s_0) \ll 1, \, \forall \sigma \in (0, \Gamma].
    \end{equation}
    Combining \eqref{proof phi_int positive, eqn 4} and \eqref{proof phi_int positive, eqn 8} thereby completes the proof.
\end{proof}

\begin{lem} \label{lem eigfunc zeros}
    Let $n \ge 3$ and $K \in \mathbb N$.
    For all $0 < s_0 \le s_0^* (n,K) \ll 1$, $1 \ll S_0^*(n,K)\le S_0 < \infty$, $0 < a \le a^* (n,K, s_0, S_0) \ll 1$, and $k \in \mathbb N $ with $0 \le k \le K$,
    the eigenfunction $\phi_k= \phi_{k, a, s_0}$ defined as in Theorem \ref{thm global eigenfunctions} has exactly $k$ zeros on the region $s \in (s_0, S_0)$.
\end{lem}
\begin{proof}
    For any $k \in \mathbb N$, let $P_k(s) = L_k^{\left( \frac{n-2}2 \right) } \left( \frac{(1+\overline c_0^2) s^2}4 \right)$ denote the generalized Laguerre polynomial as in Proposition \ref{prop plane eigenfunction}.
    By the choice of $s_0^*$ and $S_0^*$, we can assume without loss of generality that $P_k(s)$ has its $k$ distinct positive roots on the region $(s_0, S_0)$ \cite{DLMF}*{\href{https://dlmf.nist.gov/18.16.iv}{18.16(iv)}}. 
    By Lemma \ref{lem tilde phi_k,a pointwise ests} with $\delta = 1$ and Lemma \ref{lem rewriting Laguerre poly}, it follows that
    \begin{align}
        \phi_k(s) &= C P_k(s) + O ( a; n, k, s_0, S_0) && \forall s \in (s_0, S_0) \text{ and} \\
        \partial_s \phi_k(s) &= C \partial_s P_k(s) + O(a; n, k,s_0, S_0) && \forall s \in (s_0, S_0) 
    \end{align}
    for some constant $C= C(n,k) \ne 0$  depending only on $n$ and $k$.
    An implicit function theorem argument then applies to deduce that, for all $0 < a \le a^*(n,K, s_0, S_0) \ll1 $, $\phi_k$ has exactly $k$ zeros in $(s_0, S_0)$, all of which arise as perturbations of roots of $P_k$.
\end{proof}

\begin{lem} \label{lem eigfunc positive far away}
    Let $n \ge 3$ and $K \in \mathbb N$.
    For all $0 < s_0 \le s_0^*(n,K) \ll 1$,
    $1 \ll S_0^*(n,K) \le S_0 <\infty$,
    $0 < a \le a^* (n,K, s_0,S_0) \ll 1$,
    and $k \in \mathbb N$ with $0 \le k \le K$,
    the eigenfucntion $\phi_k = \phi_{k, a, s_0}(s)$ defined as in Theorem \ref{thm global eigenfunctions} is non-vanishing on $[S_0, \infty)$, i.e.
        $$|\phi_{k}(s) | > 0 \qquad \forall s \ge S_0.$$
\end{lem}
\begin{proof}
    First, consider the case where $1 \le k \le K$.
    Consider the linear operator 
    \begin{equation} \label{proof eigfunc pos far away, eqn 1}
        \mathcal L_{k,a} := H_a - ( 1-k+\tilde \lambda_{k,a}) = L_a - \frac s2 \partial_s + k - \tilde \lambda_{k,a}
    \end{equation}
    where $\tilde \lambda_{k,a} = \tilde \lambda_k(a)$ is as in Theorem \ref{thm global eigenfunctions}.
    Clearly, $\mathcal L_{k,a} \phi_k = 0$.
    We construct barriers to show $|\phi_k| > 0$ for all $s \ge S_0$.

    Since $|\tilde \lambda_{k,a} | \lesssim_{n,k, s_0} a$ (Theorem \ref{thm global eigenfunctions}), the constant function $u_0(s) = 1$ satisfies 
    \begin{equation} \label{proof eigfunc pos far away, eqn 1.5}
      \mathcal L_{k,a} u_0 = ( k - \tilde \lambda_{k,a} ) \ge 0  
    \end{equation}
    for all $0 < a \le a^*(n, K, s_0) \ll 1$.

    Consider also the function $u_1(s) = e^{\frac{s^2}8} > 0$.
    Then
    \begin{equation} \label{proof eigfunc pos far away, eqn 2}
        \partial_s u_1 = \frac s4 u_1  \qquad \text{and} \qquad
        \partial_{ss} u_1 = \frac14 u_1 + \frac {s^2}{16} u_1.
    \end{equation}
    Hence,
    \begin{gather} \label{proof eigfunc pos far away, eqn 3} \begin{aligned}
        &\mathcal L_{k,a} u_1 \\
        ={}& \frac1{1 + (\partial_s f_a)^2} \partial_{ss} u_1
        + \frac{\partial_s \left( \sqrt{ \frac{(s^2 +f_a^2)^{n-1}}{1 + (\partial_s f_a)^2}}\right)}{\sqrt{(1 + (\partial_s f_a)^2)(s^2 +f_a^2)^{n-1}}} \partial_s u_1 
        - \frac s2 \partial_s u_1 
        + (k- \tilde \lambda_{k,a}) u_1
        \\
        ={}& u_1 \left\{
        \frac{ \frac {s^2}{16} }{1 + (\partial_s f_a)^2}
        - \frac{ s^2}8  
        + k - \tilde \lambda_{k,a}
        + \frac{ \frac14  }{1 + (\partial_s f_a)^2}
        + \frac{\partial_s \left( \sqrt{ \frac{(s^2 +f_a^2)^{n-1}}{1 + (\partial_s f_a)^2}}\right)}{\sqrt{(1 + (\partial_s f_a)^2)(s^2 +f_a^2)^{n-1}}} \frac s4
        \right\} . 
    \end{aligned} \end{gather}
    The estimates for $\tilde \lambda_{k,a}$ (Theorem \ref{thm global eigenfunctions}) and the asymptotics of $f_a$ (Lemma \ref{lem: profile function of Lawlor neck}) imply there exists $C = C(n,k) > 0$ such that
    \begin{equation} \label{proof eigfunc pos far away, eqn 4}
        \left| k - \tilde \lambda_{k,a}
        + \frac{ \frac14  }{1 + (\partial_s f_a)^2}
        + \frac{\partial_s \left( \sqrt{ \frac{(s^2 +f_a^2)^{n-1}}{1 + (\partial_s f_a)^2}}\right)}{\sqrt{(1 + (\partial_s f_a)^2)(s^2 +f_a^2)^{n-1}}} \frac s4 \right| \le C < \infty 
    \end{equation}
    for all $s \ge 1$ and all $0 < a \le a^* (n, K, s_0, S_0) \ll 1$.
    Inserting estimate \eqref{proof eigfunc pos far away, eqn 4} into \eqref{proof eigfunc pos far away, eqn 3} gives, for all $s \ge S_0 \ge S_0^*(n,k)  \gg 1$ and all $0 < a \le a^* (n, K, s_0, S_0) \ll 1$,
    \begin{equation}  \label{proof eigfunc pos far away, eqn 5}
        \mathcal L_{k,a} u_1 
        \le u_1 \left\{ \frac {s^2}{16} - \frac {s^2}8 - C(n,k) \right\}  
        \le u_1 \left\{ - \frac{S_0^2}{16} - C(n,k) \right\} < 0.
    \end{equation}

    Equations \eqref{proof eigfunc pos far away, eqn 1.5} and \eqref{proof eigfunc pos far away, eqn 5} show, for any $\epsilon, \delta > 0$, 
    $u = \epsilon u_0 - \delta u_1$ is a subsolution of $\mathcal L_{k,a} u = 0$ on the region $(S_0, \infty)$.
    Since $\phi_k(S_0)$ is an $O(a; n, k, s_0, S_0)$ perturbation of a generalized Laguerre polynomial (Lemmas \ref{lem rewriting Laguerre poly} and \ref{lem tilde phi_k,a pointwise ests}), there exists $\epsilon_0 = \epsilon_0(n,k) >0$ such that
    \begin{equation} \label{proof eigfunc pos far away, eqn 6}
        |\phi_k(S_0) | \ge \epsilon_0(n,k) >0   
    \end{equation}
    for all $0 < a \le a^* (n,k, s_0,S_0) \ll 1$.
    Note also that the bounds from Lemma \ref{lem exterior eigenfunctions} imply, for any $\delta > 0$, $|\phi_k | \le \delta u_1$ near infinity.

    By \eqref{proof eigfunc pos far away, eqn 6}, $\phi_k(S_0) > 0$ or $\phi_k(S_0) < 0$.
    Consider the case that $\phi_k(S_0) > 0$.
    Then, for any $\delta > 0$, $\epsilon_0 - \delta u_1$ is a subsolution below $\phi_k$ at $S_0$ and near infinity.
    The maximum principle for $\mathcal L_{k,a}$ then implies that 
    \begin{equation} \label{proof eigfunc pos far away, eqn 7}
        \phi_k(s) \ge \epsilon_0 - \delta u_1(s) \qquad \forall s \in [S_0, \infty).
    \end{equation}
    Taking $\delta \searrow 0$ proves $\phi_k(s) >0$ for all $s \in[S_0, \infty)$.
    In the case that $\phi_k(S_0) < 0$, a similar argument applies using the supersolution $-\epsilon_0 + \delta u_1$.
    This completes the proof of the Lemma when $1 \le k \le K$.

    When $k =0$ and $\tilde \lambda_{0,a} \le 0$, $u_0 = 1$ is still a subsolution and same argument above applies to prove $\phi_{0,a}$ is non-vanishing on $[S_0, \infty)$.
    If, however, $\tilde \lambda_{0,a} > 0$, then $u_0 = 1$ is no longer a subsolution.
    In this case, estimates akin to \eqref{proof eigfunc pos far away, eqn 3}--\eqref{proof eigfunc pos far away, eqn 5} apply to show
    $u_0 = s^{-4 \tilde \lambda_{0,a}} >0$
    is a subsolution $\mathcal L_{k,a} u_0 \ge 0$ on the region $s \in (S_0, \infty)$. 
    The remainder of the argument then proceeds as above.
    The details are left to the reader.
\end{proof}

Combining Lemmas \ref{lem phi_int positive}--\ref{lem eigfunc positive far away} with a Sturmian Theorem, we can now prove a spectral gap theorem for $H_a$.

\begin{theorem}[Spectral Gap] \label{thm spectral gap}
    Let $n \ge 3$, $K \in \mathbb N$, and $\epsilon > 0$.
    For all $0 < s_0 \le s_0^*(n, K) \ll 1$ and $0 < a \le a^*(n,K, s_0, \epsilon) \ll 1$, the following holds:
    
    If $v : \R \to \R$, $v=v(s)$ is an odd $C^2$ function with
    \begin{equation}
        \| v \|_{H^1_a} + \| H_a v \|_{L^2_a} < \infty \qquad \text{and} \qquad
        \langle v , \phi_k \rangle_{L^2_a} = 0 \quad \forall 0 \le k \le K,
    \end{equation}
    where $\phi_{k} = \phi_{k,a,s_0}(s)$ denotes the $H_a$-eigenfunction as in Theorem \ref{thm global eigenfunctions},
    then 
    \begin{equation}
        \langle v, H_a v \rangle_{L^2_a} \le \left( -K + \epsilon \right) \| v \|_{L^2_a}^2 .
    \end{equation}
\end{theorem}
\begin{proof}
    Define $L^2_a( (0, \infty))$ to be the space of (equivalence classes of Lebesgue) measurable functions $u : (0, \infty) \to \R$ with $\int_0^\infty u^2 d\mu_a < \infty$.
    Define $H^1_a( (0, \infty))$ to be the completion of $C^\infty_c ([0, \infty) )$ with respect to the $H^1_a$ norm.
    Clearly, $L^2_a((0, \infty))$ and $ H^1_a((0, \infty))$ are Hilbert spaces.
    By Sobolev embedding, $H^1_a((0,\infty)) \subset C^{0, \alpha}_{loc} ( (0, \infty))$ for some $\alpha \in (0,1)$ (depending on $n, a$). 
    In fact, Proposition \ref{prop vol element bounds} implies the $H^1_a$-norm is comparable to the standard $H^1$-norm on $\R$ near $s=0$, and so $H^1_a((0, \infty)) \subset C^{0, \alpha}_{loc}([0, \infty))$.
    Thus, we can define $$H^1_{a; 0} ( (0, \infty)) = \{ u \in H^1_a((0, \infty)) \colon u(0) = 0 \}$$ and deduce $H^1_{a;0} ( (0, \infty)) \subset H^1_a ((0, \infty))$ is a closed subspace.

    Define an unbounded linear operator 
    \begin{gather} 
        \label{proof spectral gap, eqn 1}
        H_a : D(H_a) \to L^2_a((0, \infty)) \\
        \label{proof spectral gap, eqn 2}
        \text{on domain } D(H_a) = \{ u \in H^1_a((0, \infty)) \colon u(0) = 0 \text{ and } H_a u \in L^2_a((0, \infty)) \}  
    \end{gather}
    where $H_a$ is defined as in \eqref{eqn H_a defn} but with derivatives taken in the sense of distributions.
    
    \begin{claim} \label{proof spectral gap, claim 1}
        $H_a : D(H_a) \to L^2_a ((0, \infty))$
        is a self-adjoint unbounded linear operator.
    \end{claim}
    \begin{claimproof}
        Note $C^\infty_c((0, \infty)) \subset D(H_a) \subset H^1_{a;0}((0, \infty))$. The first inclusion is dense since $C^\infty_c((0, \infty)) \subset H^1_{a;0}((0, \infty))$ is dense.
        It then follows from Proposition \ref{prop H_a symmetric} that 
        \begin{equation} \label{proof spectral gap, eqn 3}
            \left( u , H_a v \right)_{L^2_a} = -\int_0^\infty \frac {(\partial_s u )(\partial_s v)}{1 + \partial_sf_a^2 } d \mu_a + \int_0^\infty u v d \mu_a \qquad \forall u , v \in D(H_a).
        \end{equation}
        Hence, $H_a : D(H_a) \to L^2_a((0, \infty))$ is a symmetric operator and bounded above by 1, that is,
        \begin{equation} \label{proof spectral gap, eqn 4}
            \left( u, H_a v \right)_{L^2_a} = (H_a u, v )_{L^2_a} \quad \text{and} \quad 
            (u, H_a u)_{L^2_a} \le \| u \|_{L^2_a}^2
            \qquad \forall u, v \in D(H_a).
        \end{equation}
        
        Since $0 \le \partial_s f_a \le \overline c_0 = \overline c_0(n)$ (Lemma \ref{lem: profile function of Lawlor neck}), we have that
        \begin{equation} \label{proof spectral gap, eqn 5}
            \frac1{1+\overline c_0^2} ( u, v)_{H^1_a} \le \int_0^\infty \frac{ (\partial_s u) (\partial_s v) }{1 + \partial_s f_a^2 } d \mu_a + \int_0^\infty u v  d\mu_a \le ( u, v)_{H^1_a}
            \qquad \forall u, v \in H^1_a((0, \infty)).
        \end{equation}
        The Lax-Milgram Theorem can then be applied to show for all $v \in L^2_a((0, \infty))$ there exists a weak solution $ u \in H^1_{a;0}((0, \infty))$ of $(-H_a + 2) u = f$.
        Hence, $-H_a +2 : D(H_a) \to L^2_a((0, \infty))$ is surjective.
        To summarize, $-H_a +1 : D(H_a) \to L^2_a((0, \infty))$ is a densely defined, symmetric operator which is bounded below by 0 and has $-H_a+1 +1$ surjective.
        \cite{ReedSimon1}*{Chapter VIII, Problem 10, p. 313} then gives that $-H_a +1$ and $H_a : D(H_a) \to L^2_a((0, \infty))$ are self-adjoint operators.
    \end{claimproof}

    \begin{claim} \label{proof spectral gap, claim 2}
        $H_a : D(H_a) \to L^2_a((0, \infty))$ has compact resolvent.
    \end{claim}
    \begin{claimproof}
        Since $H_a$ is bounded above by 1 (see the proof of Claim \ref{proof spectral gap, claim 1}), $2$ is in the resolvent and it suffices to show $(H_a - 2)^{-1} : L^2_a((0, \infty)) \to L^2_a((0, \infty))$ is compact.
        Let $f_i$ be a bounded sequence in $L^2_a((0, \infty))$ and set $u_i = (H_a- 2)^{-1} f_i \in D(H_a) \subset L^2_a((0, \infty))$.
        $u_i$ is a bounded sequence in $L^2_a((0, \infty))$.
        Integrating both sides of $(H_a - 2) u_i = f_i$ against $-u_i$ and using \eqref{proof spectral gap, eqn 3} and \eqref{proof spectral gap, eqn 5}, it follows that
        \begin{equation}
            \frac1{1 + \overline c_0^2} \| u_i \|_{H^1_a((0,\infty))}^2 \le \| f_i \|_{L^2_a((0, \infty))} \| u_i \|_{L^2_a((0, \infty))}.
        \end{equation}
        Thus, $u_i$ is a bounded sequence in $H^1_a((0, \infty))$.
        Using the embedding $H^1_a((0, \infty)) \subset C^{0, \alpha}_{loc}([0, \infty))$, it follows that some subsequence $u_{i_j}$ converges to $u_\infty$ in $C^{0}_{loc}([0, \infty))$.
        
        Set $C_0 := \sup_{i} \| u_i \|_{H^1_a((0, \infty))} < \infty$.
        Then, for any $R > 0$ and any $ j,k \in \mathbb N$,
        \begin{align*}
            &\| u_{i_j} - u_{i_k} \|_{L^2_a((0, \infty))}^2 \\
            ={}& \int_0^R ( u_{i_j} - u_{i_k})^2 d \mu_a + \int_R^\infty ( u_{i_j} - u_{i_k})^2 d \mu_a \\
            \le{}& \int_0^R ( u_{i_j} -u_\infty + u_\infty -  u_{i_k} )^2 d \mu_a + \frac1{R^2} \int_R^\infty s^2 ( u_{i_j} - u_{i_k})^2 d \mu_a \\
            \le{}& 2 \int_0^R ( u_{i_j} - u_\infty )^2 d \mu_a + 2 \int_0^R ( u_{i_k} - u_\infty)^2 d \mu_a 
            + \frac 1{R^2} \| u_{i_j} - u_{i_k} \|_{H^1_a ((0, \infty))}^2 && ( \text{Proposition \ref{prop Ecker's Sobolev ineq}} ) \\
            \le{}& 2R  \| u_{i_j} - u_\infty \|_{C^0 ([0, R])}^2 + 2 R \| u_{i_k} - u_\infty \|_{C^0([0, R])}^2   + \frac {4C_0^2}{R^2} .
        \end{align*}
        (Note that Proposition \ref{prop Ecker's Sobolev ineq} is justified by the density of $C^\infty_c ( (0, \infty)) \subset D(H_a)$.)
        It follows from this inequality that $u_{i_j}$ is a Cauchy sequence in $L^2_a((0, \infty))$ and therefore converges to $u_\infty$ in $L^2_a((0, \infty))$.
        This proves $(H_a - 2)^{-1}: L^2_a((0, \infty)) \to L^2_a((0, \infty))$ is compact.
    \end{claimproof}

    By \cite{ReedSimon4}*{Theorem XIII.64}, there exists a complete orthonormal basis $\{ u_i \}_{i \in \mathbb N}$ of $L^2_a((0, \infty))$ such that, for all $i$, $u_i \in D(H_a)$ is an eigenfunction of $H_a$ with eigenvalue say $\mu_i \in \R$.
    Moreover, $$1 \ge \mu_0 \ge \mu_1 \ge \mu_2 \ge \dots \to - \infty$$
    and the eigenvalues $\mu_i$ are given by the min-max principle.
    Elliptic regularity theory also implies the eigenfunctions $u_i$ are in fact smooth functions $u_i : [0, \infty) \to \R$ with $u_i(0) = 0$.
    Uniqueness of ODE solutions further implies the eigenvalues are all simple, that is,
        $$1 \ge \mu_0 > \mu_1 > \mu_2 > \dots \to -\infty.$$

    For $0 \le k \le K+1$, let $\phi_k = \phi_{k,a , s_0}$ denote the $H_a$-eigenfunctions as in Theorem \ref{thm global eigenfunctions} with eigenvalue say $\lambda_k = 1 - k + \tilde \lambda_k$.
    By the asymptotics of $\phi_k$ at infinity (Theorem \ref{thm global eigenfunctions} and Lemma \ref{lem exterior eigenfunctions}) and the fact that each $\phi_k$ is an odd function of $s$, $\phi_k \in D(H_a)$ after restricting $\phi_k$ to $(0, \infty)$.
    Therefore, $\{ \lambda_k \}_{0 \le k \le K+1} \subset \{ \mu_k \}_{k \in \mathbb N}$.
    In other words, for every $0 \le k \le K+1$, there exists $i = i(k)$ such that $\lambda_k = \mu_{i(k)}$ and $\phi_k$ is a constant multiple of $u_{i(k)}$.

    The Sturm Comparison Theorem implies that, for any $k \in \mathbb N$, $u_k$ has exactly $k$ zeros in $(0, \infty)$ (see e.g. \cite{DunfordSchwartz}*{Theorem XIII.7.50, page 1478}).
    By Lemmas \ref{lem phi_int positive}--\ref{lem eigfunc positive far away}, for any $0 \le k \le K+1$, $\phi_k$ also has exactly $k$ zeros in $(0, \infty)$.
    Hence, it must be the case that, for all $0 \le k \le K+1$, $\lambda_k = \mu_k$ and $\phi_k$ is a constant multiple of $u_k$.

    Let $v : \R \to \R$ be an odd $C^2$ function such that
    \begin{equation*}
        \| v \|_{H^1_a(\R)} + \| H_a v \|_{L^2_a(\R)} < \infty \qquad \text{and} \qquad \langle v, \phi_k \rangle_{L^2_a(\R)} = 0 \quad \forall 0 \le k \le K.
    \end{equation*}
    Then $v(0) = 0$, and the restriction of $v$ to $(0, \infty)$ is $L^2_a((0, \infty))$-orthogonal to $u_k = \phi_k$ for all $0 \le k \le K$.
    Because the eigenvalues are given by the min-max principle, it follows that
    \begin{multline}
        \langle v , H_a v \rangle_{L^2_a( \R) } = 2 \langle v, H_a v \rangle_{L^2_a((0, \infty))} 
        \le 2 \mu_{K+1} \| v \|_{L^2_a((0, \infty))}^2 = \lambda_{K+1} \| v \|_{L^2_a(\R)}^2  \\
        = ( 1 - K-1 + \tilde \lambda_{K+1} ) \| v \|_{L^2_a(\R)}^2 
        \le \left( - K + \epsilon \right) \| v \|_{L^2_a(\R)}^2
    \end{multline}
    where the last inequality uses $|\tilde \lambda_{K+1} | \le C(n, K, s_0) a \le \epsilon$ for $0 < a \le a^*$ (Theorem \ref{thm global eigenfunctions}).
    This completes the proof.
\end{proof}
    

Combining Theorems \ref{thm global eigenfunctions} and \ref{thm spectral gap} proves the paper's main result, Theorem \ref{main thm intro}, stated in the introduction.
Proposition \ref{prop scaling mode intro} from the introduction then follows from Lemma \ref{lem L^2_a est for phi with beta}.

\section{Approximate Lagrangian Mean Curvature Flow Solutions} \label{Section Approx Flow Solutions}

In this section, we indicate how the eigenmodes of $H_{s,a}$ may be used to construct solutions $L^n(t) \subseteq \C^n$ ($n \ge 3$) of Lagrangian mean curvature flow that form finite-time singularities.
\textbf{We emphasize that, unlike other sections, the discussions throughout this section do \emph{not} constitute rigorous proofs.} 
Instead, this section is meant to provide informal evidence for the existence of Lagrangian mean curvature flows with specific singularity-formation dynamics.
The informal descriptions contained in this section will be made rigorous in our companion paper \cite{SS26II}.

\begin{remark}
    Throughout this section, fix $n, K \in \mathbb N$ with $n \ge 3$, and fix $0 < s_0 = s_0(n,K) \ll 1$ to be a small positive constant depending only on $n$ and $K$.
    Throughout this section, all constants implicit in the $\lesssim, \sim$ notation (see Definition \ref{defn notation}) will depend only on $n$ and $K$.
\end{remark}

For $u = u(s, \tau)$, $a = a(\tau)$, we consider the PDE
    \begin{equation} \tag{\ref{eq: main equation}}
        \partial_\tau u = \left( 1 - 2 \frac{\partial_\tau a}{a} \right) \beta_a + H_{s, a} u + Q_{s,a}(u).
    \end{equation}
By Corollary \ref{cor: PDE for the potential}, solutions $(u,a)$ of \eqref{eq: main equation} yield solutions $h = f_{a(\tau)}(s) + u_s(s,\tau)$ of the parabolically rescaled Lagrangian mean curvature flow.
Thus, it suffices to obtain solutions of \eqref{eq: main equation}.

Assume throughout this section that $n \ge 3$.
We adopt the ansatz that $0 < a \ll 1 $ and $u \approx \sum_{i=0}^K b_i(\tau) \phi_{i,a}(s)$ can be approximated by a time-dependent linear combination of finitely many eigenmodes $\phi_{i,a}$ as in Theorem \ref{thm global eigenfunctions} where also the time-dependent coefficients satisfy $\sum_{i=0}^K |b_i| \ll 1$.
Inserting this ansatz into \eqref{eq: main equation} yields
    \begin{equation} \label{approx eqn 1}
        \sum_{i=0}^K b_i' \phi_{i,a} + \sum_{i=0}^K b_i (\partial_a \phi_{i,a}) a'
        \approx \left( 1 - 2 \frac{a'}a \right)\beta_a + \sum_{i=0}^K b_i ( 1 - i + \tilde \lambda_{i,a}) \phi_{i,a}  + Q_{s,a}(u)
    \end{equation}
where $' = \frac d{d \tau}$ indicates a derivative with respect to $\tau$.
Lemma \ref{lem tilde phi_k,a pointwise ests} indicates $\phi_{0,a} \approx a^{-2} \beta_a$,
Theorem \ref{thm global eigenfunctions} shows $|\tilde \lambda_{i,a}| \lesssim a \ll 1$,
and the smallness of the coefficients $a, b_i$ suggests the nonlinear term $Q_{s,a}(u)$ will be small relative to the linear terms in the PDE.
Assuming additionally that
\begin{equation} \label{approx eqn, assumption 1}
    | a'| \ll 1,
\end{equation}
\eqref{approx eqn 1} can then be approximated by
    \begin{equation} \label{approx eqn 2}
        \sum_{i=0}^K b_i' \phi_{i,a} 
        \approx \left( a^2 - 2aa' \right)\phi_{0,a} + \sum_{i=0}^K b_i ( 1 - i ) \phi_{i,a} .
    \end{equation}
This approximate equation \eqref{approx eqn 2} can be solved by finding a solution of the ODE system 
\begin{gather} \label{approx ODE system}
    \left\{ \begin{aligned}
        b_0' &= a^2 - 2 a a' + b_0, \\
        b_1' &= (1-1) b_1, \\
        \vdots& \\
        b_i'&= ( 1- i ) b_i, \\
        \vdots& \\
        b_K' &= (1-K) b_K.
    \end{aligned} \right.
\end{gather}
This ODE system \eqref{approx ODE system} has a solution given by
\begin{equation} \label{eqn approx ODE soln}
    b_K(\tau) = - b_0(\tau) = a^2(\tau) = C_0 e^{(1-K) \tau} , \qquad b_1 = b_2 = \dots = b_{K-1} = 0
\end{equation}
where $C_0 > 0$ is an arbitrary positive constant.
We henceforth require that $K \ge 2$ so that this ODE solution satisfies $a , |b_i| , |a'| \ll 1$ for $\tau \gg 1$, which is consistent with the approximations and smallness assumptions made above.
This solution \eqref{eqn approx ODE soln} to the ODE system then corresponds to the parabolically rescaled Lagrangian mean curvature flow solution $\widetilde L_K(\tau)$ given by the graph of
\begin{equation} \label{approx eqn 3}
    h_K(s, \tau) \approx  f_{a(\tau)} + C_0e^{(1-K)\tau} \partial_s \left(  \phi_{K,a(\tau)} - \phi_{0,a(\tau)}\right), \qquad a(\tau) = \sqrt{C_0} e^{\frac{1-K}2 \tau}
\end{equation}
defined for all $\tau \gg 1$.

Undoing the parabolic rescaling via
    $$\tau = - \ln ( T-t) , \qquad x =  s e^{-\tau/2}= s\sqrt{T-t} , \qquad e^{\tau/2} \gamma_{T - e^{-\tau}} = \tilde \gamma_{\tau} =  \{ (s, h(s, \tau) ) \, |\, s \in \R \}$$
the corresponding Lagrangian mean curvature flow solution $L_K^n (t) \subseteq \C^n$ is given by the profile curve
\begin{multline} \label{eqn approx unrescaled LMCF soln curve}
    \gamma_{K,t} \approx \left\{ \left. \left(x , f_{a \sqrt{T-t}  } ( x ) + a^2 \sqrt{T-t} \left( \partial_s \phi_{K,a} - \partial_s \phi_{0,a} \right) \left( \frac x{\sqrt{T-t}} \right) \right) \, \right| \, x \in \R \right\}, \\ a = \sqrt{C_0} e^{\frac{1-K}2 \tau} = \sqrt{C_0} (T-t)^{\frac{K-1}2}.
\end{multline}

Observe that, in \eqref{approx eqn 3}, the term $C_0 e^{(1-K)\tau} \partial_s ( \phi_{K,a} - \phi_{0,a} ) = a^2 \partial_s ( \phi_{K,a} - \phi_{0,a} )$ is order $a^2$, which is smaller $a^2 \ll a$ than the scale $a$ of the Lawlor neck profile $f_{a(\tau)}$ in \eqref{approx eqn 3}.
Hence, for any $K \ge 2$, the blow-up rate of the second fundamental form $A$ of $L_K(t)$ should be governed by the scale of the Lawlor neck, that is,
    $$\sup_{\mathbf x} | A_{L_K(t)} | \sim \frac1{a \sqrt{T-t}} \sim \frac1{(T-t)^{K/2}} \gg \frac1{\sqrt{T-t}}.$$
Since $K \ge2$, these expected curvature blow-up rates are always Type II, and we obtain curvature blow-up rates with arbitrarily large polynomial powers.

\subsection{Bounded Mean Curvature}

The mean curvature $\mathbf H_{\widetilde L_K(\tau)}$ of the parabolically rescaled solution $\widetilde L_K(\tau)$ satisfies 
\begin{multline*}
    {\sqrt{1+(\partial_s f_a + \partial_{ss} u)^2}}  \cdot \mathbf H_{\widetilde L_K(\tau)} = \mathcal H [f_a + u_s ] = \cancelto{0}{\mathcal H[f_a]} + ( L_{s, a} u )_s + \mathcal Q_{s,a}(u) \\
    \approx (L_{s,a} u )_s = \frac{d}{ds} \left( \frac{u_{ss}}{1 + (\partial_s f_a)^2} + (n-1) \frac{s u_s}{s^2 + f_a^2} \right) 
\end{multline*}
by the computations in Section \ref{section Prelims}.
Using this expression with the approximate solution 
    $$u \approx C_0e^{(1-K)\tau} \left(  \phi_{K,a(\tau)} - \phi_{0,a(\tau)}\right), \qquad a(\tau) = \sqrt{C_0} e^{\frac{1-K}2 \tau}$$
from \eqref{approx eqn 3}, 
we obtain
\begin{align*}
    &|\mathcal{H}[f_a + u_s] |\\
    \approx{}& | (L_{s,a}u)_s |\\
    ={}& \left| \left(H_{s,a} u + \frac s2 \partial_s u - u \right)_s \right|
    && ( \text{since } H_{s,a} = L_{s,a} - \frac s2 \partial_s + 1 ) \\
    \approx{}& C_0e^{(1-K)\tau} \left| \left( (1-K + \tilde \lambda_{K,a}) \phi_{K,a} - (1 + \tilde \lambda_{0,a}) \phi_{0,a} \right. \right. \\
    & \qquad \left. \left. + \frac s2 \partial_s (\phi_{K,a} - \phi_{0,a}) - (\phi_{K,a} - \phi_{0,a}) \right)_s \right| \\
    \lesssim{}& a^2 \left(|\partial_s \phi_{K,a}| + |\partial_s \phi_{0,a}| + |s| | \partial_{ss}  \phi_{K,a}|  + |s| |\partial_{ss} \phi_{0,a} |  \right) 
    && ( \text{by Theorem \ref{thm global eigenfunctions}}).
\end{align*}

The pointwise estimates of Lemma \ref{lem tilde phi_k,a pointwise ests} suggest 
\begin{equation*}
    |\partial_s \phi_{K,a}| + |\partial_s \phi_{0,a}| + |s| | \partial_{ss}  \phi_{K,a}|  + |s| |\partial_{ss} \phi_{0,a} |  \lesssim \frac1{\sqrt{a^2 + s^2}} \qquad \text{when } |s| \lesssim 1.
\end{equation*}
Inserting these bounds into the estimate for $\mathcal H$ above yields
\begin{equation*}
    |\mathcal H[ f_a + u_s]| \lesssim a^2 \frac1{\sqrt{a^2 + s^2}} \lesssim a \sim (T- t)^{\frac{K-1}2}  \qquad \text{for } |s| \lesssim 1.
\end{equation*}
Additionally, the pointwise estimates of Lemma \ref{lem phi_k,a - phi_0,a pointwise ests} suggest
\begin{align*}
    |\partial_{ss}(\phi_{K,a} - \phi_{0,a}) | &\lesssim 1 \qquad \text{when } |s| \lesssim 1,\\
    \text{and so } \, \frac1{\sqrt{1 + (\partial_s f_a + \partial_{ss} u)^2}} &\lesssim 1 \qquad \text{when } |s| \lesssim 1.
\end{align*}
Undoing the parabolic rescaling and using $K \ge 2$, it follows that the corresponding Lagrangian mean curvature flow solution $L_K(t)$ as in \eqref{eqn approx unrescaled LMCF soln curve} has uniformly bounded mean curvature 
\begin{equation} \label{eqn H bound at parab scales}
    |\mathbf H_{L_K(t)} |(\mathbf x) \lesssim \frac1{\sqrt{T-t}} |\mathcal H[ f_a + u_s]| \lesssim ( T- t)^{ \frac{K-1}2 - \frac12} \lesssim 1 \qquad
    \text{when } |\mathbf x| \lesssim \sqrt{T-t}.
\end{equation}
Barrier arguments as in \cite{Stolarski23}*{Section 7} or \cite{Liu24}*{Proposition 8.1} can then be adapted to bound the mean curvature $\mathbf H_{L_K(t)}(\mathbf x)$ outside the region where $|\mathbf x| \lesssim \sqrt{T-t}$.
Combined with \eqref{eqn H bound at parab scales}, this suggests that, for any $n \ge 3$ and $K \ge 2$, the Lagrangian mean curvature flow solution $L_K^n(t)$ as in \eqref{eqn approx unrescaled LMCF soln curve} has uniformly bounded mean curvature
    $$\sup_{t \in [0, T)} \sup_{\mathbf x \in L_K(t)} |\mathbf H_{L_K(t)} |(\mathbf x) < \infty.$$

\appendix

\section{Assorted Formulas and Estimates}

\subsection{Estimates for $H_a$}

\begin{lem} \label{lem H_a - H_0 pointwise est}
    Let $H_a = H_{s,a}$ and $H_0 = H_{s,0}$ be the operators given by \eqref{eqn H_a defn} and \eqref{eqn H_0 defn} respectively.
    For any $n \ge 3$ and $s_0 > 0$, there exists $C = C(n, s_0) > 0$ such that 
    \begin{equation} \label{eqn H_a - H_0 pointwise est}
        |(H_a - H_0) u | \le C \left( a^n s^{-n} | \partial_{ss} u | + a^n s^{-n-1} |\partial_s u| \right) \qquad \forall s \in (s_0, \infty) , \, a \in (0,1).
    \end{equation}
\end{lem}
\begin{proof}
    A straightforward but tedious computation shows that
    \begin{multline} \label{proof H_a - H_0 pointwise est, eqn 1}
        (H_a -H_0) u  
        = (L_a - L_0) u 
        = \left[ \frac1{1 + (\partial_s f_a)^2} - \frac1{1 + \overline c_0^2} \right] \partial_{ss} u\\
        + \left[ 
        \frac{(n-1) (s^2 + f_a^2)^{-1} (s +  f_a\partial_s f_a)(1 + \partial_s f_a^2) - \partial_s f_a \partial_{ss} f_a }{  ( 1 + \partial_s f_a^2)^2 } 
        - \frac{n-1}{(1+ \overline c_0^2)s }\right]  \partial_s u .
    \end{multline}
    By the asymptotics of $f_a(s)$ \eqref{eq: asymptotic of SL profile function}, it follows that 
    \begin{align*}
        &(H_a - H_0)u\\
        ={}& \left[ \frac1{1 + (\partial_s f_a)^2} - \frac1{1 + \overline c_0^2} \right] \partial_{ss} u\\
        &+ \left[ 
        \frac{(n-1) (s^2 + f_a^2)^{-1} (s +  f_a\partial_s f_a)(1 + \partial_s f_a^2) - \partial_s f_a \partial_{ss} f_a }{  ( 1 + \partial_s f_a^2)^2 } 
        - \frac{n-1}{(1+ \overline c_0^2)s }\right]  \partial_s u \\
        ={}& \left[ \frac1{1 + [\overline c_0 + O(a^n s^{-n})]^2  } - \frac1{1 + \overline c_0^2} \right] \partial_{ss} u\\
        &+  
        (n-1) (s^2 + [\overline c_0 s + O ( a^n s^{1-n}) ]^2)^{-1}   \\
        &\qquad \qquad \cdot \frac{(s +  [\overline c_0 s + O ( a^n s^{1-n}) ][\overline c_0 + O(a^n s^{-n}) ])(1 + [\overline c_0 + O(a^n s^{-n})]^2)}{  ( 1 + [ \overline c_0 + O ( a^n s^{-n} )]^2 )^2 }
        \partial_s u 
        \\
        &- \frac{[\overline c_0 + O(a^n s^{-n} )] O(a^n s^{-n-1}) }{  ( 1 + [ \overline c_0 + O ( a^n s^{-n} )]^2 )^2 } \partial_s u
        - \frac{n-1}{(1+ \overline c_0^2)s }  \partial_s u \\
        ={}& O(a^n s^{-n} ) \partial_{ss} u
        + O(a^n s^{-n-1}) \partial_s u
    \end{align*}
    for $s \in (s_0, \infty)$ and $a \in (0, 1)$.
    This proves \eqref{eqn H_a - H_0 pointwise est}.
\end{proof}

\begin{lem} \label{lem partial_a H_a pointwise est}
    Let $H_a = H_{s,a}$ be the operator given by \eqref{eqn H_a defn}.
    For any $n \ge 3$ and $s_0 > 0$, there exists $C = C(n, s_0) > 0$ such that for all smooth $u : (s_0, \infty) \to \R $
    \begin{equation} \label{eqn partial_a H_a pointwise est}
        | \partial_a H_a  u | \le C \left( a^{n-1} s^{-n} | \partial_{ss} u | + a^{n-1} s^{-n-1} |\partial_s u| \right) \qquad \forall s \in (s_0, \infty) , \, a \in (0,1).
    \end{equation}
\end{lem}
\begin{proof}
    By a straightforward computation (cf. \eqref{proof H_a - H_0 pointwise est, eqn 1}),
    \begin{equation} \label{proof partial_a H_a pointwise est, eqn 1}
        H_a u 
        = \frac 1{1 + \partial_s f_a^2} \partial_{ss} u
        + (n-1) \frac{ s + f_a \partial_s f_a }{( s^2 + f_a^2 )( 1 + \partial_s f_a^2)} \partial_s u - \frac{\partial_s f_a \cdot \partial_{ss} f_a}{(1 + \partial_s f_a^2)^2} \partial_s u .
    \end{equation}
    Differentiating in $a$ gives
    \begin{gather} \label{proof partial_a H_a pointwise est, eqn 2}
    \begin{aligned}
        \partial_a H_a u 
        ={}& - 2 ( 1 + \partial_s f_a^2 )^{-2} (\partial_s f_a) (\partial_a \partial_s f_a) \partial_{ss} u \\
        &+ (n-1) ( \partial_a f_a \cdot \partial_s f_a + f \cdot \partial_a \partial_s f_a ) (s^2 + f_a^2)^{-1}  ( 1 + \partial_s f_a^2)^{-1} \partial_s u \\
        &-2(n-1)( s + f_a \partial_s f_a) (s^2 + f_a^2)^{-2} ( f_a \cdot \partial_a f_a) (1 + \partial_s f_a^2)^{-1} \partial_s u\\
        &- 2(n-1) ( s + f_a \partial_s f_a) (s^2 + f_a^2)^{-1} ( 1 + \partial_s f_a^2)^{-2} ( \partial_s f_a \cdot \partial_a \partial_s f_a ) \partial_s u \\
        &- \partial_a \partial_s f_a \cdot \partial_{ss} f_a (1 + \partial_s f_a^2)^{-2} \partial_s u \\
        &- \partial_s f_a \cdot \partial_a \partial_{ss} f_a ( 1 + \partial_s f_a^2)^{-2} \partial_s u\\
        &+ 4 \partial_s f_a \cdot \partial_{ss} f_a ( 1 + \partial_s f_a^2)^{-3} ( \partial_s f_a \cdot \partial_a \partial_s f_a ) \partial_s u.
    \end{aligned}
    \end{gather}
    Inserting the asymptotics of $f_a$ and its derivatives (see Lemma \ref{lem: profile function of Lawlor neck}) into \eqref{proof partial_a H_a pointwise est, eqn 2} and simplifying the resulting expression then proves \eqref{eqn partial_a H_a pointwise est}. 
\end{proof}

\subsection{Estimates for Integral Weights}
\begin{proposition} \label{prop exp -F_a bounds}
    For $n \ge 3$, there exists a dimensional constant $C_n > 0$ such that
    \begin{multline} \label{eqn Gaussian weight est}
        \max\left\{ e^{-C_n s^2}, e^{-\frac{(1+\ol{c}_0^2)s^2}4 - C_n a^2} \right\} \le e^{-F_a(s)} =  e^{- \frac{s^2}4 - \frac12 \int_0^s \tilde s ( \partial_s f_a)^2 d \tilde s} \\ \le \min \left\{ e^{-\frac{s^2}4 }, e^{-\frac{(1+\ol{c}_0^2)s^2}4 + C_n a^2} \right\} \qquad \forall s \in \R        
    \end{multline}
    where $\overline c_0 = \overline c_0(n) = \tan \left( \frac \pi 2 - \frac \pi {2n} \right)$.
\end{proposition}
\begin{proof}
    Throughout this argument, $C_n>0$ denotes a positive dimensional constant that may change from line to line.
    It suffices to prove the estimate for $s>0$ since $e^{- \frac{s^2}4 - \frac12 \int_0^s \tilde s(\partial_s f_a)^2 d \tilde s}$ is an even function of $s$.
    Note that the asympotics of $\partial_s f_1$ from Lemma \ref{lem: profile function of Lawlor neck} imply
    \begin{equation}\label{eqn f_1' est}
        |\partial_s f_1 - \ol{c}_0 | \le \frac{C_n}{1 + s^n }\qquad \forall s > 0.
    \end{equation}
    In particular, $| \partial_s f_a(s) | = |\partial_s f_1(s/a) | \le C_n$ for all $s$, from which we deduce $$e^{-C_n s^2} \le e^{- \frac{s^2}4 - \frac12 \int_0^s \tilde s (\partial_s f_a)^2 d \tilde s}.$$
    Moreover, it follows that for $s>0$,
    \begin{align*}
        0 \le \frac12 \int_0^s \tilde s ( \partial_s f_a)^2 d \tilde s 
        &= \frac12 \int_0^s \tilde s( \partial_s f_1(\tilde s/a))^2 d \tilde s \\
        &= \frac{a^2}2 \int_0^{s/a} \sigma ( \partial_s f_1(\sigma))^2 d \sigma && (\sigma = \tilde s /a ) \\
        &= \frac{a^2}2 \int_0^1 \sigma ( \partial_s f_1)^2 d \sigma + \frac{a^2}2 \int_1^{s/a} \sigma ( \partial_s f_1)^2 d \sigma \\
        &\le \frac{a^2}2 C_n + \frac{a^2}2 \int_1^{s/a} \sigma ( \ol{c}_0 + C_n \sigma^{-n} )^2 d \sigma 
        && (\text{by } \eqref{eqn f_1' est} )\\
        &\le C_n a^2 + \frac{a^2}2 \int_1^{s/a} \ol{c}_0^2 \sigma + C_n \sigma^{1-n} d \sigma \\
        &\le C_n a^2 + \frac{a^2}2 \int_0^{s/a} \ol{c}_0^2 \sigma d \sigma + \frac{a^2}2 \int_1^\infty C_n \sigma^{1-n} d \sigma \\
        &= C_n a^2 + \frac{\ol{c}_0^2 s^2}4 && ( n \ge 3).
    \end{align*}
    By similar logic, there exists the lower bound
        $$\frac12 \int_0^s \tilde s (\partial_s f_a )^2 d \tilde s \ge \frac{\ol{c}_0^2 s^2}4 - C_n a^2 \qquad \forall s > 0.$$
    The Gaussian estimate \eqref{eqn Gaussian weight est} now follows.
\end{proof}

\begin{proposition} \label{prop vol element bounds}
    For $n\ge 3$, there exists a dimensional constant $C_n > 1$ such that 
    \begin{equation} \label{eqn volume form est}
        C_n^{-1} \rho_a(s)^{n-1} \le \sqrt{ ( 1 + \partial_s f_a^2 ) ( s^2 + f_a^2 )^{n-1} } \le C_n \rho_a(s)^{n-1}
        \qquad \forall s \in \R; \, \forall 0 < a \le 1.
    \end{equation}

    Consequently, for any $\kappa > 0$, there exists $C = C(n, \kappa) > 0$ such that
    \begin{equation} \label{eqn L^2_a est for polynomials}
        \| |s|^\kappa \|_{L^2_a} \le C(n, \kappa) \qquad \forall 0  < a \le 1.
    \end{equation}
\end{proposition}
\begin{proof}
    Throughout the proof, $C_n > 0$ denotes a dimensional constant that may change from line to line.
    By Lemma \ref{lem: profile function of Lawlor neck}, it follows that 
        $$f_1(\sigma) \le C_n( 1 + |\sigma|) \text{ and }  f_1'(\sigma) \le C_n \qquad (\forall \sigma \in \R).$$
    Therefore,
    \begin{align*}
        &\sqrt{ ( 1 + \partial_s f_a^2 ) ( s^2 + f_a^2 )^{n-1} } \\
        &= a^{n-1} \sqrt{ ( 1 + \partial_s f_1(s/a)^2 ) ( (s/a)^2 + f_1(s/a)^2 )^{n-1} }
        && ( f_a(s) = a f_1(s/a) ) \\
        &\le C_n a^{n-1} ( 1 +|s/a|^{n-1} ) \\
        &\le C_n \rho_a(s)^{n-1}.
    \end{align*}

    By Lemma \ref{lem: profile function of Lawlor neck}, $f_a(\sigma) \ge f_a(0) = a$ for all $\sigma \in \R$.
    Thus,
    \begin{equation}
        \sqrt{ ( 1 + \partial_s f_a^2 ) ( s^2 + f_a^2)^{n-1} } 
        \ge \sqrt{ (s^2 + a^2)^{n-1} } \ge C_n^{-1} \rho_a(s)^{n-1}.
    \end{equation}

    Let $\kappa > 0$.
    Then
    \begin{align*}
        &\| |s|^\kappa \|_{L^2_a}^2 \\ 
        &= \int_\R |s|^{2\kappa} e^{- \frac{s^2}4 - \frac12 \int_0^s \tilde s ( \partial_s f_a)^2 d \tilde s} \sqrt{ ( 1 + \partial_s f_a^2)(s^2 + f_a^2 )^{n-1} } ds \\
        &\le C_n \int_\R |s|^{2 \kappa} e^{-s^2/4} ( a^{n-1} + s^{n-1}) ds 
        && (\text{by \eqref{eqn Gaussian weight est}, \eqref{eqn volume form est}}) \\
        &\le C(n, \kappa) && ( 0 < a \le 1).
    \end{align*}
\end{proof}

\begin{proposition}[Computation of $\partial_a d\mu_a$] \label{prop computation of partial_a d mu_a}
    For $n \ge 3$,
    \begin{multline} \label{eqn partial_a d mu_a}
        \partial_a \ln \left( e^{- \frac{s^2}4 - \frac12 \int_0^s \tilde s ( \partial_s f_a)^2 d \tilde s} \sqrt{ ( 1 + (\partial_s f_a)^2 )( s^2 + f_a^2 )^{n-1} } \right) \\
        = \frac1a \int_0^s \tilde s^2 (\partial_s f_a) (\partial_{ss}f_a) d \tilde s + \frac{n-1}a \frac{ (s \partial_s f_a - f_a )^2 }{s^2 + f_a^2} .
    \end{multline}

    Moreover, there exists a dimensional constant $C_n > 0$ such that for all $s \in \R$
    \begin{align}
        \label{prop computation of partial_a d mu_a est 1}
        \left| \frac1a \int_0^s \tilde s^2 (\partial_s f_a) (\partial_{ss}f_a) d \tilde s \right| &\le C_n a,\\
        \label{prop computation of partial_a d mu_a est 2}
        \text{and } \left| \frac{n-1}a \frac{ (s \partial_s f_a - f_a )^2 }{s^2 + f_a^2} \right| &\le \frac{C_n}a \frac1{1 + ( s/a)^{n+1}}.
    \end{align}
\end{proposition}
\begin{proof}
    Throughout the proof, $C_n > 0$ denotes a dimensional constant that may change from line to line.
    Using $f_a (s) = a f_1 (s/a)$, it follows that
    \begin{gather}
        \partial_s f_a(s) = \partial_s f_1 (s/a), \qquad \partial_{ss} f_a(s) = \frac1a \partial_{ss}f_1(s/a) \\
        \label{eqn partial_a f_a}
        \partial_a f_a = \frac1a f_a - \frac sa \partial_s f_a, \qquad
        \partial_a \partial_s f_a = - \frac sa \partial_{ss} f_a.
    \end{gather}

    It follows that 
    \begin{align*}
        &\partial_a \ln \sqrt{ ( 1 + (\partial_s f_a)^2 )( s^2 + f_a^2 )^{n-1} } \\
        ={}& \frac12 \frac{2 \partial_s f_a \partial_a \partial_s f_a }{1 + \partial_s f_a^2 } 
        + \frac{n-1}2 \frac{2 f_a \partial_a f_a}{s^2 + f_a^2 } \\
        ={}& - \frac{s \partial_s f_a }a \frac{ \partial_{ss} f_a }{1 + \partial_s f_a^2 } 
        - \frac{f_a}a (n-1) \frac{  s \partial_s f_a - f_a }{s^2 + f_a^2 } && (\text{by \eqref{eqn partial_a f_a}})\\
        ={}& \left[ \frac{s \partial_s f_a}a - \frac{f_a}a \right] (n-1) \frac{s \partial_s f_a - f_a}{s^2 + f_a^2} &&(\text{by Lemma \ref{lem: profile function of Lawlor neck}})\\
        ={}& \frac{n-1}a \frac{ (s \partial_s f_a - f_a)^2}{s^2 + f_a^2} ,
    \end{align*}
    and
    \begin{equation*}
        \partial_a \ln e^{- \frac {s^2}4 - \frac12 \int_0^s \tilde s (\partial_s f_a)^2 d \tilde s} 
        = - \int_0^s \tilde s \partial_s f_a \partial_a \partial_s f_a d \tilde s 
        = + \frac1a \int_0^s \tilde s^2 \partial_s f_a \partial_{ss} f_a d \tilde s.
    \end{equation*}
    This proves \eqref{eqn partial_a d mu_a}.

    Additionally,
    \begin{align*}
        \left| \frac1a \int_0^s \tilde s^2 \partial_s f_a \partial_{ss} f_a d \tilde s \right|
        &= \left| \int_0^s \frac{\tilde s^2}{a^2} [ \partial_s f_1(\tilde s/a) ] [\partial_{ss} f_1 (\tilde s/a) ] d \tilde s \right| \\
        &= a \left|  \int_0^{s/a} \sigma^2 \partial_s f_1 (\sigma) \partial_{ss}f_1(\sigma) d \sigma \right| && (\sigma = \tilde s/a) \\
        &\le a \left|  \int_0^{\infty} \sigma^2 \partial_s f_1 (\sigma) \partial_{ss}f_1(\sigma) d \sigma \right| \\
        &\le C_n a,
    \end{align*}
    where the last line follows uses the fact that, for $n \ge 3$, the asymptotics of $f_1$ and its derivatives from Lemma \ref{lem: profile function of Lawlor neck} imply 
        $$\left| s^2 \partial_s f_1 (\partial_{ss} f_1) \right| \le C_n ( 1 + |s|)^{1-n} \le C_n ( 1 + |s|)^{-2} \qquad \forall s \in \R$$
    which is integrable.

    For the other estimate,
    \begin{equation*}
        \left| \frac{n-1}a \frac{[s \partial_s f_a(s) - f_a (s) ]^2}{s^2 + f_a(s)^2}  \right| 
        \le \left| \frac{n-1}a \frac{[ \frac sa \partial_s f_1(s/a) - f_1 (s/a) ]^2}{(s/a)^2 + f_1(s/a)^2}  \right| \\
        \le \frac1a \frac{C_n}{1 + ( s/a)^{n+1}}
    \end{equation*}
    where the last line follows from the asymptotics of $f_1$ in Lemma \ref{lem: profile function of Lawlor neck}.
\end{proof}

\begin{proposition} \label{prop H_a symmetric}
    For $u, w \in C^\infty_c(\R)$,
    \begin{equation}
        \int_\R w H_a u \, d\mu_a = - \int_\R \frac{(\partial_s w) (\partial_s u)}{1 + (\partial_s f_a)^2} d \mu_a + \int_\R w u \, d \mu_a.
    \end{equation}
\end{proposition}
\begin{proof}
    This follows from a straightforward computation using integration by parts.
\end{proof}

\begin{proposition} \label{prop Ecker's Sobolev ineq}
    There exists $C = C(n, a) > 0$ such that
    \begin{equation} \label{eqn Ecker's Sobolev ineq}
        \| s u \|_{L^2_a(\R)} \le C \| u \|_{H^1_a(\R) } \qquad \forall u \in C^\infty_c(\R) .
    \end{equation}
\end{proposition}
\begin{proof}
    Let $u \in C^\infty_c(\R)$.
    For a constant $C_0 > 0$ to be determined, we compute that
    \begin{gather} \label{proof Ecker's Sobolev ineq, eqn 1} \begin{aligned}
        0 \le{}& \int_\R ( C_0 \partial_s u - s u)^2 d \mu_a \\
        ={}& \int C_0^2 (\partial_s u)^2  - 2 C_0 s u \partial_s u + s^2 u^2 d\mu_a \\
        ={}& \int C_0^2 (\partial_s u)^2 + s^2 u^2 d\mu_a - C_0 \int s \partial_s (u^2)  d\mu_a \\
        ={}& \int C_0^2 (\partial_s u)^2 + s^2 u^2 d \mu_a\\
        &+ C_0 \int u^2 \left[ 1 + s \partial_s \left( \ln \left( e^{-F_a(s) }\sqrt{ ( 1 + \partial_s f_a^2 ) ( s^2 + f_a^2)^{n-1}}  \right) \right) \right] d\mu_a
    \end{aligned} \end{gather}
    where the last line follows from integration by parts.
    By Lemma \ref{lem: profile function of Lawlor neck}, it follows that
    \begin{equation} \label{proof Ecker's Sobolev ineq, eqn 2}
        1 + s \partial_s  \left( \ln \left( e^{-F_a(s) }\sqrt{ ( 1 + \partial_s f_a^2 ) ( s^2 + f_a^2)^{n-1}}  \right) \right)
        = - \frac {s^2}2 - \frac{s^2}2 (\partial_s f_a )^2+ O(1; n,a). 
    \end{equation}
    Inserting \eqref{proof Ecker's Sobolev ineq, eqn 2} into \eqref{proof Ecker's Sobolev ineq, eqn 1} gives
    \begin{multline*}
        0 \le \int C_0^2 (\partial_s u)^2 + s^2 u^2 d\mu_a + C_0 \int u^2 \left( - \frac{s^2}2 - \frac{s^2}2 (\partial_s f_a)^2 + O(1; n,a) \right) d\mu_a \\
        \le \int C_0^2 ( \partial_s u)^2 + C_0 C(n,a) u^2 
        + \left( 1 - \frac{C_0}2 \right)  s^2 u^2 \, d\mu_a. 
    \end{multline*}
    Taking $C_0 = 3$ and rearranging then gives
    \begin{equation}
        \int s^2 u^2 d\mu_a \le C(n,a) \int (\partial_s u)^2 + u^2 \, d\mu_a 
    \end{equation}
    for some (different) constant $C(n,a)$ depending only on $n,a$ as claimed.
\end{proof}

\bibliography{BibFile}

\end{document}